\documentclass[a4paper,reqno]{amsart}

\usepackage{svn}
\SVN $Revision: 2156 $ 
\SVN $Date: 2012-08-03 15:42:33 +0200 (Fre, 03. Aug 2012) $
\SVN $Author: krenn@MATH.TU-GRAZ.AC.AT $

\newif\ifprivate
\privatefalse

\newcommand{\colbw}{bw}

\usepackage[utf8]{inputenc}
\usepackage{ae,aecompl}   
\usepackage[english]{babel}

\usepackage{a4wide}

\usepackage[wnafs,backref,cellrounding,setmiddle=colon]{mathdef}

\ifprivate
\else
\renewcommand{\TODO}[1]{}
\fi

\numberwithin{equation}{section}
\numberwithin{figure}{section}
\numberwithin{table}{section}
\numberwithin{algorithm}{section}

\makeatletter \renewcommand\p@enumii{} \makeatother

\newcommand{\itemref}[1]{\eqref{#1}}

\newif\ifma
\mafalse

\begin{document}


\title[Analysis of Width-$w$ Non-Adjacent Forms]{Analysis of Width-$w$
  Non-Adjacent Forms \\ to Imaginary Quadratic Bases}

\author{Clemens Heuberger}
\author{Daniel Krenn}

\thanks{The authors are supported by the Austrian Science Foundation FWF,
  project S9606, that is part of the Austrian National Research Network
  ``Analytic Combinatorics and Probabilistic Number Theory''.}


\address{\parbox{12cm}{%
    Clemens Heuberger \\
    Institute of Optimisation and Discrete Mathematics (Math B) \\
    Graz University of Technology \\
    Steyrergasse 30/II, A-8010 Graz, Austria \\}} 
 
\email{\href{mailto:clemens.heuberger@tugraz.at}{clemens.heuberger@tugraz.at}}

\address{\parbox{12cm}{%
    Daniel Krenn \\
    Institute of Optimisation and Discrete Mathematics (Math B) \\
    Graz University of Technology \\
    Steyrergasse 30/II, A-8010 Graz, Austria \\}} 

\email{\href{mailto:math@danielkrenn.at}{math@danielkrenn.at} \textit{or}
  \href{mailto:krenn@math.tugraz.at}{krenn@math.tugraz.at}}


\ifprivate 
\thispagestyle{headings}\pagestyle{headings} 
\markboth{\jobname{} rev. \SVNRevision{} --- \SVNDate{} \SVNTime}%
{\jobname{} rev. \SVNRevision{} --- \SVNDate{} \SVNTime}
\else
\fi


\begin{abstract}
  We consider digital expansions to the base of $\tau$, where $\tau$ is an
  algebraic integer. For a $w \geq 2$, the set of admissible digits consists of
  $0$ and one representative of every residue class modulo $\tau^w$ which is
  not divisible by $\tau$. The resulting redundancy is avoided by imposing the
  width $w$-NAF condition, i.e., in an expansion every block of $w$ consecutive
  digits contains at most one non-zero digit. Such constructs can be
  efficiently used in elliptic curve cryptography in conjunction with Koblitz
  curves.

  The present work deals with analysing the number of occurrences of a fixed
  non-zero digit. In the general setting, we study all $w$-NAFs of given length
  of the expansion. We give an explicit expression for the expectation and the
  variance of the occurrence of such a digit in all expansions. Further a
  central limit theorem is proved.

  In the case of an imaginary quadratic $\tau$ and the digit set of minimal
  norm representatives, the analysis is much more refined: We give an
  asymptotic formula for the number of occurrence of a digit in the $w$-NAFs of
  all elements of $\Z[\tau]$ in some region (e.g.\ a disc). The main term
  coincides with the full block length analysis, but a periodic fluctuation in
  the second order term is also exhibited. The proof follows Delange's method.

  We also show that in the case of imaginary quadratic $\tau$ and $w \geq 2$,
  the digit set of minimal norm representatives leads to $w$-NAFs for
  \emph{all} elements of $\Z[\tau]$. Additionally some properties of the
  fundamental domain are stated.
\end{abstract}


\maketitle


\makeatletter
\def\@tocline#1#2#3#4#5#6#7{\relax
  \ifnum #1>\c@tocdepth 
  \else
    \par \addpenalty\@secpenalty\addvspace{#2}%
    \begingroup \hyphenpenalty\@M
    \@ifempty{#4}{%
      \@tempdima\csname r@tocindent\number#1\endcsname\relax
    }{%
      \@tempdima#4\relax
    }%
    \parindent\z@ \leftskip#3\relax \advance\leftskip\@tempdima\relax
    \rightskip\@pnumwidth plus4em \parfillskip-\@pnumwidth
    #5\leavevmode\hskip-\@tempdima
      \ifcase #1
       \or\or \hskip 1em \or \hskip 2em \else \hskip 3em \fi%
      #6\nobreak\relax
    \dotfill\hbox to\@pnumwidth{\@tocpagenum{#7}}\par
    \nobreak
    \endgroup
  \fi}
\makeatother
\tableofcontents


\section{Introduction}
\label{sec:intro-backgr}


Let $\tau\in\C$ be an algebraic integer. We consider \tauadic{} expansions for
an element of $\Ztau$ using a redundant digit set $\cD$. This means that our
expansions need not be unique without any further constraints. However, by
applying a width\nbd-$w$ non-adjacency property to the digits of a
representation, together with choosing an appropriate digit set, we gain
uniqueness. The mentioned property simply means that each block of $w$ digits
contains at most one non-zero digit.

Such expansions have a low Hamming weight, i.e., a low number of non-zero
digits. This is of interest in elliptic curve cryptography: There, scalar
multiples of points can be computed by using \tauadic{}-expansions, where
$\tau$ corresponds to the Frobenius endomorphism. See
Section~\ref{sec:background} for a more detailled discussion.

The aim of this paper is to give a precise analysis of the expected number of
non-zeros in \tauadic{} expansions of elements in $\Ztau$, corresponding to the
expected number of costly curve operations. Several random models can be
considered.

The easiest model is to consider all expansions of given length to be equally
likely; this is called the ``full block length'' model. The result for
arbitrary algebraic integers is given in Theorem~\vref{th:w-naf-distribution}.
The appropriateness of this random model becomes debatable when looking at the
set of complex numbers admitting such an expansion of given length: This is the
intersection of the lattice of integers in the number field with a fractal set,
a scaled version of the ``Fundamental domain'', cf.\ Figure~\vref{fig:w-weta}.

A more natural choice seems to be to consider the expansions of all integers in
$\Ztau$ whose absolute value is bounded by some $N$. The main result of this
paper (Theorem~\vref{thm:countdigits}) is exactly such a result, where we
assume $\tau$ to be a imaginary quadratic number. Theorem~\ref{thm:countdigits}
is, in fact, more general: instead of considering all integers within a scaled
version of the unit circle, we consider all integers contained in a scaled copy
of some set $U$. Instead of counting the number of non-zeros, we count the
number of occurrences of each digit. The full block length analysis result will
indeed be needed to prove Theorem~\ref{thm:countdigits}.


For given $\tau$ and block length $w\ge 2$, several digit sets could be chosen.
A rather natural choice was proposed by
Solinas~\cite{Solinas:1997:improved-algorithm,Solinas:2000:effic-koblit}:
Consider the residue classes modulo $\tau^w$ in $\Ztau$. As digit set, we use
zero and a minimal norm representative from each residue class not divisible by
$\tau$. Now let $z\in\Ztau$ with $z=\sum_{j=0}^{\ell-1} z_j \tau^j$. This
expansion is a width\nbd-$w$ \tauadic{} non-adjacent form, or \wNAF{} for
short, if each block of $w$ consecutive digits $z_j \ldots z_{j+w-1}$ contains
at most one non-zero digit. The name ``non-adjacent form'' goes back to
Reitwiesner~\cite{Reitwiesner:1960}.


It is commonly known that such expansions, if they exist, are unique, whereas
the existence was only known for special cases, see
Section~\ref{sec:background}. In this paper in
Section~\ref{sec:exist-uniqu-nafs} we show that, for imaginary quadratic $\tau$
and $w\geq2$, every element of $\Ztau$ admits a unique \wNAF{}, see
Theorem~\vref{thm:wnaf-exist-unique}.  Additionally a simple algorithm for
calculating those expansions is given.


The full block length analysis is carried out in
Section~\ref{sec:full-block-length}: We define a random variable $X_{n,w,\eta}$
for the number of occurrences of $\eta$ in all \wNAF{}s of a fixed length
$n$. It is assumed that all those \wNAF{}s are equally likely. For an arbitrary
algebraic integer $\tau$ Theorem~\ref{th:w-naf-distribution} gives explicit
expressions for the expectation and the variance of
$X_{n,w,\eta}$. Asymptotically we get $\expect{X_{n,w,\eta}} \asymptotic e_w n$
and $\variance{X_{n,w,\eta}} \asymptotic v_w n$ for constants $e_w$ and $v_w$
depending on $w$ and the norm of $\tau$. The proof uses a regular expression
describing the \wNAF{}s. This will then be translated into a generating
function. Further in this theorem it is shown that $X_{n,w,\eta}$ satisfies a
central limit theorem.


The main result is the refined analysis described above: For imaginary
quadratic $\tau$, we count the number of occurrences $Z_{\tau,w,\eta}$ of the
non-zero digit $\eta$, when we look at all \wNAF{}s contained in $NU$ for a
given positive $N$ and a region $U\subseteq\C$ (e.g.\ the unit disc).  In
Theorem~\vref{thm:countdigits}, we prove that $Z_{\tau,w,\eta} \asymptotic e_w
N^2 \lmeas{U} \log_{\abs\tau}N$. This is not surprising, since intuitively
there are about $N^2 \lmeas{U}$ \wNAF{}s in the region $NU$, and each of them
can be represented as a \wNAF{} with length $ \log_{\abs\tau}N$. We even get a
more precise result. If the region is ``nice'', there is a periodic oscillation
of order $N^2$ in the formula.

The structure of the result --- main term, oscillation term, smaller error term
--- is not uncommon in the context of digits counting. For instance, a setting
similar to ours can be found in Heuberger and
Prodinger~\cite{Heuberger-Prodinger:2006:analy-alter}. There base $2$ and
special digit sets are used, and \wNAF[2]{}s are considered. The result has the
same structure as ours. Another example can be found in Grabner, Heuberger and
Prodinger~\cite{Grabner-Heuberger-Prodinger:2004:distr-results-pairs} for joint
expansions.

As in these examples, we follow the ideas of
Delange~\cite{Delange:1975:chiffres} to prove the statements. Before finally be
able to prove the main result in Section~\ref{sec:counting-digits-region}, we
have to collect various auxiliary results.

Our digit set of minimal norm representatives is characterised in terms of the
Voronoi cell of $0$ in the lattice $\Ztau$. The required estimates are shown in
Section~\ref{sec:voronoi}. The digit set itself as well as \wNAF{}s are then
defined and discussed in Section~\ref{sec:nafs}. Apart from expansions of
elements in $\Ztau$, we will also discuss infinite expansions of elements of
$\C$, as these will be needed in our geometric
arguments. Section~\ref{sec:full-block-length} is devoted to the full block
length analysis.  In Section~\ref{sec:bounds-value} we give bounds connecting
the absolute value and the length of a \wNAF{}. This allows us to prove the
existence (Theorem~\vref{thm:wnaf-exist-unique}) of \wNAF{}s in
Section~\ref{sec:exist-uniqu-nafs}. Further in
Theorem~\vref{thm:C-has-wnaf-exp} we get that every element of $\C$ has a
\wNAF{}-expansion of the form
$\xi_{\ell-1}\ldots\xi_1\xi_0\bfldot\xi_{-1}\xi_{-2}\ldots$, where the right
hand side of the \taupoint{} is allowed to be of infinite length.  In
Section~\ref{sec:fundamental-domain} we consider numbers of the form
$0\bfldot\xi_{-1}\xi_{-2}\ldots$. The set of all values of such numbers is
called the fundamental domain $\cF$. It is shown that $\cF$ is compact and its
boundary has Hausdorff dimension smaller than $2$. Further a tiling property
with scaled versions of $\cF$ is given for the complex plane.  In
Section~\ref{sec:cell-rounding-op}, we develop a suitable notion of
``fractional value''. Occurrence of the digit $\eta$ at arbitrary position can
be characterised in terms of the so-called ``characteristic sets'' which are
introduced in Section~\ref{sec:sets-w_eta}.


While the main focus of this paper lies on imaginary quadratic bases and the
digit set of minimal norm representatives, some of the results, e.g.\ the full
block length analysis (Theorem~\vref{th:w-naf-distribution}), are valid in a
more general setting. A more detailed overview on the requirements on $\tau$
and digit set~$\cD$ for the different sections, definitions, theorems, etc.\
can be found in Table~\ref{tab:overview-requirements}.


\begin{table}[t]
  \centering
  \begin{tabular}{l|l|c|c}
    \hline
    & short description & $\tau$ & digit set $\cD$ \\
    \hline
    Section~\ref{sec:voronoi} & Voronoi cells 
    & i-q & ------ \\
    Lemma~\vref{lem:complete-res-sys} & complete residue system 
    & alg & ------ \\
    Definition~\vref{def:min-norm-digit-set} & minimal norm representatives
    digit set 
    & i-q & ------ \\
    Definition~\vref{def:wnaf} & width-$w$ non-adjacent forms 
    & gen & fin \\
    Proposition~\vref{pro:value-continuous} & continuity of $\NAFvaluename$
    & gen & fin \\
    Definition~\vref{def:wnads} & width-$w$ non-adjacent digit set 
    & gen & fin \\ 
    Theorem~\vref{th:w-naf-distribution} & full block length distribution 
    theorem
    & alg & RRS \\
    Section~\ref{sec:bounds-value} & bounds for the value
    & i-q & MNR \\
    Theorem~\vref{thm:wnaf-exist-unique} & existence theorem for lattice points
    & i-q & MNR \\
    Theorem~\vref{thm:C-has-wnaf-exp} & existence theorem for $\C$
    & i-q & MNR \\
    Definition~\vref{def:fund-domain} & fundamental domain $\cF$
    & gen & fin \\
    Proposition~\vref{pro:fund-domain-compact} & compactness of the fundamental
    domain 
    & gen & fin \\
    Corollary~\vref{cor:complex-plane-tiling} & tiling property
    & i-q & MNR \\
    Remark~\vref{rem:ifs} & iterated function system
    & gen & fin \\
    Proposition~\vref{pro:char-boundary} & characterisation of the boundary
    & i-q & MNR \\
    Proposition~\vref{pro:boundary-fund-dom-dim-upper} & upper bound for the
    dimension of $\boundary*{\cF}$
    & i-q & MNR \\
    Section~\ref{sec:cell-rounding-op} & cell rounding operations
    & i-q & ------ \\
    Section~\ref{sec:sets-w_eta} & characteristic sets
    & i-q & MNR \\
    Theorem~\vref{thm:countdigits} & counting the occurences of a
    digit 
    & i-q & MNR \\
    \hline
  \end{tabular}

  \begin{minipage}{0.48\linewidth}
    \vspace*{1em}
    \begin{center}
      Abbreviations for $\tau$ \par
      (general: $\tau\in\C$ with $\abs\tau>1$) \par
      \begin{tabular}{cp{5cm}}
        \hline
        gen & $\tau\in\C$ \\
        alg & $\tau$ algebraic integer \\
        i-q & $\tau$ imaginary quadratic algebraic integer \\
        \hline
      \end{tabular}
    \end{center}
  \end{minipage}
  \begin{minipage}{0.48\linewidth}
    \vspace*{1em}
    \begin{center}
      Abbreviations for digit sets \par
      (general: $\cD\subseteq\Ztau$, $0\in\cD$) \par
      \begin{tabular}{cp{5cm}}
        \hline
        fin & finite digit set \\
        RRS & reduced residue system digit set \\
        MNR & minimal norm representatives digit set \\
        \hline
      \end{tabular}
    \end{center}
  \end{minipage}

  \caption{Overview of requirements.}
  \label{tab:overview-requirements}
\end{table}




\section{Background}
\label{sec:background}

In this section, we outline the connection to cryptographic applications which
motivated our study.

As a first example, we consider the elliptic curve
\begin{equation*}
  \cE_{3} \colon Y^2 = X^3 - X - \mu \qquad\text{with $\mu\in\set{-1,1}$}
\end{equation*}
defined over $\F_3$. This curve was studied by
Koblitz~\cite{Koblitz:1998:ellip-curve}. We are interested in the group
$\f{\cE_3}{\F_{3^m}}$ of rational points over a field extension $\F_{3^m}$ of
$\F_3$ for an $m\in\N$. The Frobenius endomorphism
\begin{equation*}
  \varphi \colon \f{\cE_3}{\F_{3^m}} \fto \f{\cE_3}{\F_{3^m}} , \quad
  \left(x,y\right) \fmapsto \left(x^3,y^3\right)
\end{equation*}
satisfies the relation $\varphi^2 - 3 \mu \varphi + 3 = 0$. So $\varphi$ may be
identified with the imaginary quadratic number $\tau = \frac{3}{2}\mu +
\frac{1}{2}\sqrt{-3}$, which is a solution of the mentioned relation. Thus we
have an isomorphism between $\Ztau$ and the endomorphism ring of
$\f{\cE_3}{\F_{3^m}}$.
 
Let $z\in\Ztau$ and $P\in\f{\cE_3}{\F_{3^m}}$. If we write the element $z$ as
$\sum_{j=0}^{\ell-1} z_j \tau^j$ for some digits $z_j$ belonging to a digit set
$\cD$, then we can compute the action $zP$ as $\sum_{j=0}^{\ell-1} z_j
\f{\varphi^j}{P}$ via a Horner scheme. The resulting Frobenius-and-add
method~\cite{Koblitz:1992:cm,Solinas:1997:improved-algorithm,Solinas:2000:effic-koblit}
is much faster than the classic double-and-add scalar multiplication.

So we are interested in a \tauadic{} expansion for an element of $\Ztau$ such
that the mentioned computation of the action is as efficient as possible. The
main computational effort are point additions, and we need one addition per
non-zero element of the expansions. 

But usually fewer non-zero coefficients means larger digit sets and thus a
higher pre-computation effort. So for optimal performance, a balance between
digit set size and number of non-zeros has to be found.

Another example is the elliptic curve
\begin{equation*}
    \cE_{2} \colon Y^2 + XY = X^3 + aX^2 + 1 \qquad\text{with $a\in\set{0,1}$}
\end{equation*}
defined over $\F_2$, cf.\ Koblitz~\cite{Koblitz:1992:cm}. There we get the
relation $\varphi^2-\mu\varphi+2=0$ with $\mu=(-1)^{1-a}$ for the Frobenius
endomorphism $\varphi$, and thus $\tau = \frac{1}{2}\mu +
\frac{1}{2}\sqrt{-7}$.


For the $\tau$ corresponding to $\cE_3$ and $w\geq2$, existence of \wNAF{}s was
shown in Koblitz~\cite{Koblitz:1998:ellip-curve} and Blake, Kumar Murty and
Xu~\cite{Blake-Kumar-Xu:2005:effic-algor}, for the $\tau$ corresponding to
$\cE_2$ and $w\geq2$ in Solinas~\cite{Solinas:2000:effic-koblit} and Blake,
Kumar Murty and Xu~\cite{Blake-Murty-Xu:2005:naf}. Some other $\tau$ are
handled in Blake, Kumar Murty and Xu~\cite{Blake-Murty-Xu:ta:nonad-radix}.


\section{Voronoi Cells}
\label{sec:voronoi}


Let $\tau\in\C$ be an algebraic integer, imaginary quadratic, i.e., $\tau$ is
solution of an equation $\tau^2 - p \tau + q = 0$ with $p,q\in\Z$, such that
$4q-p^2>0$. 

We will use the digit set of minimal norm representatives. In order to describe
this digit set, we will rewrite the minimality condition in terms of the
Voronoi cell for the lattice $\Ztau$, cf.\ Gordon~\cite{Gordon:1998}. 


\begin{definition}[Voronoi Cell]
  \label{def:voronoi}

  We set
  \begin{equation*}
    V := \set*{z\in\C}{ \forall y\in\Ztau : \abs{z} \leq \abs{z-y}}.
  \end{equation*}
  $V$ is the \emph{Voronoi cell for $0$} corresponding to the set $\Ztau$. Let
  $u \in \Ztau$. We define the \emph{Voronoi cell for $u$} as
  \begin{equation*}
    V_u := u + V = \set*{u+z}{z \in V}
    = \set*{z\in\C}{ \forall y\in\Ztau : \abs{z-u} \leq \abs{z-y}}.  
  \end{equation*}
  The point $u$ is called \emph{centre of the Voronoi cell} or \emph{lattice
    point corresponding to the Voronoi cell}.
\end{definition}


An example of a Voronoi cell in a lattice $\Ztau$ is shown in
Figure~\vref{fig:voronoi}. Whenever the word ``cells'' is used in this paper,
these Voronoi cells or scaled Voronoi cells will be meant.


\begin{figure}
  \centering
  \includegraphics{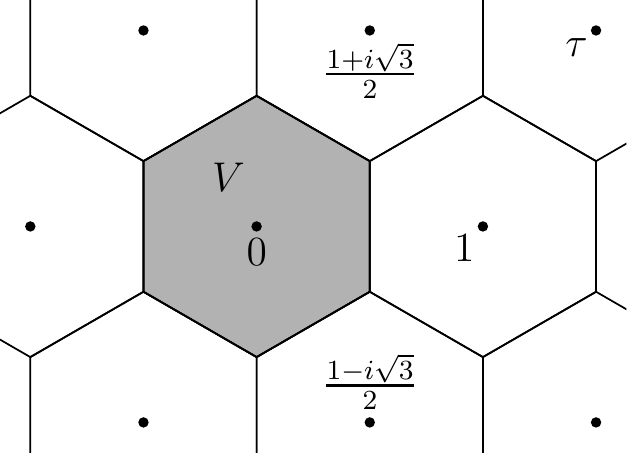}
  \caption[Voronoi cell $V$ for $0$]{Voronoi cell $V$ for $0$ corresponding to
    the set $\Ztau$ with $\tau = \frac{3}{2}+\frac{1}{2}\sqrt{-3}$.}
  \label{fig:voronoi}
\end{figure}


Two neighbouring Voronoi cells have at most a subset of their boundary in
common. This can be a problem, when we tile the plane with Voronoi
cells and want that each point is in exactly one cell. To fix this problem we
define a restricted version of $V$. This is very similar to the construction
used in Avanzi, Heuberger and Prodinger~\cite{Avanzi-Heuberger-Prodinger:2010:arith-of}.


\begin{definition}[Restricted Voronoi Cell]
  \label{def:restr-voronoi}

  Let $V_u$ be a Voronoi cell as above and $u$ its centre. Let
  $v_0,\dots,v_{m-1}$ with appropriate $m\in\N$ be the vertices of $V_u$
  labelled counter-clockwise. We denote the midpoint of the line segment from
  $v_k$ to $v_{k+1}$ by $v_{k+1/2}$, and we use the convention that the indices
  are meant modulo $m$.

  The \emph{restricted Voronoi cell} $\wt{V}_u$ consists of
  \begin{itemize}
  \item the interior of $V_u$,
  \item the line segments from $v_{k+1/2}$ (excluded) to $v_{k+1}$ (excluded)
    for all $k$,
  \item the points $v_{k+1/2}$ for $k\in\set{0,\dots,\floor{\frac{m}{2}}-1}$,
    and
  \item the points $v_k$ for $k\in\set{1,\dots,\floor{\frac{m}{3}}}$.
  \end{itemize}
  Again we set $\wt{V} := \wt{V}_0$.
\end{definition}


In Figure~\vref{fig:voronoi-restr} the restricted Voronoi cell for $0$ is
shown. The second condition is used, because it benefits symmetries. The third
condition is just to make the midpoints unique. Obviously, other
rules\footnote{The rule has to make sure that the complex plane can be covered
  entirely and with no overlaps by restricted Voronoi cells, i.e., the
  condition $\C = \biguplus_{z\in\Ztau} V_z$ has to be fulfilled.} could have
been used to define the restricted Voronoi cell.

As a generalisation of the usual fractional part of elements in $\R$ with
respect to the integers, we define the fractional part of an element of $\C$
corresponding to the restricted Voronoi cell $\wt{V}$ and thus corresponding to
the lattice $\Ztau$.


\begin{figure}
  \centering
  \includegraphics{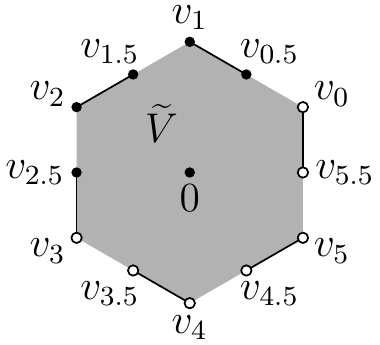}
  \caption[Restricted Voronoi cell $\wt{V}$ for $0$]{Restricted Voronoi cell
    $\wt{V}$ for $0$ corresponding to the set $\Ztau$ with $\tau =
    \frac{3}{2}+\frac{1}{2}\sqrt{-3}$.}
  \label{fig:voronoi-restr}
\end{figure}


\begin{definition}[Fractional Part in $\Ztau$]
  \label{def:frac-voronoi}

  Let $z\in\C$, $z = u + v$ with $u \in \Ztau$ and $v \in \wt{V}$. Then we
  define the \emph{fractional part corresponding to the lattice $\Ztau$} by
  $\fracpartZtau{z} := v$.

\end{definition}


This definition is valid, because of the construction of the restricted Voronoi
cell. The fractional part of a point $z\in\C$ simply means, to search for the
nearest lattice point $u$ of $\Ztau$ and returning the difference $z-u$.


Throughout this paper we will use the following notation for discs in the
complex plane.


\begin{definition}[Opened and Closed Discs]
  Let $z\in\C$ and $r\geq0$. The \emph{open disc $\ball{z}{r}$ with centre $z$
    and radius $r$} is denoted by
  \begin{equation*}
    \ball{z}{r} := \set*{y\in\C}{\abs{z-y} < r}
  \end{equation*}
  and the \emph{closed disc $\ball*{z}{r}$ with centre $z$ and radius $r$} by
  \begin{equation*}
    \ball*{z}{r} := \set*{y\in\C}{\abs{z-y} \leq r}.
  \end{equation*}
  The disc $\ball{0}{1}$ is called \emph{unit disc}.
\end{definition}


We will need suitable bounds for the digits in our digit set. These require
precise knowledge on the Voronoi cells, such as the position of the vertices
and bounds for the size of $V$. Such information is derived in the following
proposition.


\begin{proposition}[Properties of Voronoi Cells]
  \label{pro:voronoi-prop}
  
  We get the following properties:

  \begin{enumerate}[(a)]

  \item The vertices of $V$ are given by
    \begin{align*}
      v_0 &= 1/2 + \frac{i}{2\im{\tau}} 
      \left( \im{\tau}^2 + \fracpart{\re{\tau}}^2 
        - \fracpart{\re{\tau}} \right), \\
      v_1 &= \fracpart{\re{\tau}}-\frac12 + \frac{i}{2\im{\tau}} 
      \left( \im{\tau}^2 - \fracpart{\re{\tau}}^2 
        + \fracpart{\re{\tau}} \right), \\
      v_2 &= -1/2 + \frac{i}{2\im{\tau}}   
      \left( \im{\tau}^2 + \fracpart{\re{\tau}}^2 
        - \fracpart{\re{\tau}} \right)
      = v_0 - 1, \\
      v_3 &= -v_0, \\
      v_4 &= -v_1 \\
      \intertext{and}
      v_5 &= -v_2.
    \end{align*}
    All vertices have the same absolute value. If $\re{\tau}\in\Z$, then
    $v_1=v_2$ and $v_4=v_5$, i.e., the hexagon degenerates to a rectangle.

  \item The Voronoi-cell $V$ is convex.

  \item \label{enu:voronoi-bounds} 
    We get the bounds
    \begin{equation*}
      \ball*{0}{\textstyle\frac12} 
      \subseteq V \subseteq \ball*{0}{\abs{\tau} c_V}
    \end{equation*}
    with $c_V=\sqrt{\frac{7}{12}}$.

  \item The Lebesgue measure of $V$ in the complex plane is
    \begin{equation*}
      \lmeas{V} = \abs{\im{\tau}}.
    \end{equation*}
    
  \item The inclusion $\tau^{-1} V \subseteq V$ holds.

  \end{enumerate}
\end{proposition}


Before we start with the proof of this proposition, we add some remarks on the
constant $c_V$. Solinas~\cite{Solinas:2000:effic-koblit} uses the Voronoi cell
for the special $\tau=\frac12+\frac12\sqrt{-7}$ with $c_V=\sqrt{\frac27}$. The
upper bound $c_V=\sqrt{\frac{7}{12}}$ in \itemref{enu:voronoi-bounds} of the
proposition is not sharp. Indeed, with some effort, one could prove
$c_V=\sqrt{\frac{3}{8}}$. A smaller $c_V$ leads to better bounds in
Section~\ref{sec:bounds-value}. Further, the set of ``problematic values'' ---
those arise in the proof of the upper bound in
Proposition~\vref{pro:upper-bound-fracnafs} and the lower bound in
Proposition~\vref{pro:lower-bound-fracnafs}, as well as in the existence result
in Theorem~\vref{thm:wnaf-exist-unique} --- is decreased. This means fewer
configurations have to be checked separately. As some of the computational
verifications would still be necessary even with $c_V=\sqrt{\frac{3}{8}}$, the
improvement does not seem to outweigh the effort.


Now back to the proof of Proposition~\vref{pro:voronoi-prop}. We will use some
properties of Voronoi cells there, which can, for example, be found in
Aurenhammer~\cite{Aurenhammer:1991:voron-diagr}.


\begin{proof}
  \begin{enumerate}[(a)]

  \item Since $V$ is point-symmetric with respect to $0$, we get $v_0=-v_3$,
    $v_1=-v_4$ and $v_2=-v_5$. Thus we suppose without loss of generality
    $\im{\tau} > 0$. Strict greater holds, because $\tau$ is imaginary
    quadratic. Even more, we get $\im{\tau} \geq \frac{\sqrt{3}}{2}$, since
    \begin{equation*}
      \tau=\frac{p}{2} \pm \frac{i}{2} \sqrt{4q-p^2}
    \end{equation*}
    is solution of $\tau^2 - p \tau + q = 0$ for $p,q\in\Z$ and either $4q-p^2
    \equiv 0 \pmod{4}$ or $4q-p^2 \equiv -1 \pmod{4}$.

    All elements of the lattice $\Ztau$ can be written as $a+b\tau$, since
    $\tau$ is quadratic. We have to consider the neighbours of $0$ in the
    lattice. The Voronoi cell is the area enclosed by the line segment bisectors
    of the lines from each neighbour to zero, see
    Figure~\ref{fig:voronoi-points}. 

    \begin{figure}
      \centering
      \includegraphics{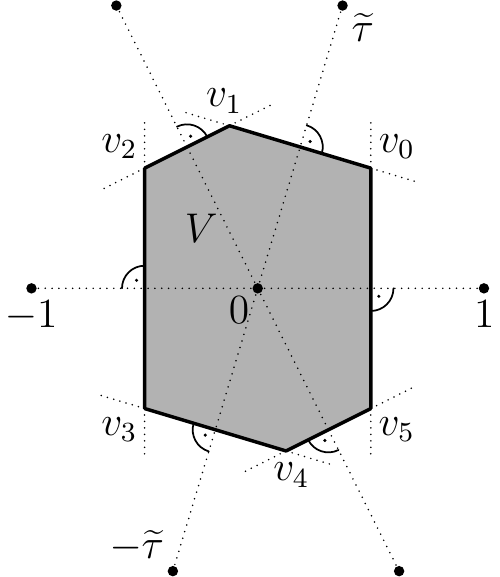}
      \caption[Construction of the Voronoi cell $V$ for $0$]{Construction of the
        Voronoi cell $V$ for $0$. The picture shows a general situation. Since
        $\tau$ is an imaginary quadratic algebraic integer, we will have
        $\re{\wt\tau}\in\set{0,\frac12}$.}
      \label{fig:voronoi-points}
    \end{figure}

    Clearly $\re{v_0}=\frac12$ and $\re{v_2}=-\frac12$, since $-1$ and $1$ are
    neighbours. Set $\wt\tau = \fracpart{\re{\tau}} + i \im{\tau}$. Consider the
    line from $0$ to $\wt\tau$ with midpoint $\frac12 \wt\tau$. We get
    \begin{equation*} 
      v_0 = \frac12 \wt\tau - x_A i \wt\tau
    \end{equation*}
    and
    \begin{equation*} 
      v_1 = \frac12 \wt\tau + x_B i \wt\tau
    \end{equation*}
    for some $x_A\in\R_{\geq0}$ and $x_B\in\R_{\geq0}$. Analogously, for the
    line from $0$ to $\wt\tau-1$, we have
    \begin{equation*} 
      v_1 = \frac12 \left(\wt\tau-1\right) - x_C i \left(\wt\tau-1\right)
    \end{equation*}
    and
    \begin{equation*} 
      v_2 = \frac12 \left(\wt\tau-1\right) + x_D i \left(\wt\tau-1\right).
    \end{equation*}
    for some $x_C\in\R_{\geq0}$ and $x_D\in\R_{\geq0}$. Solving this system of
    linear equations leads to the desired result.  An easy calculation shows
    that $\abs{v_0}=\abs{v_1}=\abs{v_2}$.  

    Until now, we have constructed the Voronoi cell of the points
    \begin{equation*}
      P := \set{0,1,-1,\wt\tau,\wt\tau-1,-\wt\tau,-(\wt\tau-1)}.
    \end{equation*}
    We want to rule out all other points, i.e., make sure, that none of the
    other points changes the already constructed cell. So let $z=x+iy\in\Ztau$
    and consider $\frac{z}{2}$. Because of symmetry reasons, we can assume
    $x\geq0$ and $y\geq0$. Clearly all points $z\in\Z$ with $z\geq2$ do not
    change the Voronoi cell, since $\frac{z}{2}>\frac12$ and the corresponding
    line segment bisector is vertical. So we can assume $y>0$.
 
    Now we will proceed in the following way. A point $z$ can be ruled out, if
    the absolute value of $\frac{z}{2}$ is larger than
    \begin{equation*}
      R = \abs{v_0} = \abs{v_1} = \abs{v_2}.
    \end{equation*}
    Let $\tau=a+ib$. If $\fracpart{a}=0$, then $R^2=\frac14
    \left(1+b^2\right)$. We claim that
    \begin{equation*}
      R^2 < \frac{x^2+y^2}{4} 
      \equivalent 1+b^2 < x^2+y^2.
    \end{equation*}
    Since $y>0$, we have $y \geq b$. If $y=b$, then points with $x>1$ need not
    be taken into account. But the remaining points are already in $P$ (at least
    using symmetry and $\wt\tau+1$ instead of $\wt\tau-1$). If $y\geq2b$, then
    all points except the ones with $x=0$ can be ruled out, since $1-b^2 \leq 1
    - \frac34 = \frac14 < x$. But the points $z$ with $x=0$ can be ruled out,
    too, because there is already the point $ib$ in $P$.

    So let $\fracpart{a}=\frac12$. Then $R^2=\frac14
    \left(\frac12+b^2+\frac{1}{16b^2}\right)$ and we claim that
    \begin{equation*}
      R^2 < \frac{x^2+y^2}{4} 
      \equivalent \frac12+b^2+\frac{1}{16b^2} < x^2+y^2.
    \end{equation*}
    If $y=b$, then $x>\sqrt{\frac{7}{12}}$ suffices to rule out a point $z$,
    since $b\geq\frac{\sqrt{3}}{2}$. But the only point $z$ with
    $x\leq\sqrt{\frac{7}{12}}$ is $\frac12+ib$, which is already in $P$. If
    $y\geq2b$, then $\frac12-b^2+\frac{1}{16b^2} \leq
    \frac12-\frac34+\frac1{12}<0$, so all points can be ruled out.
    
  \item Follows directly from the fact that all vertices have the same absolute
    value.

  \item From
    \begin{equation*}
      v_0 = 1/2 + \frac{i}{2\im{\wt\tau}} 
      \underbrace{\left( \im{\wt\tau}^2 + \re{\wt\tau}^2 - \re{\wt\tau}
        \right)}_{\leq \im{\wt\tau}^2}
    \end{equation*}
    we obtain
    \begin{equation*}
      \frac{\abs{v_0}}{\abs{\wt\tau}} 
      \leq \frac{\abs{1+ i \im{\wt\tau}}}{2\abs{\wt\tau}} 
      \leq \frac{\im{\wt\tau}}{2\abs{\wt\tau}} 
      \sqrt{\frac{1}{\im{\wt\tau}^2} + 1}
      \leq \sqrt{\frac{7}{12}} =: c_V
    \end{equation*}
    since $\frac{\sqrt3}{2} \leq \im{\wt\tau} \leq \abs{\wt\tau}$. Therefore
    $V\subseteq\ball*{0}{\abs{\tau} c_V}$.

    Since $0\leq\re{\wt\tau}\leq1$, we see that $\im{v_1} \geq \im{v_0} =
    \im{v_2}$. By construction, the line from $0$ to $\wt\tau$ intersects the
    line from $v_0$ to $v_1$ at $\frac12 \wt\tau$, so $\frac12 \abs{\wt\tau}$ is
    an upper bound for the largest circle inside $V$. Analogously we get
    $\frac12 \abs{\wt\tau-1}$ as a bound, and from the line from $0$ to $1$ we
    get $\frac12$. Since $\wt\tau$ and $\wt\tau-1$ are lattice points and not
    zero, their norms are at least $1$, so $\ball*{0}{\frac12}$ is inside $V$.

  \item The area of $V$ can be calculated easily, because
    $\im{v_0}=\im{v_2}$. Thus, splitting up the region in a rectangle and a
    triangle and using symmetry, the result follows.

  \item Let $x\in\tau^{-1} V$. Thus $x=\tau^{-1}z$ for an appropriate $z \in
    V$. For every $y\in\Ztau$ we obtain
    \begin{equation*}
      \abs{x} = \abs{\tau^{-1}} \abs{z} 
      \leq \abs{\tau^{-1}} \abs{z-y} = \abs{x - \tau^{-1} y}.
    \end{equation*}
    For an arbitrary $u \in \Ztau$ we can choose $y = \tau u$, and therefore
    $\abs{x} \leq \abs{x-u}$, i.e., $x \in V$. \qedhere

  \end{enumerate}
\end{proof}


We make an extensive use of Voronoi cells throughout this article, especially
in Section~\ref{sec:cell-rounding-op}. There we define cell rounding operations
which are working on subsets of the complex plane.


\section{Digit Sets and Non-Adjacent Forms}
\label{sec:nafs}


In this section $\tau\in\C$ will be an algebraic integer with $\abs{\tau}>1$,
and let $w\in\N$ with $w\geq2$. Further let $\normsymbol\colon\Ztau\fto\Z$
denote the norm function. We want to build a numeral system for the elements of
$\Ztau$ with base $\tau$. Thus we need a digit set $\cD$, which will be a finite
subset of $\Ztau$ containing $0$.


\begin{definition}[Reduced Residue Digit Set]
  \label{def:red-residue-digit-set}

  Let $\cD\subseteq\Ztau$. The set $\cD$ is called a \emph{reduced residue digit
    set modulo $\tau^w$}, if it consists of $0$ and exactly one
  representative for each residue class of $\Ztau$ modulo $\tau^w$ that is not
  divisible by $\tau$.
\end{definition}

From now on suppose $\cD$ is a reduced residue digit set modulo $\tau^w$. The
following two auxiliary results are well-known\footnote{Although those results
  are well-known, we were not able to find a reference. Any hints are
  welcome.}; we include a proof for the sake of completeness.


\begin{lemma}
  \label{le:divisibility-lemma}

  Let $c$ be a rational integer. Then $\tau$ divides $c$ in $\Z[\tau]$ if and
  only if $\normtau$ divides $c$ in $\Z$.
\end{lemma}
\begin{proof}
  From the minimal polynomial, it is clear that $\tau$ divides $\normtau$ in
  $\Z[\tau]$, so $\normtau \mid c$ implies $\tau\mid c$.

  For the converse direction, assume that $\tau\cdot\left(\sum_{j=0}^{d-1}x_j
    \tau^j\right)=c$ for some rational integers $x_j$. Here $d$ is the degree
  of $\tau$. Write the minimal polynomial of $\tau$ as
  \begin{equation*}
    \tau^d+\sum_{j=0}^{d-1}a_j\tau^j=0.
  \end{equation*}
  Thus we obtain
  \begin{equation*}
    c=-x_{d-1} a_0+\sum_{j=1}^{d-1}(x_{j-1}-x_{d-1}a_j)\tau^j.
  \end{equation*}
  Comparing coefficients in $\tau^j$ yields $c=-x_{d-1}a_0$, which implies that
  $a_0=(-1)^d \normtau$ divides $c$ in $\Z$, as required.
\end{proof}


Next, we determine the cardinality of $\cD$ by giving an explicit system of
representatives of the residue classes.


\begin{lemma}
  \label{lem:complete-res-sys}

  A complete residue system modulo $\tau^w$ is given by
  \begin{equation}\label{eq:residue-system}
    \sum_{j=0}^{w-1}a_j\tau^j\text{ with }a_j\in\{0,\ldots, \abs{\normtau}-1\}\text{
      for }0\le j<w.
  \end{equation}
  In particular, there are $\abs{\normtau}^w$ residue classes modulo $\tau^w$ in
  $\Z[\tau]$.
  
  A representative $\sum_{j=0}^{w-1}a_j\tau^j$ with $a_j\in\{0,\ldots,
  \abs{\normtau}-1\}$ is divisible by $\tau$ if and only if $a_0=0$. In particular,
  the cardinality of $\cD$ equals $\abs{\normtau}^{w-1}(\abs{\normtau}-1)+1$.
\end{lemma}


\begin{proof}
  Every element $z$ of $\Z[\tau]$ can be written as 
  \begin{equation*}
    z=x\tau^w+\sum_{j=0}^{w-1} a_j\tau^j
  \end{equation*}
  for some $a_j\in\{0,\ldots, \abs{\normtau}-1\}$ and an appropriate
  $x\in\Z[\tau]$: Take the expansion of $z$ with respect to the $\Z$-basis
  $\tau^j$, $0\le j<d$ and subtract appropriate multiples of the minimal
  polynomial of $\tau$ in order to enforce $0\le a_j<\abs{\normtau}$ for $0\le
  j\le w-1$.  This shows that \eqref{eq:residue-system} indeed covers all
  residue classes modulo $\tau^w$.

  Assume that $\sum_{j=0}^{w-1}a_j\tau^j\equiv \sum_{j=0}^{w-1}b_j\tau^j
  \pmod{\tau^w}$ for some $a_j$, $b_j\in\{0,\ldots,\abs{\normtau}-1\}$, but
  $a_j\neq b_j$ for some $j$. We choose $0\le j_0\le w-1$ minimal such that
  $a_{j_0}\neq b_{j_0}$. We obtain
  \begin{equation*}
    \sum_{j=j_0}^{w-1}a_j \tau^{j-j_0}\equiv 
    \sum_{j=j_0}^{w-1} b_j\tau^{j-j_0} \pmod{\tau^{w-j_0}},
  \end{equation*}
  which implies that $a_{j_0}\equiv b_{j_0} \pmod{\tau}$.  By
  Lemma~\vref{le:divisibility-lemma}, this implies that $a_{j_0}=b_{j_0}$,
  contradiction. Thus \eqref{eq:residue-system} is indeed a complete system of
  residues modulo $\tau^w$.

  From Lemma~\vref{le:divisibility-lemma} we also see that exactly the
  $\abs{\normtau}^{w-1}$ residue classes $\sum_{j=1}^{w-1} a_j\tau^j$ are
  divisible by $\tau$. We conclude that
  $\card*{\cD}=\abs{\normtau}^{w-1}(\abs{\normtau}-1)+1$.
\end{proof}


Since our digit set $\cD$ is constructed of residue classes, we want a
uniqueness in choosing the representative. We have the following definition,
where the restricted Voronoi $\wt{V}$ for the point $0$ from
Definition~\vref{def:restr-voronoi} is used. 


\begin{definition}[Representatives of Minimal Norm]
  \label{def:min-norm}

  Let $\tau$ be an algebraic integer, imaginary quadratic, and let
  $\eta\in\Ztau$ be not divisible by $\tau$. Then $\eta$ is called a
  \emph{representative of minimal norm of its residue class}, if $\eta \in
  \tau^w \wt{V}$.
\end{definition}

With this definition we can define the following digit set, cf.\
Solinas~\cite{Solinas:1997:improved-algorithm,Solinas:2000:effic-koblit} or
Blake, Kumar Murty and Xu~\cite{Blake-Kumar-Xu:2005:effic-algor}.

\begin{definition}[Minimal Norm Representatives Digit Set]
  \label{def:min-norm-digit-set}

  Let $\tau$ be an algebraic integer, imaginary quadratic, and let $\cD$ be a
  reduced residue digit set modulo $\tau^w$ consisting of representatives of
  minimum norm of its residue classes. Then we will call such a digit set
  \emph{minimal norm representatives digit set modulo $\tau^w$}.
\end{definition}

From now on we will suppose that our digit set $\cD$ is a minimal norm
representatives digit set modulo $\tau^w$. Some examples are shown in
Figure~\vref{fig:digit-sets}. There are also other definitions of a minimal
norm representative digit set, as discussed in the following remark.


\begin{figure}
  \centering \subfloat[Digit set for $\tau=\frac{1}{2} + \frac{1}{2} \sqrt{-7}$
  and $w=2$.]{
    \includegraphics[width=0.20\linewidth]{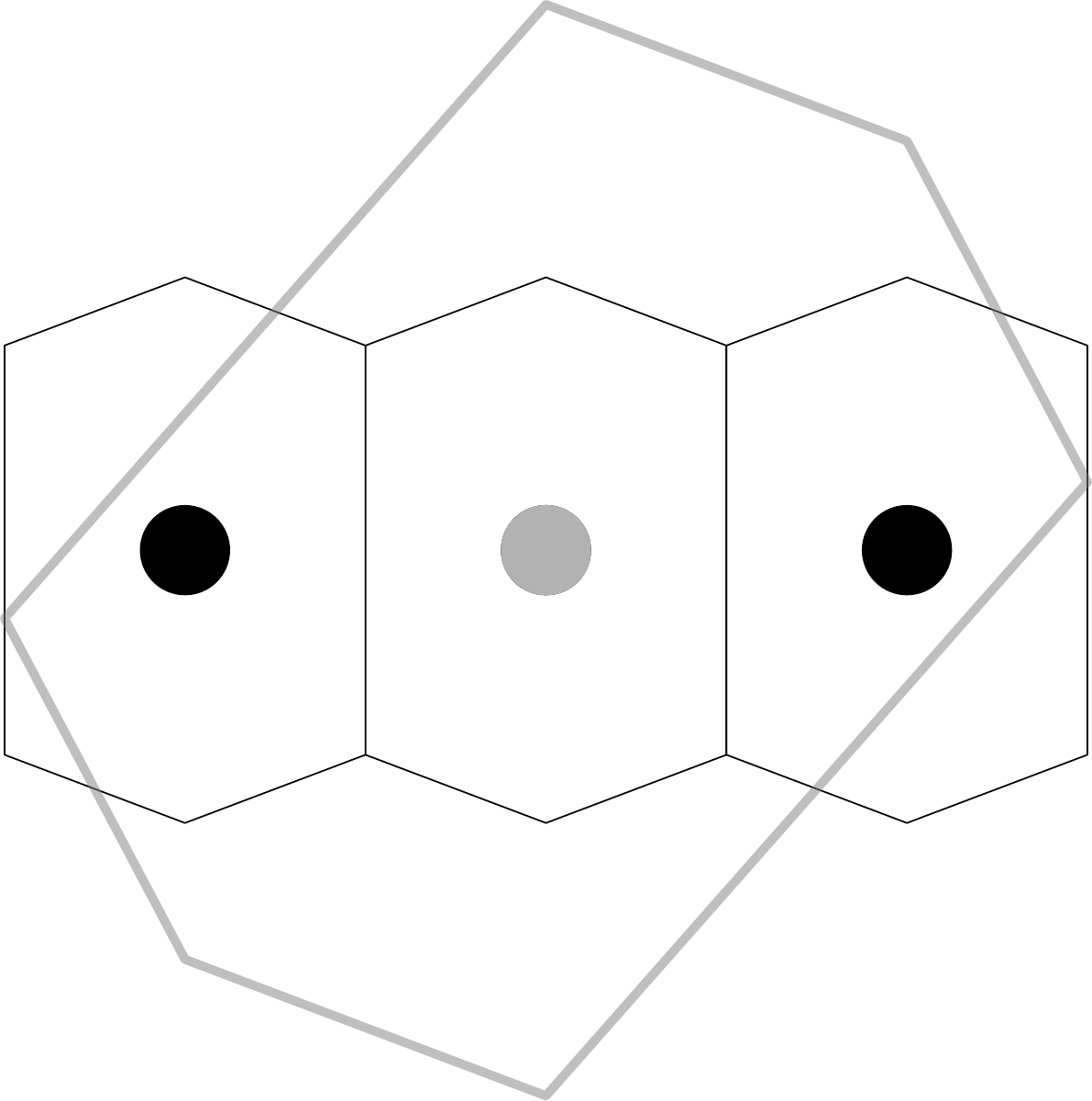}
    \label{fig:ds:1}}
  \quad 
  \subfloat[Digit set for $\tau=1 + \sqrt{-1}$ and $w=4$.]{
    \includegraphics[width=0.20\linewidth]{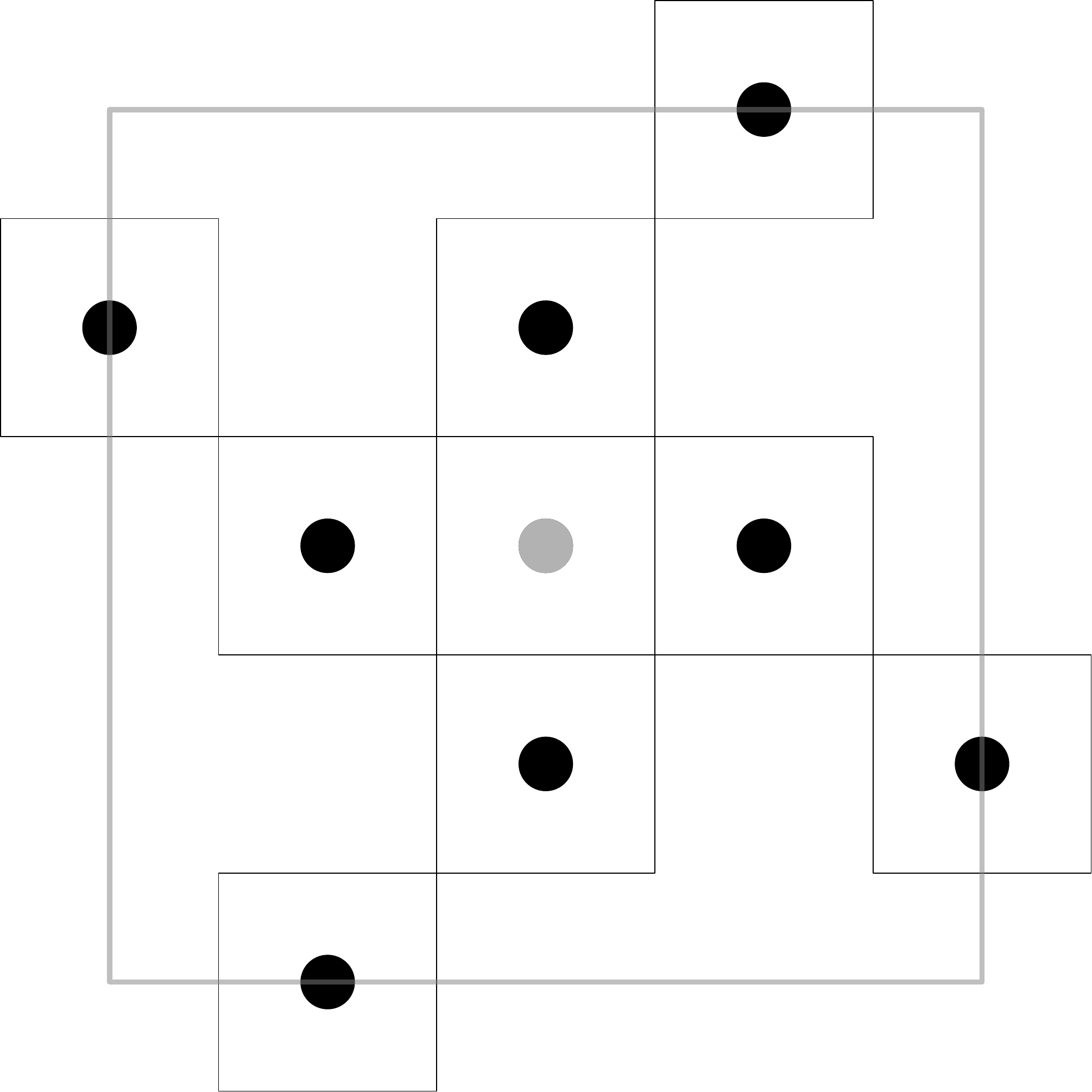}
    \label{fig:ds:2}}
  \quad
  \subfloat[Digit set for $\tau=\frac{3}{2} + \frac{1}{2}
  \sqrt{-3}$ and $w=2$.]{
    \includegraphics[width=0.20\linewidth]{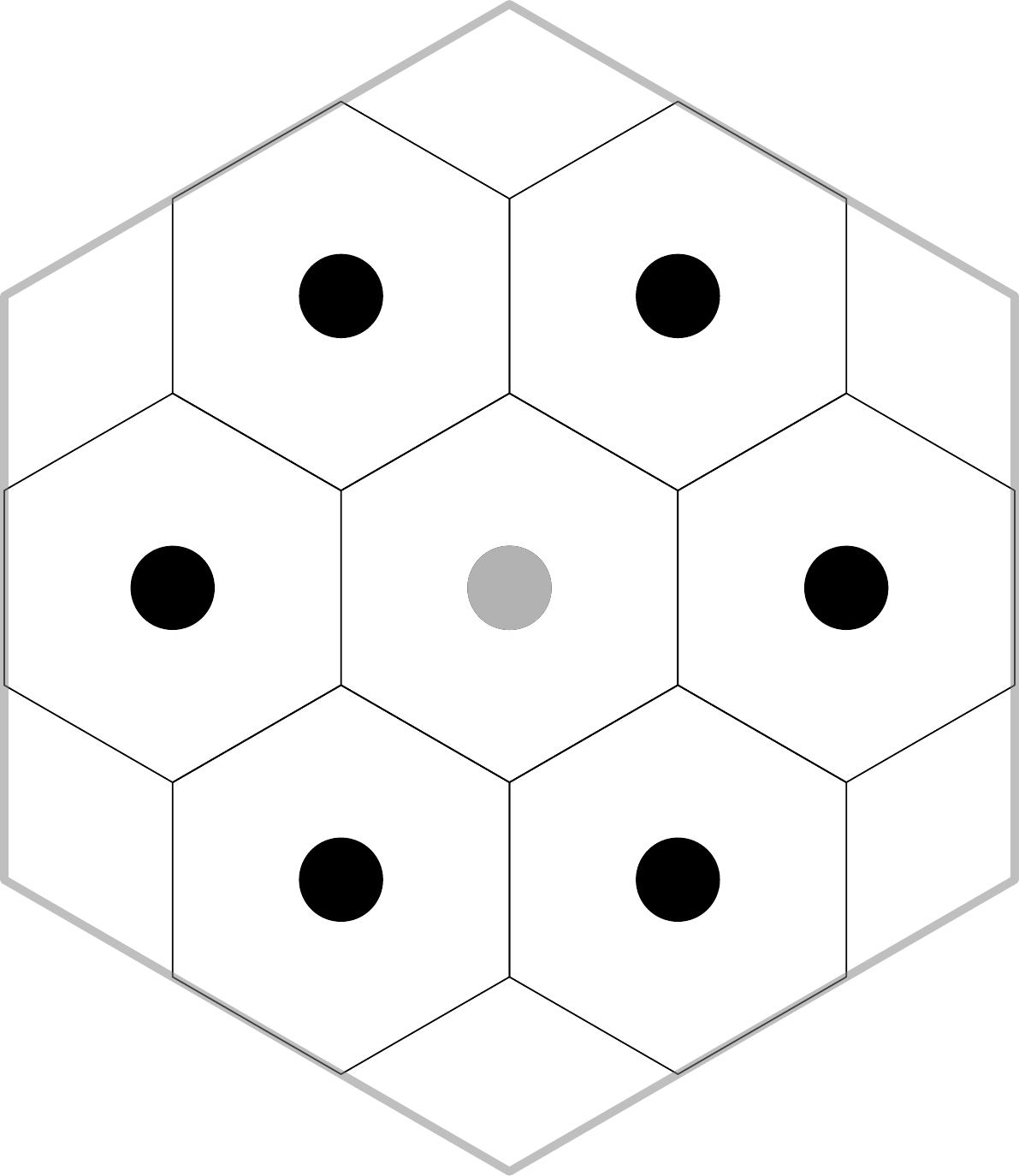}
    \label{fig:ds:3}}
  \quad
  \subfloat[Digit set for $\tau=\frac{3}{2} + \frac{1}{2}
  \sqrt{-3}$ and $w=3$.]{
    \includegraphics[width=0.20\linewidth]{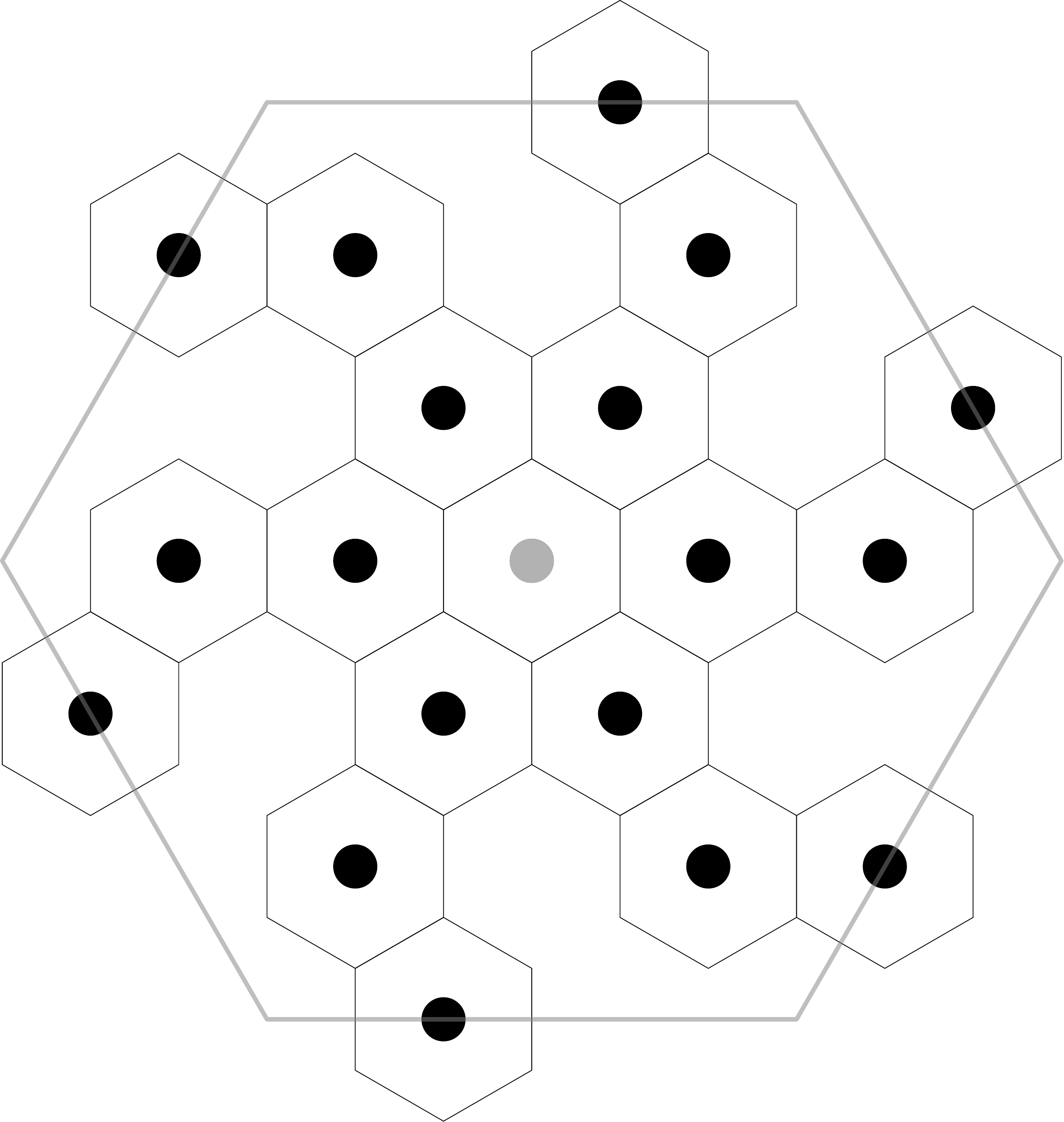}
    \label{fig:ds:4}}

  \caption[Digit sets for different $\tau$ and $w$]{Minimal norm representatives
    digit sets modulo $\tau^w$ for different $\tau$ and $w$. For each digit
    $\eta$, the corresponding Voronoi cell $V_\eta$ is drawn. The large scaled
    Voronoi cell is $\tau^w V$.}
  \label{fig:digit-sets}
\end{figure}


\begin{remark}
  \label{rem:choice-digit-set-voronoi-boundary}

  The definition of a representative of minimal norm,
  Definition~\vref{def:min-norm} --- and therefore the definition of a minimal
  norm representative digit set, Definition~\vref{def:min-norm-digit-set} ---
  depends on the definition of the restricted Voronoi cell $\wt{V}$,
  Definition~\ref{def:restr-voronoi}. There we had some freedom choosing which
  part of the boundary is included in $\wt{V}$, cf.\ the remarks after
  Definition~\ref{def:restr-voronoi}. 

  We point out that all results given here in this article are valid
  for any admissible configuration of the restricted Voronoi cell, although
  only the case corresponding to Definition~\ref{def:restr-voronoi} will be
  considered 
  (to prevent the paper from getting
  any longer). For most of the proofs given, this makes no difference, but
  there are some exceptions. The cases with ``problematic values'' in
  Sections~\ref{sec:bounds-value} and~\ref{sec:exist-uniqu-nafs} are solved
  algorithmically, and therefore the results clearly depend on choice of the
  digit set. Luckily the final results are true in each situation.
\end{remark}


The following remark summarises some basic properties of minimal norm
representatives and the defined digit sets.

\begin{remark}
  \label{rem:digitsets}

  Let $\tau$ be an algebraic integer, imaginary quadratic. We have the following
  equivalence. The condition
  \begin{equation*}
    \abs{\eta} \leq \abs{\xi} \text{ for all $\xi\in\Ztau$ with
      $\eta \equiv \xi \pmod{\tau^w}$}
  \end{equation*}
  is fulfilled, if and only if $\eta \in \tau^w V$. The advantage of using the
  restricted Voronoi cell in Definition~\vref{def:min-norm} is that also points
  on the boundary are handled uniquely. 

  Further we get for all $\eta\in\cD$ that $\abs{\eta} \leq \abs{\tau}^{w+1}
  c_V$. On the other side, if an element of $\Ztau$, which is not divisible by
  $\tau$, has absolute value less than $\frac12 \abs\tau^w$, cf.\
  Proposition~\vref{pro:voronoi-prop}, it is a digit. See also
  Lemma~\vref{lem:complete-res-sys}.

  Since $\cD \subseteq \Ztau$, all non-zero digits have absolute value at least
  $1$.
\end{remark}


We can assume that $0\leq\arg\tau\leq\frac{\pi}{2}$. Using any other $\tau$
lead to the same digit sets, except some mirroring at the real axis, imaginary
axis, or at the origin. By adapting the definition of the boundary of the
restricted Voronoi cell, Definition~\vref{def:restr-voronoi}, these mirroring
effects can be handled.


Now we are ready to define the numbers built with our digit set $\cD$. 


\begin{definition}[Width-$w$ \tauadic{} Non-Adjacent Forms]
  \label{def:wnaf}

  Let $\bfeta = \sequence{\eta_j}_{j\in\Z} \in \cD^\Z$. The sequence $\bfeta$ is
  called a \emph{width\nbd-$w$ \tauadic{} non-adjacent form}, or
  \emph{\wNAF{}} for short, if each factor $\eta_{j+w-1}\ldots\eta_j$, i.e.,
  each block of length $w$, contains at most one non-zero digit.

  Let $J=\set*{j\in\Z}{\eta_j\neq0}$. We call $\f{\sup}{\set{0} \cup (J+1)}$ the
  \emph{left-length of the \wNAF{} $\bfeta$} and $-\f{\inf}{\set{0} \cup J}$ the
  \emph{right-length of the \wNAF{} $\bfeta$}.

  Let $\lambda$ and $\rho$ be elements of $\N_0 \cup \set{\wNAFsetFIN,\infty}$,
  where $\wNAFsetFIN$ means finite. We denote the \emph{set of all \wNAF{}s of
    left-length at most $\lambda$ and right-length at most $\rho$} by
  $\wNAFsetellell[w]{\lambda}{\rho}$. If $\rho=0$, then we will simply write
  $\wNAFsetell[w]{\lambda}$. The elements of the set $\wNAFsetfin$ will be
  called \emph{integer \wNAF{}s}.

  For $\bfeta\in\wNAFsetfininf$ we call
  \begin{equation*}
    \NAFvalue{\bfeta} := \sum_{j\in\Z} \eta_j \tau^j
  \end{equation*}
  the \emph{value of the \wNAF{} $\bfeta$}.
\end{definition}


The following notations and conventions are used. A block of zero digits is
denoted by $\bfzero$. For a digit $\eta$ and $k\in\N_0$ we will use
\begin{equation*}
  \eta^k:=\underbrace{\eta\ldots\eta}_{k},
\end{equation*}
with the convention $\eta^0 := \eps$, where $\eps$ denotes the empty word.  A
\wNAF{} $\bfeta = \sequence{\eta_j}_{j\in\Z}$ will be written as
$\bfeta_I\bfldot\bfeta_F$, where $\bfeta_I$ contains the $\eta_j$ with $j\geq0$
and $\bfeta_F$ contains the $\eta_j$ with $j<0$. $\bfeta_I$ is called
\emph{integer part}, $\bfeta_F$ \emph{fractional part}, and the dot is called
\emph{\taupoint{}}. Left-leading zeros in $\bfeta_I$ can be skipped, except
$\eta_0$, and right-leading zeros in $\bfeta_F$ can be skipped as well. If
$\bfeta_F$ is a sequence containing only zeros, the \taupoint{} and this
sequence is not drawn.

Further, for a \wNAF{} $\bfeta$ (a bold, usually small Greek letter) we will
always use $\eta_j$ (the same letter, but indexed and not bold) for the
elements of the sequence.


To see where the values, respectively the fractional values of our \wNAF{}s lie
in the complex plane, have a look at Figure~\vref{fig:w-weta}. There some
examples are drawn.


The set $\wNAFsetfininf$ can be equipped with a metric. It is defined in the
following way. Let $\bfeta\in\wNAFsetfininf$ and $\bfxi\in\wNAFsetfininf$, then
\begin{equation*}
  \NAFd{\bfeta}{\bfxi} :=
  \begin{cases}
    \abs{\tau}^{\max\set*{j\in\Z}{\eta_j\neq\xi_j}} 
    & \text{if $\bfeta\neq\bfxi$,} \\
    0 & \text{if $\bfeta=\bfxi$.}
  \end{cases}
\end{equation*}
So the largest index, where the two \wNAF{}s differ, decides their distance. See
for example Edgar~\cite{Edgar:2008:measur} for details on such metrics.

We get the following continuity result.


\begin{proposition}
  \label{pro:value-continuous}

  The value function $\NAFvaluename$ is Lipschitz continuous on
  $\wNAFsetfininf$.
\end{proposition}


\begin{proof}
  Let $c_{\cD}$ be a bound for the absolute value of the digits in the digit set
  $\cD$. Let $\bfeta\in\wNAFsetfininf$ and
  $\bfxi\in\wNAFsetfininf$, $\bfeta\neq\bfxi$, with $\NAFd{\bfeta}{\bfxi} =
  \abs{\tau}^{J}$. Since $\bfeta$ and $\bfxi$ are equal on all digits
  with index larger than $J$ we obtain
  \begin{equation*}
    \abs{\NAFvalue{\bfeta} - \NAFvalue{\bfxi}}
    \leq \sum_{j \leq J} \abs{\eta_j - \xi_j} \abs{\tau}^j
    \leq 2 c_{\cD} \frac{\abs{\tau}^J}{1-\abs{\tau}^{-1}}
    =  \frac{2 c_{\cD}}{1-\abs{\tau}^{-1}}\NAFd{\bfeta}{\bfxi}.
  \end{equation*}
  Thus Lipschitz continuity is proved.
\end{proof}


Furthermore, we get a compactness result on the metric space
$\wNAFsetellinf{\ell} \subseteq \wNAFsetfininf$ in the proposition below. The
metric space $\wNAFsetfininf$ is not compact, because if we fix a non-zero digit
$\eta$, then the sequence $\sequence{\eta0^j}_{j\in\N_0}$ has no convergent
subsequence, but all $\eta0^j$ are in the set $\wNAFsetfininf$.


\begin{proposition}
  \label{pro:naf-compact}

  For every $\ell\geq0$ the metric space
  $\left(\wNAFsetellinf{\ell},\NAFdname\right)$ is compact.
\end{proposition}


\begin{proof}
  Let $\sequence{\bfxi_{0,j}}_{j\in\N_0}$ be a sequence with
  $\bfxi_{0,j}\in\wNAFsetellinf{\ell}$. We can assume
  $\bfxi_{0,j}\in\wNAFsetinf$, therefore each word $\bfxi_{0,j}$ has digits zero
  for non-negative index. Now consider the digit with index $-1$. There is a
  subsequence $\sequence{\bfxi_{1,j}}_{j\in\N_0}$ of
  $\sequence{\bfxi_{0,j}}_{j\in\N_0}$, such that digit $-1$ is a fixed digit
  $\eta_{-1}$. Next there is a subsequence $\sequence{\bfxi_{2,j}}_{j\in\N_0}$
  of $\sequence{\bfxi_{1,j}}_{j\in\N_0}$, such that digit $-2$ is a fixed digit
  $\eta_{-2}$. This process can be repeated for each $k \geq 1$ to get sequences
  $\sequence{\bfxi_{k,j}}_{j\in\N_0}$ and digits $\eta_{-k}$.

  The sequence $\sequence{\bfvartheta_j}_{j\in\N_0}$ with $\bfvartheta_j :=
  \bfxi_{j,j}$ converges to $\bfeta$, since for $\eps>0$ there is an $J\in\N_0$
  such that for all $j \geq J$
  \begin{equation*}
    \NAFd{\bfeta}{\bfvartheta_j} \leq \abs\tau^{-(J+1)} < \eps.
  \end{equation*}
  It is clear that $\bfeta$ is indeed an element of $\wNAFsetinf$,
  as its first $k$ digits coincide with $\bfxi_{k,k}\in \wNAFsetinf$
  for all $k$.
  So we have found a converging subsequence of
  $\sequence{\bfxi_{0,j}}_{j\in\N_0}$, which proves the compactness.
\end{proof}


\begin{remark}
  \label{rem:naf-compact}

  The compactness of $\left(\wNAFsetellinf{\ell},\NAFdname\right)$ can also be
  deduced from general theory. As a consequence of Tychonoff's Theorem the set
  $\cD^\N$ is a compact space, the product topology (of the discrete topology
  on $\cD$) coincides with the topology
  induced by the obvious generalisation of the metric $\NAFdname$. The subset
  $\wNAFsetinf\subseteq\cD^\N$ is closed and therefore compact, too.
\end{remark}


We want to express all integers in $\Z[\tau]$ by finite \wNAF{}s. Thus we
restrict ourselves to suitable digit sets, cf. Muir and
Stinson~\cite{Muir-Stinson:2004:alter-digit}.


\begin{definition}[Width-$w$ Non-Adjacent Digit Set]
  \label{def:wnads}

  A digit set $\cD$ is called a \emph{width\nbd-$w$ non-adjacent digit set}, or
  \wNADS{} for short, when every element $z\in\Ztau$ admits a unique \wNAF{}
  $\bfeta\in\wNAFsetfin$, i.e., $\NAFvalue{\bfeta} = z$. When this is the case,
  the function
  \begin{equation*}
    \NAFvaluename\restricted{\wNAFsetfin}\colon \wNAFsetfin \fto \Ztau
  \end{equation*}
  is bijective, and we will denote its inverse function by $\NAFwname$.
\end{definition}


Later, namely in Section~\ref{sec:exist-uniqu-nafs}, we will see that the digit
set of minimal norm representatives is a \wNADS{} if $\tau$ is imaginary
quadratic.


\section{Full Block Length Analysis of Non-Adjacent Forms}
\label{sec:full-block-length}


Let $\tau\in\C$ be an algebraic integer, $w\in\N$ with $w\geq2$, and $\cD$ be a
reduced residue digit set, cf.\
Definition~\vref{def:red-residue-digit-set}. Let $\normsymbol\colon\Ztau\fto\Z$
denote the norm function.

Further, in this section all \wNAF{}s will be out of the set $\wNAFsetfin$, and
with \emph{length} the left-length is meant.

This general setting allows us to analyse digit frequencies under the
\emph{full block length model,}, i.e., we assume that all \wNAF{}s of given
length are equally likely. We will prove the following theorem.


\begin{theorem}[Full Block Length Distribution
  Theorem]\label{th:w-naf-distribution}
  We denote the number of \wNAF{}s of length $n\in\N_0$ by $C_{n,w}$, i.e.,
  $C_{n,w}=\card{\wNAFsetell{n}}$, and we get
  \begin{equation*}
    C_{n,w} = \frac{1}{(\abs{\normtau}-1)w+1}\abs{\normtau}^{n+w}
    +\Oh{(\rho\abs{\normtau})^n},
  \end{equation*}
  where $\rho=(1+\frac{1}{\abs{\normtau} w^3})^{-1}<1$.

  Further let $0\neq \eta\in\cD$ be a fixed digit and define the random
  variable $X_{n,w,\eta}$ to be the number of occurrences of the digit $\eta$ in
  a random \wNAF{} of length $n$, where every \wNAF{} of length $n$ is assumed
  to be equally likely.

  Then the following explicit expressions hold for the expectation and the
  variance of $X_{n,w,\eta}$:
  
  \begin{align}
    \expect{X_{n,w,\eta}}&=e_{w}n +\frac{(\abs{\normtau} -1) (w-1) w
    }{\abs{\normtau}^{w-1}((\abs{\normtau}-1)
      w+1)^2}+\Oh{n\rho^n}\label{eq:naf-expectation}\\
    \variance{X_{n,w,\eta}}&=v_w n \notag \\
    &\phantom{=}+\frac{ \scriptstyle (w-1) w
      \left(-(w-1)^2-\abs{\normtau}^2
        w^2+(\abs{\normtau}-1) \abs{\normtau}^{w-1} ((\abs{\normtau}-1)
        w+1)^2+2 \abs{\normtau}
        \left(w^2-w+1\right)\right)}{\abs{\normtau}^{2w-2}
      ((\abs{\normtau}-1) w+1)^4} 
    \notag \\
    &\phantom{=}+\Oh{n^2\rho^n}, \label{eq:naf-variance}
  \end{align}
  where 
  \begin{align*}
    e_w&=\frac{ 1}{\abs{\normtau}^{w-1}((\abs{\normtau} -1) w+1)},\\
    \intertext{and}
    v_w&=\frac{\abs{\normtau}^{w-1} 
      ((\abs{\normtau} -1)w+1)^2-\left((\abs{\normtau} -1)w^2+2 
        w-1\right)}{\abs{\normtau}^{2w-2}((\abs{\normtau}-1)
      w+1)^3}.
  \end{align*}

  Furthermore, $X_{n,w,\eta}$ satisfies the central limit theorem
  \begin{equation*}
    \prob{\frac{X_{n,w,\eta}-e_w n}{\sqrt{v_w n}} \le
      x}=\Phi(x)+\Oh{\frac1{\sqrt n}},
  \end{equation*}
  uniformly with respect to $x\in\R$, where $\Phi(x)=(2\pi)^{-1/2}\int_{t\le x}
  e^{-t^2/2}\,dt$ is the standard normal distribution.
\end{theorem}


For the proof we need estimates for the zeros of a polynomial which will be
needed for estimating the non-dominant roots of our generating function.


\begin{lemma}\label{le:wagner-lemma}
  Let $t\ge 2$ and 
  \begin{equation*}
    f(z)=1-\frac1t z-\left(1-\frac1t\right)z^w.
  \end{equation*}
  Then $f(z)$ has exactly one root with $|z|\le 1+\frac1{ t w^3}$, namely $z=1$.
\end{lemma}
\begin{proof}
  It is easily checked that $f(1)=0$.
  Assume that $z\neq 1$ is another root of $f$. As the coefficients of $f$ are
  reals, it is sufficient to consider $z$ with $\im{z}\ge 0$. If $|z|<1$, then 
  \begin{equation*}
    1=\left|\frac1t z+\left(1-\frac1t\right)z^w\right|\le\frac1t |z|
    +\left(1-\frac1t\right)|z|^w<\frac1t+\left(1-\frac1t\right)=1,
  \end{equation*}
  which is a contradiction. Therefore, we have $|z|\ge 1$.  We write $z=r
  e^{i\psi}$ for appropriate $r\ge 1$ and $0\le \psi\le \pi$. For $r>0$, $f(r)$
  is strictly decreasing, so we can assume that $\psi>0$.

  For $\psi<\pi/w$, we have  $\sin(w\psi)> 0$ and
  $\sin(\psi)>0$, which implies that $\im{f(re^{i\psi})}=-\frac{1}t
  \sin\psi-\left(1-\frac1t\right)\sin(\psi w)<0$, a contradiction.
  We conclude that $\psi\ge \pi/w$.

  Next, we see that $f(r e^{i\psi})=0$ implies that
  \begin{equation*}
    1-\frac{2r}{t}\cos\psi+\frac{r^2}{t^2}
    =\left| 1-\frac1t re^{i\psi}\right|^2=\left(1-\frac1t\right)^2 r^{2w}.
  \end{equation*}
  We have $\cos\psi\le \cos(\pi/w)$, which implies that 
  \begin{equation}\label{eq:wagner-inequality}
    \left(1-\frac1t\right)^2 r^{2w} \ge 1-\frac{2r}{t}\cos\frac{\pi}w
    +\frac{r^2}{t^2}
    =\left(1-\frac{r}t\right)^2+\left(1-\cos\frac{\pi}w\right)\frac{2r}t.
  \end{equation}
  For $w\ge 4$ and $r<\sqrt{2}$, the right hand side of
  \eqref{eq:wagner-inequality} is decreasing and the left hand side is
  increasing. Thus, for $r\le 1+1/(t w^3)$, \eqref{eq:wagner-inequality} yields
  \begin{equation*}
    \left(1-\frac1t\right)^2 \left(1+\frac1{t w^3}\right)^{2w}
    \ge\left(1-\frac1t\cdot\left(1+\frac1{t
          w^3}\right)\right)^2+\left(1-\cos\frac{\pi}w\right)
    \frac{2}t\left(1+\frac1{t w^3}\right).
  \end{equation*}
  Using the estimates $(1+\frac1{tw^3})^{2w}\le 1+\frac2{ t w^2}+\frac{2}{t^2
    w^4}$ and $\cos\left(\frac{\pi}{w}\right)\le
  1-\frac{\pi^2}{2w^2}+\frac{\pi^4}{24w^4}$, we obtain
  \begin{multline*}
    \left(\frac{2-\pi^2}{t}-\frac{4}{t^2}+\frac{2}{t^3}\right)\frac{1}{w^2}
    +\left(\frac{2}{t^2}-\frac{2}{t^3}\right)\frac{1}{w^3}\\
    +\left(\frac{\pi^4}{12t}+\frac{2}{t^2}-\frac{4}{t^3}
      +\frac{2}{t^4}\right)\frac{1}{w^4}
    -\frac{\pi^2}{t^2w^5}
    -\frac{1}{t^4w^6}
    +\frac{\pi^4}{12t^2w^7}\ge 0,  
  \end{multline*}
  which is a contradiction for $w\ge 4$ and $t\ge 2$.

  For $w=3$, we easily check that $|z|=\sqrt{\frac{t}{t-1}}\ge 1+1/(27t)$;
  similarly, for $w=2$, we have $|z|=t/(t-1)>1+1/(4t)$.
\end{proof}


\begin{proof}[Proof of Theorem~\ref{th:w-naf-distribution}]
  For simplicity we set $\cD^\bullet := \cD \setminus \set{0}$. A \wNAF{} can be
  described by the regular expression
  \begin{equation*}
    \left(\eps+\sum_{d\in\cD^\bullet} 
      \sum_{k=0}^{w-2}0^k d\right)\left(0+\sum_{d\in\cD^\bullet} 
      0^{w-1}d\right)^\bfast
  \end{equation*}

  Let $a_{mn}$ be the number of \wNAF s of length $n$ containing exactly $m$
  occurrences of the digit $\eta$. We consider the generating function
  $G(Y,Z)=\sum_{m,n} a_{mn}Y^m Z^n$. From the regular expression we see that
  \begin{equation*}
    G(Y,Z)=\frac{1+(Y+(\#\cD^\bullet -1))
      \frac{Z^{w}-Z}{Z-1}}{1-Z-Y Z^w-(\#\cD^\bullet-1) Z^w}.
  \end{equation*}

  We start with determining the number of \wNAF s of length $n$. This amounts
  to extracting the coefficient of $Z^n$ of
  \begin{equation*}
    G(1,Z)=\frac{1+(\#\cD^\bullet -1)Z-\#\cD^\bullet\cdot Z^w}{(1-Z)
      (1-Z-\#\cD^\bullet\cdot Z^w)}.
  \end{equation*}
  This requires finding the dominant root of the denominator. Setting
  $z=\abs{\normtau} Z$ in the second factor yields
  \begin{equation*}
    1-\frac{1}{\abs{\normtau}}z-\left(1-\frac{1}{\abs{\normtau}}\right)z^{w}.
  \end{equation*}
  From Lemma~\vref{le:wagner-lemma}, we see that the dominant root of the
  denominator of $G(1,Z)$ is $Z=1/\abs{\normtau}$, and that all other roots
  satisfy $|Z|\ge 1/\abs{\normtau}+1/(\abs{\normtau}^2w^3)$.  Extracting the
  coefficient of $Z^n$ of $G(1,Z)$ then yields the number $C_{n,w}$ of $w$-NAFs
  of length $n$ as
  \begin{equation}\label{eq:naf-moment-0}
    \frac{1}{(\abs{\normtau}-1)w+1}\abs{\normtau}^{n+w}
    +\Oh{(\rho\abs{\normtau})^n},
  \end{equation}
  where $\rho=(1+\frac{1}{\abs{\normtau} w^3})^{-1}$.

  The number of occurrences of the digit $\mu$ amongst all $w$-NAFs of length
  $n$ is
  \begin{multline*}
    \coefficient{Z^n}\left.\frac{\partial G(Y,Z)}{\partial Y}\right|_{Y=1}=
    \frac{Z}{\left(1-Z -\left(1-\frac{1}{\abs{\normtau}}\right)
        (\abs{\normtau} Z)^w\right)^2}\\
    =
    \frac{ 1}{((\abs{\normtau} -1) w+1)^2}n 
    \abs{\normtau}^{n+1}+\frac{(\abs{\normtau} -1) (w-1) w
    }{((\abs{\normtau}-1)  w+1)^3}\abs{\normtau}^{n+1}
    +\Oh{(\rho\abs{\normtau})^n}.
  \end{multline*}
  Dividing this by \eqref{eq:naf-moment-0} yields \eqref{eq:naf-expectation}.

  In order to compute the second moment, we compute
  \begin{multline*}
    \coefficient{Z^n}\left.\frac{\partial^2 G(Y,Z)}{\partial Y^2}\right|_{Y=1}=
    \frac{2 Z^{w+1}}{\left(1- Z-\left(1-\frac1{\abs{\normtau}}\right)
        (\abs{\normtau} Z)^w\right)^3} \\
    =\frac{1}{((\abs{\normtau}-1) w+1)^3}n^2\abs{\normtau}^{n-w+2}
    +\frac{ \left((\abs{\normtau}-1)w^2 -2w\abs{\normtau} +1\right)
    }{((\abs{\normtau}-1)
      w+1)^4} n\abs{\normtau}^{n-w+2}\\
    -\frac{(w-1) w \left(w \abs{\normtau}^2-2
        \abs{\normtau}-w+1\right)
    }{((\abs{\normtau}-1)
      w+1)^5}\abs{\normtau}^{n-w+2}
    +\Oh{(\rho\abs{\normtau})^n},
  \end{multline*}
  which after division by \eqref{eq:naf-moment-0} yields
  \begin{multline*}
    \expect{X_{n,w,\eta}(X_{n,w,\eta}-1)}=
    \frac{1}{\abs{\normtau}^{2w-2}((\abs{\normtau}-1) w+1)^2}n^2 \\
    +\frac{ \left((\abs{\normtau}-1)w^2 -2w\abs{\normtau} +1\right) 
    }{\abs{\normtau}^{2w-2}((\abs{\normtau}-1) w+1)^3} n
    -\frac{(w-1) w \left(w \abs{\normtau}^2-2
        \abs{\normtau}-w+1\right)
    }{\abs{\normtau}^{n-w+2}((\abs{\normtau}-1)  w+1)^4} \\
    +\Oh{n^2\rho^n}.
  \end{multline*}
  Adding $\expect{X_{n,w,\eta}}-\expect{X_{n,w,\eta}}^2$ yields the variance
  given in \eqref{eq:naf-variance}.

  The asymptotic normality follows from Hwang's
  Quasi-Power-Theorem~\cite{Hwang:1998}.
\end{proof}


\section{Bounds for the Value of Non-Adjacent Forms}
\label{sec:bounds-value}


Let $\tau\in\C$ be an algebraic integer, imaginary quadratic with minimal
polynomial $X^2 - p X + q$ with $p$, $q\in\Z$ such that $4q-p^2>0$. Suppose that
$\abs{\tau}>1$. Let $w\in\N$ with $w\geq2$. Further let $\cD$ be a minimal norm
representatives digit set modulo $\tau^w$ as in
Definition~\vref{def:min-norm-digit-set}.

In this section the \emph{fractional value} of a \wNAF{} means the value of a
\wNAF{} of the form $0\bfldot\bfeta$. The term \emph{most significant digit} is
used for the digit $\eta_{-1}$.


So let us have a closer look at the fractional value of a \wNAF{}. We want to
find upper bounds and if we fix a digit, e.g.\ the most significant one, a lower
bound. We need two different approaches to prove those results. The first one is
analytic. The results there are valid for all combinations of $\tau$ and $w$
except finitely many. These exceptional cases will be called ``problematic
values''. To handle those, we will use an other idea. We will show an
equivalence, which directly leads to a simple procedure to check, whether a
condition is fulfilled. If this is the case, the procedure terminates and
returns the result. This idea is similar to a proof in
Matula~\cite{Matula:1982:basic}.

The following proposition deals with three upper bounds, one for
the absolute value and two give us regions containing the fractional value.


\begin{proposition}[Upper Bounds for the Fractional Value]
  \label{pro:upper-bound-fracnafs}

  Let $\bfeta\in\wNAFsetinf$, and let   
  \begin{equation*}
    f_U = \frac{\abs{\tau}^w c_V}{1 - \abs{\tau}^{-w}}.
  \end{equation*}
  Then the following statements are true:
 
  \begin{enumerate}[(a)]

  \item \label{enu:upper-bound:leq-f} We get
    \begin{equation*}
      \abs{\NAFvalue{\bfeta}} \leq f_U.
    \end{equation*}
  
  \item \label{enu:upper-bound:in-balls} Further we have
    \begin{equation*}
      \NAFvalue{\bfeta} \in 
      \bigcup_{z \in \tau^{w-1} V} \ball*{z}{\abs{\tau}^{-w} f_U}.
    \end{equation*}

  \item \label{enu:upper-bound:equiv} The following two statements are
    equivalent:
    \begin{enumequivalences}

    \item \label{enu:proc-ub:1} There is an $\ell\in\N_0$, such that for
      all $\bfxi\in\wNAFsetellell{0}{\ell}$ the condition
      \begin{equation*}
        \ball*{\NAFvalue{\bfxi}}{\abs{\tau}^{-\ell} f_U}
        \subseteq \tau^{2w-1} \interior{V}
      \end{equation*}
      is fulfilled. 

    \item \label{enu:proc-ub:2} There exists an $\eps>0$, such that for all
      $\bfvartheta \in \wNAFsetinf$ the condition
      \begin{equation*}
        \ball{\NAFvalue{\bfvartheta}}{\eps} \subseteq \tau^{2w-1} V
      \end{equation*}
      holds.

    \end{enumequivalences}

  \item \label{enu:upper-bound:in-v} We get
    \begin{equation*}
      \NAFvalue{\bfeta} \in \tau^{2w-1} V.
    \end{equation*}

  \item \label{enu:upper-bound:approx-in-v} For $\ell\in\N_0$ we have
    \begin{equation*}
      \NAFvalue{0\bfldot\eta_{-1}\ldots\eta_{-\ell}} + \tau^{-\ell}V 
      \subseteq \tau^{2w-1} V.
    \end{equation*}

  \end{enumerate}
\end{proposition}


\begin{proof}
  \begin{enumerate}[(a)]

  \item We have
    \begin{equation*}
      \abs{\NAFvalue{\bfeta}}
      = \abs{\sum_{j=1}^{\infty} \eta_{-j} \tau^{-j}}
      \leq \sum_{j=1}^{\infty} \abs{\eta_{-j}} \abs{\tau}^{-j}.
    \end{equation*}
    We consider \wNAF{}s, which have $\eta_{-j} \neq 0$ for $-j \equiv 1
    \pmod{w}$. For all other \wNAF{}s the upper bound is smaller. To see this,
    assume that there are more than $w-1$ adjacent zeros in a \wNAF{} or the
    first digits are zero. Then we could build a larger upper bound by shifting
    digits to the left, i.e., multiplying parts of the sum by $\abs{\tau}$,
    since $\abs{\tau} > 1$.

    We get
    \begin{align*}
      \abs{\NAFvalue{\bfeta}}
      &\leq \sum_{j=1}^{\infty} \abs{\eta_{-j}} \abs{\tau}^{-j}
      = \sum_{j=1}^{\infty} \tiverson{-j \equiv 1 \pmod{w}}
      \abs{\eta_{-j}} \abs{\tau}^{-j} \\
      &\leq \abs{\tau}^{-1} \sum_{k=0}^{\infty} 
      \abs{\eta_{-(wk+1)}} \abs{\tau}^{-wk},
    \end{align*}
    in which we changed the summation index according to $wk+1=j$ and the
    Iversonian notation $\iverson{\var{expr}}=1$ if $\var{expr}$ is true and
    $\iverson{\var{expr}}=0$ otherwise, cf.\ Graham, Knuth and
    Patashnik~\cite{Graham-Knuth-Patashnik:1994}, has been used. Using
    $\abs{\eta_{-(wk+1)}} \leq \abs{\tau}^{w+1} c_V$, see
    Remark~\vref{rem:digitsets}, yields
    \begin{equation*}
      \abs{\NAFvalue{\bfeta}}
      \leq \abs{\tau}^{-1} \abs{\tau}^{w+1} c_V \frac{1}{1-\abs{\tau}^{-w}}
      = \underbrace{\frac{\abs{\tau}^w c_V}{1 - \abs{\tau}^{-w}}}_{=: f_U}.
    \end{equation*}

  \item There is nothing to show if the \wNAF{} $\bfeta$ is zero, and it is
    sufficient to prove it for $\eta_{-1} \neq 0$. Otherwise, let $k\in\N$ be
    minimal, such that $\eta_{-k} \neq 0$. Then
    \begin{equation*}
      \tau^{-(k-1)} \tau^{k-1} \NAFvalue{\bfeta} \in 
      \tau^{-(k-1)} \bigcup_{z \in \tau^{w-1} V} \ball*{z}{\abs{\tau}^{-w} f_U}
      \subseteq \bigcup_{z \in \tau^{w-1} V} \ball*{z}{\abs{\tau}^{-w} f_U},
    \end{equation*}
    since $\abs{\tau} > 1$ and $\tau^{-1} V \subseteq V$, see
    Proposition~\vref{pro:voronoi-prop}.

    Since $\eta_{-1} \in \tau^w V$, see Remark~\vref{rem:digitsets}, we obtain
    $\eta_{-1}\tau^{-1} \in \tau^{w-1} V$. Thus, using
    \itemref{enu:upper-bound:leq-f}, yields
    \begin{equation*}
      \abs{\tau^w \left(\NAFvalue{\bfeta} - \eta_{-1}\tau^{-1}\right)}
      \leq f_U,
    \end{equation*}
    i.e., 
    \begin{equation*}
      \NAFvalue{\bfeta} \in \ball*{\eta_{-1}\tau^{-1}}{\abs{\tau}^{-w} f_U},
    \end{equation*}
    which proves the statement.

  \item
    \begin{descproofequivalences}

    \item[\labelproofequivalent{\ref{enu:proc-ub:1}}{\ref{enu:proc-ub:2}}]
      Suppose there exists such an $\ell\in\N$. Then there exists an $\eps>0$
      such that
      \begin{equation*}
        \ball*{\NAFvalue{\bfxi}}{\abs{\tau}^{-\ell} f_U + \eps}
        \subseteq \tau^{2w-1} V
      \end{equation*}
      for all $\bfxi\in\wNAFsetellell{0}{\ell}$, since there are only finitely
      many $\bfxi\in\wNAFsetellell{0}{\ell}$. Let
      $\bfvartheta\in\wNAFsetinf$. Then there is a
      $\bfxi\in\wNAFsetellell{0}{\ell}$ such that the digits from index $-1$ to
      $-\ell$ of $\bfxi$ and $\bfvartheta$ coincide. By using
      \itemref{enu:upper-bound:leq-f} we obtain
      \begin{equation*}
        \abs{\NAFvalue{\bfvartheta} - \NAFvalue{\bfxi}} 
        \leq \abs{\tau}^{-\ell} f_U,
      \end{equation*}
      and thus 
      \begin{equation*}
        \ball*{\NAFvalue{\bfvartheta}}{\eps} \subseteq
        \ball*{\NAFvalue{\bfxi}}{\abs{\tau}^{-\ell} f_U + \eps}
        \subseteq \tau^{2w-1} V.
      \end{equation*}

    \item[\labelproofequivalent{\ref{enu:proc-ub:2}}{\ref{enu:proc-ub:1}}]
      Now suppose there is such an $\eps>0$. Since there is an
      $\ell\in\N$ such that $\abs{\tau}^{-\ell} f_U < \eps$, the statement
      follows. 
      
    \end{descproofequivalences}

  \item We know from Proposition~\vref{pro:voronoi-prop} that
    $\ball*{0}{\frac12\abs{\tau}^{2w-1}} \subseteq \tau^{2w-1} V$. Therefore, if
    the upper bound found in \itemref{enu:upper-bound:leq-f} fulfils
    \begin{equation*}
      f_U 
      = \frac{\abs{\tau}^w c_V}{1 - \abs{\tau}^{-w}}
      \leq \frac12\abs{\tau}^{2w-1},
    \end{equation*}
    the statement follows. 
    
    The previous inequality is equivalent to
    \begin{equation*}
      \nu :=  \frac{1}{2} - \frac{\abs{\tau} c_V}{\abs{\tau}^{w}-1} \geq 0.
    \end{equation*}
    The condition is violated for $w=2$ and $\abs{\tau}$ equal to $\sqrt{2}$,
    $\sqrt{3}$ or $\sqrt{4}$, and for $w=3$ and $\abs{\tau}=\sqrt{2}$, see
    Table~\vref{tab:values-of-nu}. Since $\nu$ is monotonic increasing for
    $\abs{\tau}$ and for $w$, there are no other ``problematic cases''.

    For those cases we will use~\itemref{enu:upper-bound:equiv}. For each of the
    ``problematic cases'' an $\ell$ satisfying the
    condition~\itemref{enu:proc-ub:1} of equivalences
    in~\itemref{enu:upper-bound:equiv} was found, see
    Table~\vref{tab:upper-bound-problematic} for the results. Thus the statement
    is proved.

  \item Analogously to the proof
    of~\itemref{enu:upper-bound:leq-f}, except that we use $\ell$ for the upper
    bound of the sum, we obtain for $v\in V$
    \begin{align*}
      \abs{\NAFvalue{0\bfldot\eta_{-1}\ldots\eta_{-\ell}} + \tau^{-\ell}v}
      &\leq \abs{\NAFvalue{0\bfldot\eta_{-1}\ldots\eta_{-\ell}}} 
      + \abs{\tau^{-\ell}}\abs{\tau} c_V \\
      &\leq \frac{\abs{\tau}^w c_V}{1 - \abs{\tau}^{-w}}
      \left( 1 - \abs{\tau}^{-w \floor{\frac{\ell-1+w}{w}}} \right)
      + \abs{\tau}^{-\ell+1} c_V \\
      &\leq \frac{\abs{\tau}^w c_V}{1 - \abs{\tau}^{-w}}
      \left( 1 - \abs{\tau}^{-\ell+1-w} 
        + \abs{\tau}^{-\ell+1-w} \left(1 - \abs{\tau}^{-w}\right) \right) \\
      &= \frac{\abs{\tau}^w c_V}{1 - \abs{\tau}^{-w}}
      \left( 1 - \abs{\tau}^{-\ell+1-2w} \right).
    \end{align*}
    Since $1 - \abs{\tau}^{-\ell+1-2w} < 1$ we get
    \begin{equation*}
      \abs{\NAFvalue{0\bfldot\eta_{-1}\ldots\eta_{-\ell}} + \tau^{-\ell}V}
      \leq \frac{\abs{\tau}^w c_V}{1 - \abs{\tau}^{-w}} = f_U
    \end{equation*}
    for all $\ell\in\N_0$.

    Let $z\in\C$. Have again a look at the proof
    of~\itemref{enu:upper-bound:in-v}. If $\nu>0$ there, we get that $\abs{z}
    \leq f_U$ implies $z \in \tau^{2w-1}V$.

    Combining these two results yields the inclusion for $\nu>0$, i.e., the
    ``problematic cases'' are left. Again, each of these cases has to be
    considered separately.

    For each of the problematic cases, we find a $k\in\N_0$ such that
    \begin{equation*}
      \ball*{\NAFvalue{\bfxi}}{2\abs{\tau}^{-k} f_U}\subseteq \tau^{2w-1} V
    \end{equation*}
    holds for all $\bfxi\in\wNAFsetellell{0}{k}$. These $k$ are listed in
    Table~\vref{tab:upper-bound-approx-problematic}.

    For $\ell>k$ and  $v \in V$, we obtain
    \begin{equation*}
      \begin{split}
        \abs{\NAFvalue{0\bfldot0\ldots0\eta_{-(k+1)}\ldots\eta_{-\ell}} 
          + \tau^{-\ell}v}
        &\leq \abs{\tau}^{-k} f_U + \abs{\tau}^{-\ell} \abs{\tau} c_V \\
        &\leq \abs{\tau}^{-k} f_U \left( 1 + \frac{c_V}{f_U} \right) \\
        &= \abs{\tau}^{-k} f_U \left( 1 
          + \frac{1-\abs{\tau}^{-w}}{\abs{\tau}^w} \right) \\
        &\leq 2 \abs{\tau}^{-k} f_U
      \end{split}
    \end{equation*}
    using \itemref{enu:upper-bound:leq-f}, Proposition~\vref{pro:voronoi-prop},
    and $\abs{\tau}^w > 1$. Thus the  desired inclusion follows for $\ell>k$.
    
    For the finitely many $\ell \leq k$ we additionally check all possibilities,
    i.e., whether for all combinations of $\bfvartheta \in
    \wNAFsetellell{0}{\ell}$ and vertices of the boundary of
    $\NAFvalue{\bfvartheta} + \tau^{-\ell} V$ the corresponding value is inside
    $\tau^{2w-1}V$. Convexity of $V$ is used here. All combinations were valid,
    see last column of Table~\vref{tab:upper-bound-approx-problematic}, thus the
    inclusion proved.  \qedhere
  \end{enumerate}
\end{proof}


\begin{table}
  \centering
  \begin{equation*}
\begin{array}{c|ccc}
& w=2
& w=3
& w=4
\\
\hline
\abs{\tau}=\sqrt{2}
& -0.58012
& -0.09074
& \phantom{+}0.13996
\\
\abs{\tau}=\sqrt{3}
& -0.16144
& \phantom{+}0.18474
& \phantom{+}0.33464
\\
\abs{\tau}=\sqrt{4}
& -0.00918
& \phantom{+}0.28178
& \phantom{+}0.39816
\\
\abs{\tau}=\sqrt{5}
& \phantom{+}0.07304
& \phantom{+}0.33224
& \phantom{+}0.42884
\end{array}
\end{equation*}

  \caption[Values  of $\nu$]{Values (given five decimal places) of 
    $\nu = \frac{1}{2} - \frac{\abs{\tau} c_V}{\abs{\tau}^{w}-1}$ for 
    different $\abs{\tau}$ and 
    $w$. A negative sign means that this value is a ``problematic value''.}
  \label{tab:values-of-nu}
\end{table}


\begin{table}
  \centering
  \begin{equation*}
\begin{array}{cc|cc|c|cccc}
\hline
q=\abs{\tau}^2 & p & \re{\tau} & \im{\tau} & w & 
\text{$\ell$ found?} & \ell & \abs{\tau}^{-\ell}f_U & \eps \\
\hline
2 & -2 & -1 & 1 & 2 & \var{true} & 8 & 0.1909 & 0.03003 \\
2 & -1 & -0.5 & 1.323 & 2 & \var{true} & 4 & 0.7638 & 0.02068 \\
2 & 0 & 0 & 1.414 & 2 & \var{true} & 6 & 0.3819 & 0.1484 \\
2 & 1 & 0.5 & 1.323 & 2 & \var{true} & 4 & 0.7638 & 0.02068 \\
2 & 2 & 1 & 1 & 2 & \var{true} & 7 & 0.27 & 0.08352 \\
\hline
2 & -2 & -1 & 1 & 3 & \var{true} & 2 & 1.671 & 0.4505 \\
2 & -1 & -0.5 & 1.323 & 3 & \var{true} & 2 & 1.671 & 0.4726 \\
2 & 0 & 0 & 1.414 & 3 & \var{true} & 2 & 1.671 & 0.4505 \\
2 & 1 & 0.5 & 1.323 & 3 & \var{true} & 2 & 1.671 & 0.4726 \\
2 & 2 & 1 & 1 & 3 & \var{true} & 2 & 1.671 & 0.4505 \\
\hline
3 & -3 & -1.5 & 0.866 & 2 & \var{true} & 1 & 1.984 & 0.03641 \\
3 & -2 & -1 & 1.414 & 2 & \var{true} & 2 & 1.146 & 0.5543 \\
3 & -1 & -0.5 & 1.658 & 2 & \var{true} & 2 & 1.146 & 0.4581 \\
3 & 0 & 0 & 1.732 & 2 & \var{true} & 1 & 1.984 & 0.03641 \\
3 & 1 & 0.5 & 1.658 & 2 & \var{true} & 2 & 1.146 & 0.4581 \\
3 & 2 & 1 & 1.414 & 2 & \var{true} & 2 & 1.146 & 0.5543 \\
3 & 3 & 1.5 & 0.866 & 2 & \var{true} & 1 & 1.984 & 0.03641 \\
\hline
4 & -3 & -1.5 & 1.323 & 2 & \var{true} & 1 & 2.037 & 0.9164 \\
4 & -2 & -1 & 1.732 & 2 & \var{true} & 1 & 2.037 & 0.4633 \\
4 & -1 & -0.5 & 1.936 & 2 & \var{true} & 1 & 2.037 & 0.3227 \\
4 & 0 & 0 & 2 & 2 & \var{true} & 1 & 2.037 & 0.9633 \\
4 & 1 & 0.5 & 1.936 & 2 & \var{true} & 1 & 2.037 & 0.3227 \\
4 & 2 & 1 & 1.732 & 2 & \var{true} & 1 & 2.037 & 0.4633 \\
4 & 3 & 1.5 & 1.323 & 2 & \var{true} & 1 & 2.037 & 0.9164 \\
\hline
\end{array}
\end{equation*}

  \caption[Upper bound inclusion $\NAFvalue{\bfeta} \in \tau^{2w-1} V$]{Upper 
    bound inclusion $\NAFvalue{\bfeta} \in \tau^{2w-1} V$
    checked for ``problematic values'' of $\abs{\tau}$ and $w$, cf.\ 
    \itemref{enu:upper-bound:in-v} of 
    Proposition~\vref{pro:upper-bound-fracnafs}. 
    The dependence of $p$, $q$ and $\tau$ is given by $\tau^2-p\tau+q=0$. We 
    have $p^2<4q$, since $\tau$ is assumed to be imaginary quadratic.}
  \label{tab:upper-bound-problematic}
\end{table}


\begin{table}
  \centering
  \begin{equation*}
\begin{array}{cc|cc|c|cccc|c}
\hline
q=\abs{\tau}^2 & p & \re{\tau} & \im{\tau} & w & 
\text{$k$ found?} & k & 2\abs{\tau}^{-k}f_U & \eps & \text{valid for $\ell \leq k$?} \\
\hline
2 & -2 & -1 & 1 & 2 & \var{true} & 10 & 0.1909 & 0.03003 & \var{true} \\
2 & -1 & -0.5 & 1.323 & 2 & \var{true} & 7 & 0.5401 & 0.138 & \var{true} \\
2 & 0 & 0 & 1.414 & 2 & \var{true} & 8 & 0.3819 & 0.1484 & \var{true} \\
2 & 1 & 0.5 & 1.323 & 2 & \var{true} & 7 & 0.5401 & 0.138 & \var{true} \\
2 & 2 & 1 & 1 & 2 & \var{true} & 9 & 0.27 & 0.03933 & \var{true} \\
\hline
2 & -2 & -1 & 1 & 3 & \var{true} & 4 & 1.671 & 0.2737 & \var{true} \\
2 & -1 & -0.5 & 1.323 & 3 & \var{true} & 4 & 1.671 & 0.2682 & \var{true} \\
2 & 0 & 0 & 1.414 & 3 & \var{true} & 4 & 1.671 & 0.0969 & \var{true} \\
2 & 1 & 0.5 & 1.323 & 3 & \var{true} & 4 & 1.671 & 0.2682 & \var{true} \\
2 & 2 & 1 & 1 & 3 & \var{true} & 4 & 1.671 & 0.2737 & \var{true} \\
\hline
3 & -3 & -1.5 & 0.866 & 2 & \var{true} & 3 & 1.323 & 0.5054 & \var{true} \\
3 & -2 & -1 & 1.414 & 2 & \var{true} & 3 & 1.323 & 0.04922 & \var{true} \\
3 & -1 & -0.5 & 1.658 & 2 & \var{true} & 4 & 0.7638 & 0.4729 & \var{true} \\
3 & 0 & 0 & 1.732 & 2 & \var{true} & 3 & 1.323 & 0.5054 & \var{true} \\
3 & 1 & 0.5 & 1.658 & 2 & \var{true} & 4 & 0.7638 & 0.4729 & \var{true} \\
3 & 2 & 1 & 1.414 & 2 & \var{true} & 3 & 1.323 & 0.04922 & \var{true} \\
3 & 3 & 1.5 & 0.866 & 2 & \var{true} & 3 & 1.323 & 0.5054 & \var{true} \\
\hline
4 & -3 & -1.5 & 1.323 & 2 & \var{true} & 2 & 2.037 & 0.9164 & \var{true} \\
4 & -2 & -1 & 1.732 & 2 & \var{true} & 2 & 2.037 & 0.4633 & \var{true} \\
4 & -1 & -0.5 & 1.936 & 2 & \var{true} & 2 & 2.037 & 0.3227 & \var{true} \\
4 & 0 & 0 & 2 & 2 & \var{true} & 2 & 2.037 & 0.9633 & \var{true} \\
4 & 1 & 0.5 & 1.936 & 2 & \var{true} & 2 & 2.037 & 0.3227 & \var{true} \\
4 & 2 & 1 & 1.732 & 2 & \var{true} & 2 & 2.037 & 0.4633 & \var{true} \\
4 & 3 & 1.5 & 1.323 & 2 & \var{true} & 2 & 2.037 & 0.9164 & \var{true} \\
\hline
\end{array}
\end{equation*}

  \caption[Upper bound inclusion $\NAFvalue{\eta_1\ldots\eta_\ell} 
  + \tau^{-\ell}V \subseteq \tau^{2w-1} V$]{Upper bound inclusion 
    $\NAFvalue{\eta_1\ldots\eta_\ell} 
    + \tau^{-\ell}V \subseteq \tau^{2w-1} V$
    checked for ``problematic values'' of $\abs{\tau}$ and $w$, cf.\ 
    \itemref{enu:upper-bound:approx-in-v} of 
    Proposition~\vref{pro:upper-bound-fracnafs}. The
    dependence of $p$, $q$ and $\tau$ is given by $\tau^2-p\tau+q=0$. We have
    $p^2<4q$, since $\tau$ is assumed to be imaginary quadratic.}
  \label{tab:upper-bound-approx-problematic}
\end{table}


We remark that the check of the ``problematic cases'' in the proofs
of~\itemref{enu:upper-bound:in-v} and~\itemref{enu:upper-bound:approx-in-v}
depends on the choice of the digit set $\cD$, cf.\
Remark~\vref{rem:choice-digit-set-voronoi-boundary}. The results in
Tables~\ref{tab:upper-bound-problematic}
and~\ref{tab:upper-bound-approx-problematic} are for the choice of $\cD$ that
corresponds to Definition~\vref{def:restr-voronoi}. Similar values can be found
for all other choices with the exception of the case $p=0$ and $q=2$. In this
case, there are
digit set choices such that the equivalent conditions
in~\itemref{enu:upper-bound:equiv} --- which are used to prove the last two
statements of the proposition --- are not fulfilled. But also in those
exceptional cases it is not too hard to show that~\itemref{enu:upper-bound:in-v}
and~\itemref{enu:upper-bound:approx-in-v} of
Proposition~\vref{pro:upper-bound-fracnafs} still hold.


Next we want to find a lower bound for the fractional value of a
\wNAF{}. Clearly the \wNAF{} $0$ has fractional value $0$, so we are interested
in cases, where we have a non-zero digit somewhere.


\begin{proposition}[Lower Bound for the Fractional Value]
  \label{pro:lower-bound-fracnafs}

  The following is true:
  \begin{enumerate}[(a)]

  \item \label{enu:lower-bound:equiv} The following two statements are
    equivalent:
    \begin{enumequivalences}

    \item \label{enu:proc-lb:1} There is an $\ell\in\N_0$, such that for all
      $\bfxi\in\wNAFsetellell{0}{\ell}$ with non-zero most significant digit the
      condition
      \begin{equation*}
        \abs{\NAFvalue{\bfxi}} > \abs{\tau}^{-\ell} f_U
      \end{equation*}
      is fulfilled. 

    \item \label{enu:proc-lb:2} There exists a $\wt\nu>0$, such that for all
      $\bfvartheta \in \wNAFsetinf$ with non-zero most significant digit the
      condition
      \begin{equation*}
        \abs{\NAFvalue{\bfvartheta}} \geq \abs{\tau}^{-1} \wt\nu.
      \end{equation*}
      holds.

    \end{enumequivalences}

  \item \label{enu:lower-bound:geq-nu} Let $\bfeta\in\wNAFsetinf$ with non-zero
    most significant digit. Then
    \begin{equation*}
      \abs{\NAFvalue{\bfeta}} \geq \abs{\tau}^{-1} f_L
    \end{equation*}
    with $f_L=\nu$ if $\nu>0$, where
    \begin{equation*}
      \label{eq:lower-bound-nu:formula}
      \nu = \frac{1}{2} - \frac{\abs{\tau} c_V}{\abs{\tau}^{w}-1}.
    \end{equation*}
    If $\nu \leq 0$, see Table~\vref{tab:values-of-nu}, then we set $f_L =
    \wt\nu$ from Table~\vref{tab:lower-bounds-problematic}.

  \end{enumerate}
\end{proposition}


\begin{proof}[Proof of Proposition~\ref{pro:lower-bound-fracnafs}]

  \begin{enumerate}[(a)]

  \item We have to prove both directions.
    \begin{descproofequivalences}

    \item[\labelproofequivalent{\ref{enu:proc-lb:1}}{\ref{enu:proc-lb:2}}]

      Suppose there exists such an $\ell\in\N_0$. We set
      \begin{equation*}
        \wt\nu = \min\set*{\abs{\tau} 
          \left( \abs{\NAFvalue{\bfxi}} - \abs{\tau}^{-\ell} f_U 
          \right)}{\text{$\bfxi \in \wNAFsetellell{0}{\ell}$ 
            with $\xi_{-1}\neq0$}}.
      \end{equation*}
      Then clearly $\wt\nu>0$. Using \itemref{enu:upper-bound:leq-f} of
      Proposition~\vref{pro:upper-bound-fracnafs} with digits shifted $\ell$ to
      the right, i.e., multiplication by $\tau^{-\ell}$, the desired result
      follows by using the triangle inequality.

    \item[\labelproofequivalent{\ref{enu:proc-lb:2}}{\ref{enu:proc-lb:1}}]

      Now suppose there exists such a lower bound $\wt\nu>0$. Then there is an
      $\ell\in\N_0$ such that $\abs{\tau}^{-\ell} f_U < \abs\tau^{-1}
      \wt\nu$. Since
      \begin{equation*}
        \abs{\NAFvalue{0\bfldot\eta_{-1}\ldots\eta_{-\ell}}} 
        \geq \abs\tau^{-1} \wt\nu > \abs{\tau}^{-\ell} f_U
      \end{equation*}
      for all \wNAF{}s $0\bfldot\eta_{-1}\ldots\eta_{-\ell}$, the statement
      follows.
    
    \end{descproofequivalences}
    
  \item Set
    \begin{equation*}
      M := \tau^w \NAFvalue{\bfeta} 
      = \eta_{-1} \tau^{w-1} + \sum_{i=2}^{\ell} \eta_{-i} \tau^{w-i}
    \end{equation*}
    Since $\eta_{-1} \neq 0$, we can rewrite this to get
    \begin{equation*}
      M = \eta_{-1} \tau^{w-1} + \sum_{i=w+1}^{\ell} \eta_{-i} \tau^{w-i}
      = \eta_{-1} \tau^{w-1} + \sum_{k=1}^{\ell-w} \eta_{-(w+k)} \tau^{-k}.
    \end{equation*}
    
    Now consider the Voronoi cell $V_{\eta_{-1}}$ for $\eta_{-1}$ and $V_0=V$
    for $0$. Since $\eta_{-1} \neq 0$, these two are disjoint, except parts of
    the boundary, if they are adjacent.
    
    We know from \itemref{enu:upper-bound:in-balls} of
    Proposition~\vref{pro:upper-bound-fracnafs}, that
    \begin{equation*}
      M - \eta_{-1} \tau^{w-1} 
      = \sum_{k=1}^{\ell-w} \eta_{-(w+k)} \tau^{-k}
      \in \bigcup_{z \in \tau^{w-1} V} \ball*{z}{\abs{\tau}^{-w} f_U},
    \end{equation*}
    so
    \begin{equation*}
      M \in \bigcup_{z \in \tau^{w-1} V_{\eta_{-1}}} \ball*{z}{\abs{\tau}^{-w} f_U}.
    \end{equation*}
    This means that $M$ is in $\tau^{w-1} V_{\eta_{-1}}$ or in a
    $\abs{\tau}^{-w} f_U$\nbd-strip around this cell.
    
    Now we are looking at $\tau^{w-1} V_0$ and using
    Proposition~\vref{pro:voronoi-prop}, from which we know that
    $\ball*{0}{\frac{1}{2} \abs{\tau}^{w-1}}$ is inside such a Voronoi
    cell. Thus, we get
    \begin{equation*}
      \abs{M} 
      \geq \frac{1}{2} \abs{\tau}^{w-1} - \abs{\tau}^{-w} f_U 
      = \frac{1}{2} \abs{\tau}^{w-1} - \frac{c_V}{1 - \abs{\tau}^{-w}}
      = \abs{\tau}^{w-1} \nu 
    \end{equation*}
    for our lower bound of $M$ and therefore, by multiplying with $\tau^{-w}$
    one for $\NAFvalue{\bfeta}$.

    Looking in Table~\vref{tab:values-of-nu}, we see that there are some values
    where $\nu$ is not positive. As in
    Proposition~\vref{pro:upper-bound-fracnafs}, this is the case, if $w=2$ and
    $\abs{\tau}$ is $\sqrt{2}$, $\sqrt{3}$ or $\sqrt{4}$, and if $w=3$ and
    $\abs{\tau}=\sqrt{2}$. Since $\nu$ is monotonic increasing with $\abs{\tau}$
    and monotonic increasing with $w$, there are no other non-positive values of
    $\nu$ than the above mentioned.

    For those finite many problem cases, we use~\itemref{enu:lower-bound:equiv}
    to find a $\wt\nu$. The results are listed in
    Table~\vref{tab:lower-bounds-problematic} and an example is drawn in
    Figure~\vref{fig:lower-bound-1}.
    \qedhere
\end{enumerate}
\end{proof}


\begin{table}
  \centering
  \begin{equation*}
\begin{array}{cc|cc|c|ccc|c}
\hline
q=\abs{\tau}^2 & p & \re{\tau} & \im{\tau} & w & 
\ell & \abs{\tau}^{-\ell}f_U & \wt\nu & \log_{\abs\tau}(f_U/\wt\nu) \\
\hline
2 & -2 & -1 & 1 & 2 & 9 & 0.135 & 0.004739 & 18.66 \\
2 & -1 & -0.5 & 1.323 & 2 & 7 & 0.27 & 0.105 & 9.726 \\
2 & 0 & 0 & 1.414 & 2 & 8 & 0.1909 & 0.07422 & 10.73 \\
2 & 1 & 0.5 & 1.323 & 2 & 7 & 0.27 & 0.105 & 9.726 \\
2 & 2 & 1 & 1 & 2 & 9 & 0.135 & 0.04176 & 12.39 \\
\hline
2 & -2 & -1 & 1 & 3 & 6 & 0.4177 & 0.1126 & 9.782 \\
2 & -1 & -0.5 & 1.323 & 3 & 6 & 0.4177 & 0.04999 & 12.13 \\
2 & 0 & 0 & 1.414 & 3 & 6 & 0.4177 & 0.0153 & 15.54 \\
2 & 1 & 0.5 & 1.323 & 3 & 6 & 0.4177 & 0.04999 & 12.13 \\
2 & 2 & 1 & 1 & 3 & 6 & 0.4177 & 0.1126 & 9.782 \\
\hline
3 & -3 & -1.5 & 0.866 & 2 & 4 & 0.3819 & 0.003019 & 12.81 \\
3 & -2 & -1 & 1.414 & 2 & 5 & 0.2205 & 0.04402 & 7.933 \\
3 & -1 & -0.5 & 1.658 & 2 & 5 & 0.2205 & 0.08717 & 6.689 \\
3 & 0 & 0 & 1.732 & 2 & 4 & 0.3819 & 0.003019 & 12.81 \\
3 & 1 & 0.5 & 1.658 & 2 & 5 & 0.2205 & 0.08717 & 6.689 \\
3 & 2 & 1 & 1.414 & 2 & 5 & 0.2205 & 0.04402 & 7.933 \\
3 & 3 & 1.5 & 0.866 & 2 & 4 & 0.3819 & 0.003019 & 12.81 \\
\hline
4 & -3 & -1.5 & 1.323 & 2 & 4 & 0.2546 & 0.07613 & 5.742 \\
4 & -2 & -1 & 1.732 & 2 & 5 & 0.1273 & 0.03807 & 6.742 \\
4 & -1 & -0.5 & 1.936 & 2 & 4 & 0.2546 & 0.0516 & 6.303 \\
4 & 0 & 0 & 2 & 2 & 5 & 0.1273 & 0.07035 & 5.856 \\
4 & 1 & 0.5 & 1.936 & 2 & 4 & 0.2546 & 0.0516 & 6.303 \\
4 & 2 & 1 & 1.732 & 2 & 5 & 0.1273 & 0.0467 & 6.447 \\
4 & 3 & 1.5 & 1.323 & 2 & 4 & 0.2546 & 0.07613 & 5.742 \\
\hline
\end{array}
\end{equation*}

  \caption[Lower bounds for ``problematic values'' of $\abs{\tau}$ and $w$]{
    Lower bounds for ``problematic values'' of $\abs{\tau}$ and $w$, cf.\
    \itemref{enu:lower-bound:geq-nu} of 
    Proposition~\vref{pro:lower-bound-fracnafs}. 
    The dependence of $p$, $q$ and $\tau$ is given by $\tau^2-p\tau+q=0$. We 
    have $p^2<4q$, since $\tau$ is assumed to be imaginary quadratic.}
  \label{tab:lower-bounds-problematic}
\end{table}


\begin{figure}
  \centering
  \includegraphics{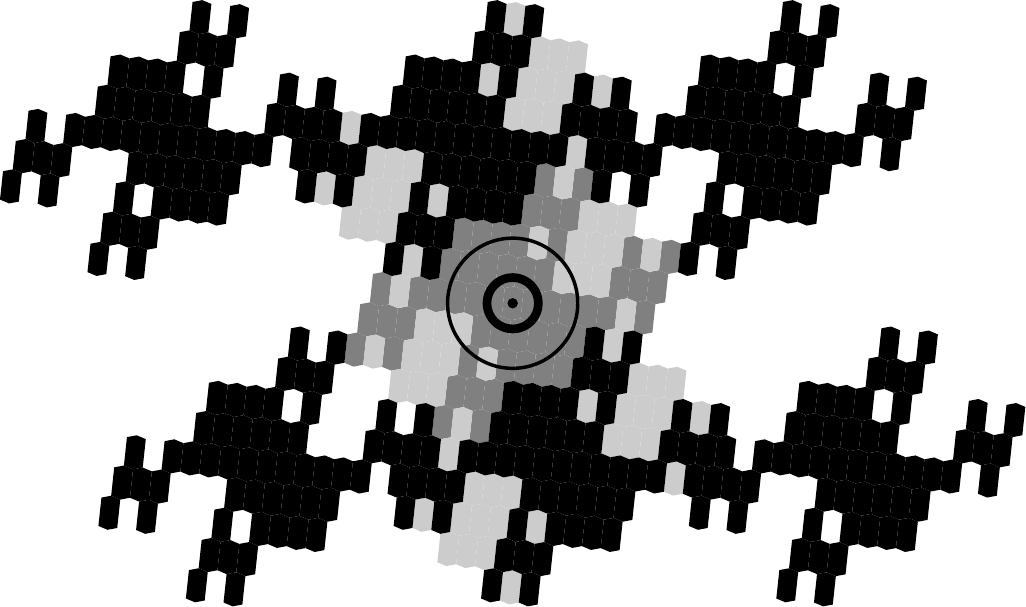}
  \caption[Lower bound for $\tau=\frac12 + \frac{1}{2}\sqrt{-11}$ and
  $w=2$]{Lower bound for $\tau=\frac12 + \frac{1}{2}\sqrt{-11}$ and $w=2$. The
    procedure stopped at $\ell=6$. The large circle has radius
    $\abs{\tau}^{-\ell} f_U$, the small circle is our lower bound with radius
    $f_L=\wt\nu$. The dot inside represents zero. The grey region has most
    significant digit zero, the black ones non-zero.}
  \label{fig:lower-bound-1}
\end{figure}


Again, as in the proof of Proposition~\vref{pro:lower-bound-fracnafs}, the
check of the ``problematic cases'' depends on the choice of the digit set
$\cD$, cf.\ Remark~\vref{rem:choice-digit-set-voronoi-boundary}. And again, the
results stay true for any choice of $\cD$.


Combining the previous two Propositions leads to the following corollary, which
gives an upper and a lower bound for the absolute value of a \wNAF{} by
looking at the largest non-zero index.


\begin{corollary}[Bounds for the Value]
  \label{cor:bounds-value}
  Let $\bfeta\in\wNAFsetfininf$, then we get
  \begin{equation*}
     \NAFd{\bfeta}{\bfzero} f_L
    \leq \abs{\NAFvalue{\bfeta}}
    \leq \NAFd{\bfeta}{\bfzero} f_U \abs\tau.
  \end{equation*}
\end{corollary}


\begin{proof}
  Follows directly from Proposition~\vref{pro:upper-bound-fracnafs}
  and Proposition~\vref{pro:lower-bound-fracnafs}.
\end{proof}


Last in this section, we want to find out, if there are special \wNAF{}s, for
which we know for sure that all their expansions start with a certain finite
\wNAF{}. We will show the following lemma.


\begin{lemma}
  \label{lem:choosing-k0}

  Let
  \begin{equation*}
    k \geq k_0 = \max\set{19,2w+5},
  \end{equation*}
  let $\bfeta\in\wNAFsetinf$ start with the word $0^k$, i.e.,
  $\eta_{-1}=0$, \ldots, $\eta_{-k}=0$, and set $z=\NAFvalue{\bfeta}$.  Then we
  get for all $\bfxi\in\wNAFsetfininf$ that $z=\NAFvalue{\bfxi}$ implies
  $\bfxi\in\wNAFsetellinf{0}$.
\end{lemma}


\begin{proof}
  Let $\bfxi_I\bfldot\bfxi_F\in\wNAFsetfininf$. Then
  $\abs{\NAFvalue{\bfxi_I\bfldot\bfxi_F}} < f_L$ implies $\bfxi_I=\bfzero$,
  cf.\ Proposition~\vref{pro:lower-bound-fracnafs}. For our $\bfeta$ we obtain
  $z=\abs{\NAFvalue{\bfeta}} \leq \abs\tau^{-k} f_U$, cf.\
  Proposition~\vref{pro:upper-bound-fracnafs}. So we have to show that
  \begin{equation*}
    \abs\tau^{-k} f_U < f_L,
  \end{equation*} 
  which is equivalent to
  \begin{equation*}
    k > \log_{\abs\tau} \frac{f_U}{f_L}.
  \end{equation*}

  For the ``non-problematic cases'', cf.\
  Propositions~\vref{pro:upper-bound-fracnafs}
  and~\vref{pro:lower-bound-fracnafs}, we obtain
  \begin{equation*}
    k > 2w-1 + \log_{\abs\tau} A
  \end{equation*}
  with 
  \begin{equation*}
    A := \frac{1}{\nu} \frac{\abs\tau c_V}{\abs\tau^w-1}
    = \left(\frac{\abs\tau^w-1}{2\abs\tau c_V}-1\right)^{-1} > 0,
  \end{equation*}
  where we just inserted the formulas for $f_U$, $f_L$ and $\nu$, and used
  $\nu>0$.

  Consider the partial derivation of $\log_{\abs\tau} A$ with respect to
  $\abs\tau$. We get
  \begin{equation*}
    \frac{\partial \log_{\abs\tau} A}{\partial \abs\tau} 
    = \underbrace{\vphantom{\Bigg\vert}\frac{1}{\log_e \abs\tau}}_{>0}
    \underbrace{\vphantom{\Bigg\vert}\nu}_{>0}
    \underbrace{\vphantom{\Bigg\vert}\frac{\abs\tau^w-1}{\abs\tau c_V}}_{>0}
    \underbrace{\frac{\partial A}{\partial \abs\tau}}_{<0} < 0,
  \end{equation*}
  where we used $\abs\tau>1$, $w\geq2$, and the fact that the quotient of
  polynomials $\frac{\abs\tau^w-1}{2\abs\tau c_V}$ is monotonic increasing with
  $\abs\tau$. Further we see that $A$ is monotonic
  decreasing with $w$, therefore $\log_{\abs\tau} A$, too.

  For $\abs\tau=\sqrt{5}$ and $w=2$ we get $\log_{\abs\tau} A = 5.84522$, for
  $\abs\tau=\sqrt{3}$ and $w=3$ we get $\log_{\abs\tau} A = 1.70649$, and for
  $\abs\tau=\sqrt{2}$ and $w=4$ we get $\log_{\abs\tau} A = 2.57248$. Using the
  monotonicity from above yields $k \geq 2w+5$ for the ``non-problematic
  cases''.

  For our ``problematic cases'', the value of $\log_{\abs\tau} \frac{f_U}{f_L}$
  is calculated in Table~\vref{tab:lower-bounds-problematic}. Therefore we
  obtain $k \geq 19$.
\end{proof}


\section{Numeral Systems with Non-Adjacent Forms}
\label{sec:exist-uniqu-nafs}


Let $\tau\in\C$ be an algebraic integer, imaginary quadratic.  Suppose that
$\abs{\tau}>1$. Let $w\in\N$ with $w\geq2$. Further let $\cD$ be a minimal norm
representatives digit set modulo $\tau^w$ as in
Definition~\vref{def:min-norm-digit-set}.


We are now able to show that in this setting, the digit set of minimal norm
representatives is indeed a width\nbd-$w$ non-adjacent digit set. This is then
extended to infinite fractional expansions of elements in $\C$.


\begin{theorem}[Existence and Uniqueness Theorem concerning Lattice Points]
  \label{thm:wnaf-exist-unique}

  For each lattice point $z\in\Ztau$ there is a unique element
  $\bfxi\in\wNAFsetfin$, such that $z = \NAFvalue{\bfxi}$. Thus $\cD$ is a
  width\nbd-$w$ non-adjacent digit set. The \wNAF{} $\bfxi$ can be calculated
  using Algorithm~\vref{alg:get-wnaf}, i.e., this algorithm terminates and is
  correct.
\end{theorem}

The uniqueness result is well known. The existence result is only known for
special $\tau$ and $w$. For example in Koblitz~\cite{Koblitz:1998:ellip-curve}
the case $\tau=\pm\frac{3}{2} + \frac{1}{2} \sqrt{-3}$ and $w=2$ was
shown. There the digit set $\cD$ consists of $0$ and powers of primitive sixth
roots of unity. Blake, Kumar Murty and Xu~\cite{Blake-Kumar-Xu:2005:effic-algor}
generalised that for $w\geq2$. Another example is given in
Solinas~\cite{Solinas:2000:effic-koblit}. There $\tau=\pm\frac{1}{2} +
\frac{1}{2}\sqrt{-7}$ and $w=2$ is used, and the digit set $\cD$ consists of $0$
and $\pm1$. This result was generalised by Blake, Kumar Murty and
Xu~\cite{Blake-Murty-Xu:2005:naf} for $w\geq2$. The cases $\tau=1+\sqrt{-1}$,
$\tau=\sqrt{-2}$ and $\tau=\frac{1}{2} + \frac{1}{2}\sqrt{-11}$ were studied in
Blake, Kumar Murty and Xu~\cite{Blake-Murty-Xu:ta:nonad-radix}.


\begin{algorithm}
  \begin{algorithmic}[1]
    \STATE $\ell \leftarrow 0$
    \STATE $y \leftarrow z$
    \WHILE{$y \neq 0$}
    \IF{$\tau \divides y$}
    \STATE $\xi_\ell \leftarrow 0$
    \ELSE
    \STATE Let $\xi_\ell \in \cD$ such that $\xi_\ell \equiv y \pmod{\tau^w}$
    \ENDIF
    \STATE $y \leftarrow \left(y-\xi_\ell\right) / \tau$
    \STATE $\ell \leftarrow \ell+1$
    \ENDWHILE
    \STATE $\bfxi \leftarrow \xi_{\ell-1}\xi_{\ell-2}\ldots\xi_0$
    \RETURN $\bfxi$
  \end{algorithmic}  

  \caption{Algorithm to calculate a \wNAF{} $\bfxi\in\wNAFsetfin$ for an element
    $z\in\Ztau$.}
  \label{alg:get-wnaf}
\end{algorithm}


The proof follows a similar idea as in Section~\vref{sec:bounds-value} and in
Matula~\cite{Matula:1982:basic}. There are again two parts, one analytic part
for all but finitely many cases, and the other, which proves the remaining by
the help of a simple procedure.


\begin{proof}  
  First we show that the algorithm terminates. Let $y\in\Ztau$ and consider
  Algorithm~\vref{alg:get-wnaf} in cycle $\ell$. If $\tau \divides y$, then in
  the next step the norm $\abs{y}^2\in\N_0$ becomes smaller since
  $\abs\tau^2>1$.
  
  Let $\tau \ndivides y$. If $\abs{y} < \frac{1}{2} \abs{\tau}^w$, then
  $y\in\cD$, cf.\ Proposition~\vref{pro:voronoi-prop} and
  Remark~\vref{rem:digitsets}. Thus the algorithm terminates in the next
  cycle. If
  \begin{equation*}
    \abs{y} > \frac{\abs{\tau} c_V}{1-\abs{\tau}^{-w}} 
    = \frac{\abs{\tau}^{w+1} c_V}{\abs{\tau}^w-1},
  \end{equation*}
  we obtain
  \begin{equation*}
    \abs{y} \abs{\tau}^w
    > \abs{y} + \abs{\tau}^{w+1} c_V
    \geq \abs{y} + \abs{\xi_\ell} 
    \geq \abs{y-\xi_\ell},
  \end{equation*}
  which is equivalent to
  \begin{equation*}
    \abs{\frac{y-\xi_\ell}{\tau^w}} < \abs{y}.
  \end{equation*}
  So if the condition
  \begin{equation*}
    \frac{\abs{\tau}^{w+1} c_V}{\abs{\tau}^w-1} < \frac{1}{2} \abs{\tau}^w
    \equivalent
    \nu = \frac{1}{2} - \frac{\abs{\tau}c_V}{\abs{\tau}^w-1} > 0,
  \end{equation*}
  with the same $\nu$ as in Proposition~\vref{pro:lower-bound-fracnafs}, is
  fulfilled, the norm $\abs{y}^2\in\N_0$ is descending and therefore the
  algorithm terminating.

  Now we consider the case, when $\nu\leq0$. According to
  Table~\vref{tab:values-of-nu} there are the same finitely many combinations of
  $\tau$ and $w$ to check as in Proposition~\vref{pro:upper-bound-fracnafs} and
  Proposition~\vref{pro:lower-bound-fracnafs}. For each of them, there is only a
  finite number of elements $y\in\Ztau$ with
  \begin{equation*}
    \abs{y} \leq \frac{\abs{\tau}^{w+1} c_V}{\abs{\tau}^w-1},
  \end{equation*}
  so altogether only finitely many $y\in\Ztau$ left to check, whether they
  admit a \wNAF{} or not.  \ifma The results can be found in the table in
  Appendix~\ref{cha:tab:existence-wnafs}.  \else The results can be found in
  the table available as online-resource\footnote{Table available at
    \href{http://www.danielkrenn.at/wnaf-analysis}{\texttt{www.danielkrenn.at/wnaf-analysis}}.}. \fi
  Every element that was to check, has a \wNAF{}.

  To show the correctness, again let $y\in\Ztau$ and consider
  Algorithm~\vref{alg:get-wnaf} in cycle $\ell$. If $\tau$ divides $y$, then we
  append a digit $\xi_\ell=0$. Otherwise $y$ is congruent to a non-zero element
  $\xi_\ell$ of $\cD$ modulo $\tau^w$, since the digit set~$\cD$ was constructed
  in that way, cf.\ Definitions~\ref{def:red-residue-digit-set}
  and~\ref{def:min-norm-digit-set}. The digit $\xi_\ell$ is appended. Because
  $\tau^w$ divides $y-\xi_\ell$, the next $w-1$ digits will be zero. Therefore a
  correct \wNAF{} is produced.

  For the uniqueness let $\bfxi\in\wNAFsetfin$ be an expansions for the element
  $z\in\Ztau$. If $\tau \divides z$, then
  \begin{equation*}
    0 \equiv z = \NAFvalue{\bfxi} \equiv \xi_0 \pmod{\tau},
  \end{equation*}
  so $\tau \divides \xi_0 \in \cD$. Therefore $\xi_0=0$. If $\tau \ndivides z$,
  then $\tau \ndivides \xi_0$ and so $\xi_0\neq0$. This implies $\xi_1=0$,
  \dots, $\xi_{w-1}=0$. This means $\xi_0$ lies in the same residue class modulo
  $\tau^w$ as exactly one non-zero digit of $\cD$ (per construction of the digit
  set, cf.\ Definitions~\ref{def:red-residue-digit-set}
  and~\ref{def:min-norm-digit-set}), hence they are equal. Induction finishes
  the proof of the uniqueness.
\end{proof}


The existence check of the ``problematic cases'' depends on the choice of the
digit set $\cD$, cf.\ Remark~\vref{rem:choice-digit-set-voronoi-boundary}, but
for all possible choices the result stays true.


So we get that all elements of our lattice $\Ztau$ have a unique expansion. Now
we want to get a step further and look at all elements of $\C$. We will need
the following three lemmata, to prove that every element of $\C$ has a
\wNAF{}-expansion.


\begin{lemma}
  \label{lem:value-injective}
  
  The function $\NAFvaluename\restricted{\wNAFsetfinfin}$ is injective.
\end{lemma}


\begin{proof}
  Let $\bfeta$ and $\bfxi$ be elements of $\wNAFsetfinfin$ with
  $\NAFvalue{\bfeta} = \NAFvalue{\bfxi}$. This implies that $\tau^J
  \NAFvalue{\bfeta} = \tau^J \NAFvalue{\bfxi} \in \Ztau$ for some $J\in\Z$. By
  uniqueness of the integer \wNAF{}s, see Theorem~\vref{thm:wnaf-exist-unique},
  we conclude that $\bfeta=\bfxi$.
\end{proof}


\begin{lemma}
  \label{lem:Zinvtau-is-wnaffinfin}
  
  We have $\NAFvalue{\wNAFsetfinfin} = \Zinvtau$.
\end{lemma}


\begin{proof}
  Let $\bfeta\in\wNAFsetfinfin$ and $\eta_j=0$ for all $\abs{j}>J$ for some
  $J \geq 1$. Then $\tau^J\NAFvalue{\bfeta}\in\Ztau$, which implies that there
  are some $a$, $b \in\Z$ such that
  \begin{equation*}
    \NAFvalue{\bfeta}=a \tau^{-(J-1)}+b\tau^{-J}\in\Zinvtau.
  \end{equation*}
  Conversely, if
  \begin{equation*}
    z=\sum_{j=0}^J \eta_j \tau^{-j}\in\Zinvtau,
  \end{equation*}
  we have $\tau^J z \in\Ztau$. Since every element of $\Ztau$ admits an integer
  \wNAF{}, see Theorem~\vref{thm:wnaf-exist-unique}, there is an
  $\bfxi\in\wNAFsetfinfin$ with $\NAFvalue{\bfxi}=z$.
\end{proof}


\begin{lemma}
  \label{lem:Zinvtau-dense}
  
  $\Zinvtau$ is dense in $\C$.
\end{lemma}


\begin{proof}
  Let $z\in\C$ and $K\geq 0$. Then $\tau^K z=u+v\tau$ for some reals $u$ and
  $v$. We have
  \begin{equation*}
    \abs{z-\frac{\floor u+\floor v \tau}{\tau^K}}
    < \frac{1+\abs\tau}{\abs\tau^K},
  \end{equation*}
  which proves the lemma.
\end{proof}


Now we can prove the following theorem.


\begin{theorem}[Existence Theorem concerning $\C$]
  \label{thm:C-has-wnaf-exp}
  
  Let $z\in\C$. Then there is an $\bfeta\in\wNAFsetfininf$ such that
  $z=\NAFvalue{\bfeta}$, i.e., each complex number has a \wNAF{}-expansion.
\end{theorem}


\begin{proof}
  By Lemma~\vref{lem:Zinvtau-dense}, there is a sequence $z_n\in\Zinvtau$
  converging to $z$. By Lemma~\vref{lem:Zinvtau-is-wnaffinfin}, there is a
  sequence $\bfeta_n\in\wNAFsetfinfin$ with $\NAFvalue{\bfeta_n}=z_n$ for all
  $n$. By Corollary~\vref{cor:bounds-value} the sequence $\NAFd{\bfeta_n}{0}$ is
  bounded from above, so there is an $\ell$ such that
  $\bfeta_n\in\wNAFsetellfin{\ell} \subseteq \wNAFsetellinf{\ell}$. By
  Proposition~\vref{pro:naf-compact}, we conclude that there is a convergent
  subsequence $\bfeta'_n$ of $\bfeta_n$. Set
  $\bfeta:=\lim_{n\to\infty}{\bfeta'_n}$. By continuity of $\NAFvaluename$, see
  Proposition~\vref{pro:value-continuous}, we conclude that
  $\NAFvalue{\bfeta}=z$.
\end{proof}


\section{The Fundamental Domain}
\label{sec:fundamental-domain}


Let $\tau\in\C$ be an algebraic integer, imaginary quadratic.  Suppose that
$\abs{\tau}>1$. Let $w\in\N$ with $w\geq2$. Further let $\cD$ be a minimal norm
representatives digit set modulo $\tau^w$ as in
Definition~\vref{def:min-norm-digit-set}.

We now derive properties of the \emph{Fundamental Domain}, i.e., the set of
complex numbers representable by \wNAF{}s which vanish left of the
\taupoint{}. The boundary of the fundamental domain is shown to correspond to
complex numbers which admit more than one \wNAF{} differing left of the
\taupoint{}. Finally, an upper bound for the Hausdorff dimension of the boundary
is derived.


\begin{definition}[Fundamental Domain]
  \label{def:fund-domain}

  The set
  \begin{equation*}
    \cF := \NAFvalue{\wNAFsetinf} 
    = \set*{\NAFvalue{\bfxi}}{\bfxi\in\wNAFsetinf}.
  \end{equation*}
  is called \emph{fundamental domain}.
\end{definition}


The pictures in Figure~\vref{fig:w-weta} can also be reinterpreted as
fundamental domains for the $\tau$ and $w$ given there. The definition of the
fundamental domain for a general $\tau\in\C$ and a general finite digit set
containing zero is meaningful, too. The same is true for following proposition,
which is also valid for general $\tau\in\C$ and a general finite digit set
including zero.


\begin{proposition}
  \label{pro:fund-domain-compact}

  The fundamental domain $\cF$ is compact.
\end{proposition}

\begin{proof}
  The set $\wNAFsetinf$ is compact, cf.\ Proposition~\vref{pro:naf-compact}. The
  compactness of the fundamental domain $\cF$ follows, since $\cF$ is the image
  of $\wNAFsetinf$ under the continuous function $\NAFvaluename$, cf.\
  Proposition~\vref{pro:value-continuous}.
\end{proof}

We can also compute the Lebesgue measure of the fundamental domain. This result
can be found in Remark~\vref{rem:meas-fund-domain}. We will need the results of
Sections~\ref{sec:cell-rounding-op} and~\ref{sec:sets-w_eta} for
calculating $\lmeas{\cF}$.


Next we want to get more properties of the fundamental domain. We will need the
following proposition, which will be extended in
Proposition~\vref{pro:char-boundary}.
 

\begin{proposition}
  \label{pro:char-boundary:part1}

  Let $z\in\cF$. If there exists a \wNAF{}
  $\bfxi_I\bfldot\bfxi_F\in\wNAFsetfininf$ with $\bfxi_I\neq\bfzero$ and such
  that $z=\NAFvalue{\bfxi_I\bfldot\bfxi_F}$, then $z\in\boundary*{\cF}$.
\end{proposition}


\begin{proof}
  Assume that $z\in\interior*{\cF}$. Then there is an $\eps_z>0$ such that
  $\ballo{z}{\eps_z} \subseteq \interior*{\cF}$. Let $\eps_z$ be small enough
  such that there exists a $y\in\ball{z}{\eps_z} \cap \Zinvtau$ and a
  $\bfvartheta=\bfvartheta_I\bfldot\bfvartheta_F \in\wNAFsetfinfin$ with
  $y=\NAFvalue{\bfvartheta}$ and such that $\bfvartheta_I\neq\bfzero$. Let $k$
  be the right-length of $\bfvartheta$.

  Choose $0<\varepsilon_y<\tau^{-k-w}f_L$ such that $\ballo{y}{\eps_y} \subseteq
  \interior*{\cF}$.  Since $y\in\cF$ there is an $\bfeta\in\wNAFsetinf$ with
  $y=\NAFvalue{\bfeta}$. Therefore, there is a $y'\in\ballo{y}{\eps_y}\cap
  \Zinvtau$ with $y'=\NAFvalue{\bfeta'}$ for some
  $\bfeta'=\bfeta_I'\bfldot\bfeta_F' \in\wNAFsetfinfin$ with $\bfeta_I'=\bfzero$
  (by ``cutting'' the infinite right side of $\bfeta$).
  
  As $y'-y \in \Zinvtau$, there is a $\bfxi\in\wNAFsetfinfin$ with
  $\NAFvalue{\bfxi}=y'-y$. By Corollary~\ref{cor:bounds-value} we obtain
  $\NAFd{\bfxi}{\bfzero} < \eps_y/f_L<\tau^{-k-w}$. Thus $\xi_\ell=0$ for all
  $\ell \geq -k-(w-1)$.

  Now $y'=y+(y'-y)$ and we get a
  $\bfvartheta'=\bfvartheta_I'\bfldot\bfvartheta_F' \in\wNAFsetfinfin$ with
  $\NAFvalue{\bfvartheta'}=y'$ by digit-wise addition of $\bfvartheta$ and
  $\bfxi$. Note that at each index at most one summand (digit) is non-zero and
  that the \wNAF{}-condition is fulfilled. We have $\bfvartheta_I'\neq\bfzero$,
  since $\bfvartheta_I\neq\bfzero$.

  So we got two different \wNAF{}s in $\wNAFsetfinfin$ for one
  element $y'\in \Zinvtau$, which is impossible due to uniqueness, see
  Lemma~\vref{lem:value-injective}. Thus we have a contradiction.
\end{proof}


The complex plane has a tiling property with respect to the fundamental
domain. This fact is stated in the following corollary to
Theorem~\vref{thm:wnaf-exist-unique} and Theorem~\vref{thm:C-has-wnaf-exp}.

\begin{corollary}[Tiling Property]
  \label{cor:complex-plane-tiling}

  The complex plane can be tiled with scaled versions of the fundamental domain
  $\cF$. Only finitely many different size are needed. More precisely: Let
  $K\in\Z$, then
  \begin{equation*}
    \C = 
    \bigcup_{\substack{k\in\set{K,K+1,\dots,K+w-1} \\ \bfxi \in \wNAFsetfin
      \\ \text{$k\neq K+w-1$ implies $\bfxi_0\neq0$}}}
    \left( \tau^k \NAFvalue{\bfxi} + \tau^{k-w+1} \cF \right),
  \end{equation*}
  and the intersection of two different $\tau^k \NAFvalue{\bfxi} + \tau^{k-w+1}
  \cF$ and $\tau^{k'-w+1} \NAFvalue{\bfxi'} + \tau^{k'} \cF$ in this union is a
  subset of the intersection of their boundaries.
\end{corollary}


Later, after Proposition~\vref{pro:boundary-fund-dom-dim-upper}, we will know
that the intersection of the two different sets of the tiling in the
previous corollary has Lebesgue measure $0$.


\begin{proof}[Proof of Corollary~\ref{cor:complex-plane-tiling}]
  Let $z\in\C$. Then, according to Theorem~\vref{thm:C-has-wnaf-exp}, there is a
  $\bfeta\in\wNAFsetfininf$ with $z=\NAFvalue{\bfeta}$. We look at the block
  $\eta_{K+w-1}\ldots\eta_{K+1}\eta_K$. If this block is $\bfzero$, then set
  $k=K+w-1$, otherwise there is at most one non-zero digit in it, which we call
  $\eta_k$. So the digits $\eta_{k-1},\dots,\eta_{k-w+1}$ are always zero. We
  set $\bfxi=\ldots\eta_{k+1}\eta_k\bfldot0\ldots$, and we obtain
  \begin{equation*}
    z - \tau^k \NAFvalue{\bfxi} \in \tau^{k-w+1} \cF.
  \end{equation*}
  
  Now set $F = \tau^k \NAFvalue{\bfxi} + \tau^{k-w+1} \cF$ and $F' = \tau^{k'}
  \NAFvalue{\bfxi'} + \tau^{k'-w+1} \cF$ in a way that both are in the union of
  the tiling with $(k,\bfxi)\neq(k',\bfxi')$ and consider their intersection
  $I$. Since every point in there has two different representations per
  construction, we conclude that $I\subseteq \boundary*{F}$ and $I\subseteq
  \boundary*{F'}$ by Proposition~\vref{pro:char-boundary:part1}.
\end{proof}


\begin{remark}[Iterated Function System]
  \label{rem:ifs}

  Let $\tau\in\C$ and $\cD$ be a general finite digit set containing zero. We
  have two possibilities building the elements $\bfxi\in\wNAFsetinf$ from left
  to right. We can either append $0$, what corresponds to a division through
  $\tau$, so we define $\f{f_0}{z} = \frac{z}{\tau}$. Or we can append a
  non-zero digit $\vartheta\in\cD^\bullet$ and then add $w-1$ zeros. In this
  case, we define $\f{f_\vartheta}{z} = \frac{\vartheta}{\tau} +
  \frac{z}{\tau^w}$. Thus we get the \emph{iterated function system}
  $\ifs{f_\vartheta}{\vartheta\in\cD}$, cf.\ Edgar~\cite{Edgar:2008:measur} or
  Barnsley~\cite{Barnsley:1988:fractals-ew}. All $f_\vartheta$ are
  \emph{similarities}, and the iterated function system realizes the
  \emph{ratio list} $\ral{r_\vartheta}{\vartheta\in\cD}$ with
  $r_0=\abs{\tau}^{-1}$ and for $\vartheta\in\cD^\bullet$ with
  $r_\vartheta=\abs{\tau}^{-w}$, i.e., we have \emph{contracting
    similarities}. So our set can be rewritten as
  \begin{equation*}
    \cF = \bigcup_{\vartheta\in\cD} \f{f_\vartheta}{\cF}
    = \frac{1}{\tau} \cF \cup \bigcup_{\vartheta\in\cD^\bullet} 
    \left(  \frac{\vartheta}{\tau} + \frac{1}{\tau^w} \cF \right).
  \end{equation*}

  Furthermore, if we have an imaginary quadratic algebraic integer $\tau$ and a
  minimal norm representatives digit set, the iterated function system
  $\ifs{f_\vartheta}{\vartheta\in\cD}$ fulfils \emph{Moran's open set
    condition}\footnote{``Moran's open set condition'' is sometimes just called
    ``open set condition''}, cf.\ Edgar~\cite{Edgar:2008:measur} or
  Barnsley~\cite{Barnsley:1988:fractals-ew}. The \emph{Moran open set} used is
  $\interior*{\cF}$. This set satisfies
  \begin{equation*}
    \f{f_\vartheta}{\interior*{\cF}}\cap\f{f_{\vartheta'}}{\interior*{\cF}}
    =\emptyset
  \end{equation*}
  for $\vartheta \neq \vartheta'\in\cD$ and
  \begin{equation*}
    \interior*{\cF} \supseteq \f{f_\vartheta}{\interior*{\cF}}
  \end{equation*}
  for all $\vartheta\in\cD$. We remark that the first condition follows directly
  from the tiling property in Corollary~\vref{cor:complex-plane-tiling} with
  $K=-1$. The second condition follows from the fact that $f_\vartheta$ is an
  open mapping.
\end{remark}

Now we want to have a look at a special case. 

\begin{remark}[Koch snowflake]
  Let $\tau=\frac{3}{2} + \frac{1}{2}\sqrt{-3}$ and $w=2$. Then our digit set
  consists of $0$ and powers of primitive sixth roots of unity, i.e., $\cD =
  \set{0} \cup \set*{\zeta^k}{\text{$k\in\N_0$ with $0 \leq k < 6$}}$ with
  $\zeta = e^{i\pi/3}$, cf.\ Koblitz~\cite{Koblitz:1998:ellip-curve}.

  We get
  \begin{equation*}
    \cF = \frac{1}{\tau} \cF \cup \bigcup_{0 \leq k < 6}
    \left(  \frac{\zeta^k}{\tau} + \frac{1}{\tau^2} \cF \right).
  \end{equation*}
  Since the digit set is invariant with respect to multiplication by $\zeta$,
  i.e., rotation by $\frac{\pi}{3}$, the same is true for $\cF$.  Using this and
  $\tau = \sqrt{3} e^{i\pi/6}$ yields
  \begin{equation*}
    \cF = \frac{e^{i\pi/2}}{\sqrt{3}} \cF \cup \bigcup_{0 \leq k < 6}
    \left(  \frac{e^{ik\pi/3+i\pi/2}}{\sqrt{3}} + \frac{1}{3} \cF \right).
  \end{equation*}
  This is an iterated function system of the \emph{Koch snowflake}\footnote{The
    fact that the Koch snowflake has the mentioned iterated function system
    seems to be commonly known, although we were not able to find a reference,
    where this statement is proved. Any hints are welcome.}, it is drawn in
  Figure~\vref{fig:w-weta:3}.
\end{remark}


Next we want to have a look at the Hausdorff dimension of the boundary of
$\cF$. We will need the following characterisation of the boundary, which is an
extension to Proposition~\vref{pro:char-boundary:part1}. 

\begin{proposition}[Characterisation of the Boundary]
  \label{pro:char-boundary}

  Let $z\in\cF$. Then $z\in\boundary*{\cF}$ if and only if there exists a
  \wNAF{} $\bfxi_I\bfldot\bfxi_F\in\wNAFsetfininf$ with $\bfxi_I\neq\bfzero$,
  such that $z=\NAFvalue{\bfxi_I\bfldot\bfxi_F}$.
\end{proposition}


\begin{proof}
  Let $z\in\boundary*{\cF}$. For every $\eps>0$, there is a
  $y\in\ball{z}{\eps}$, such that $y\not\in\cF$. Thus we have a sequence
  $\sequence{y_j}_{j \geq 1}$ converging to $z$, where the $y_j$ are not in
  $\cF$. Therefore each $y_j$ has a \wNAF{}-representation
  $\bfeta_j\in\wNAFsetfininf$ with non-zero integer part.  Now we will use the
  tiling property stated in Corollary~\vref{cor:complex-plane-tiling}. The
  fundamental domain $\cF$ can be surrounded by only finitely many scaled
  versions of $\cF$. So there is a subsequence
  $\sequence{\bfvartheta_j}_{j\in\N_0}$ of $\sequence{\bfeta_j}_{j\in\N_0}$ with
  fixed integer part $\bfxi_I\neq\bfzero$. Due to compactness of $\cF$, cf.\
  Proposition~\vref{pro:fund-domain-compact}, we find a
  $\bfxi=\bfxi_I\bfldot\bfxi_F$ with value~$z$ as limit of a converging
  subsequence of $\sequence{\bfvartheta_j}_{j\in\N_0}$.

  The other direction is just Proposition~\vref{pro:char-boundary:part1}, thus
  the proof is finished.
\end{proof}


The following proposition deals with the Hausdorff dimension of the boundary of
$\cF$.


\begin{proposition}
  \label{pro:boundary-fund-dom-dim-upper}

  For the Hausdorff dimension of the boundary of the fundamental domain we get
  $\hddim \boundary*{\cF} < 2$.
\end{proposition}


The idea of this proof is similar to a proof in Heuberger and
Prodinger~\cite{Heuberger-Prodinger:2006:analy-alter}.


\begin{proof}
  Set $k:=k_0+w-1$ with $k_0$ from Lemma~\vref{lem:choosing-k0}. For $j\in\N$
  define
  \begin{equation*}
    U_j := \set*{\bfxi\in\wNAFsetellell{0}{j}}{
      \text{$\xi_{-\ell}\xi_{-(\ell+1)}\ldots\xi_{-(\ell+k-1)} \neq 0^k$ 
        for all $\ell$ with $1 \leq \ell \leq j-k+1$}}.
  \end{equation*}
  The elements of $U_j$ --- more precisely the digits from index $-1$ to
  $-j$ --- can be described by the regular expression
  \begin{equation*}
    \left( \eps + \sum_{d\in\cD^\bullet} \sum_{\ell=0}^{w-2} 0^\ell d \right)
    \left( \sum_{d\in\cD^\bullet} \sum_{\ell=w-1}^{k-1} 0^\ell d \right)^\bfast
    \left( \sum_{\ell=0}^{k-1} 0^\ell \right).
  \end{equation*}
  This can be translated to the generating function
  \begin{equation*}
    \f{G}{Z} = \sum_{j\in\N} \card*{U_j} Z^j
    = \left( 1 + \card*{\cD^\bullet} \sum_{\ell=0}^{w-2} Z^{\ell+1} \right)
    \frac{1}{1 - \card*{\cD^\bullet} \sum_{\ell=w-1}^{k-1} Z^{\ell+1}}
    \left( \sum_{\ell=0}^{k-1} Z^\ell \right)
  \end{equation*}
  used for counting the number of elements in $U_j$. Rewriting yields
  \begin{equation*}
    \f{G}{Z} = \frac{1-Z^k}{1-Z} 
    \frac{1 + (\card*{\cD^\bullet} - 1) Z - \card*{\cD^\bullet} Z^w}{
      1 - Z - \card*{\cD^\bullet} Z^w + \card*{\cD^\bullet} Z^{k+1}},
  \end{equation*}
  and we set
  \begin{equation*}
    \f{q}{Z} := 1 - Z - \card*{\cD^\bullet} Z^w + \card*{\cD^\bullet} Z^{k+1}.
  \end{equation*}

  \begin{figure}
    \centering
    \includegraphics{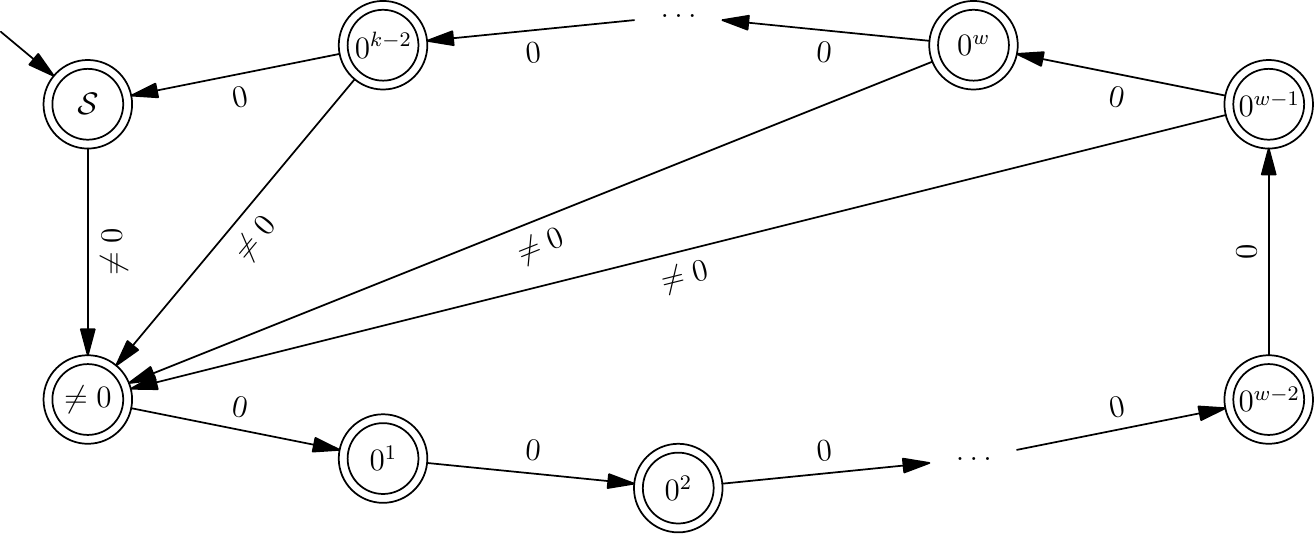}
    \caption[Automaton recognising $\bigcup_{j\in\N} \wt{U}_j$]{Automaton $\cA$
      recognising $\bigcup_{j\in\N} \wt{U}_j$ from right to left, see proof of
      Proposition~\ref{pro:boundary-fund-dom-dim-upper}. The state $\cS$ is the
      starting state, all states are valid end states. An edges marked with
      $\neq0$ means one edge for each non-zero digit in the digit set $\cD$. The
      state $\neq0$ means that there was an non-zero digit read, a state
      $0^\ell$ means that $\ell$ zeros have been read.}
    \label{fig:uj-automata}
  \end{figure}
  
  Now we define
  \begin{equation*}
    \wt{U}_j := \set*{\bfxi \in U_j}{\xi_{-j} \neq 0}
  \end{equation*}
  and consider $\wt{U} := \bigcup_{j\in\N} \wt{U}_j$. The \wNAF{}s in this set
  --- more precisely the finite strings from index $-1$ to the index of the
  largest non-zero digit --- will be recognised by the automaton $\cA$ which
  reads its input from right to left, see Figure~\vref{fig:uj-automata}. It is
  easy to see that the underlying directed graph $G_\cA$ of the automaton $\cA$
  is strongly connected, therefore its adjacency matrix $M_\cA$ is
  irreducible. Since there are cycles of length $w$ and $w+1$ in the graph and
  $\f{\gcd}{w,w+1}=1$, the adjacency matrix is primitive. Thus, using the
  Perron-Frobenius theorem we obtain
  \begin{align*}
    \card*{\wt{U}_j} &= \card{\text{walks in $G_\cA$ of length $j$ from starting
        state $\cS$ to some other state}} \\
    &= \begin{pmatrix} 1 & 0 & \dots & 0 \end{pmatrix}
    M_\cA^j
    \begin{pmatrix} 1 \\ \vdots \\ 1 \end{pmatrix}
    = \wt{c} \left(\sigma \abs\tau^2 \right)^j \left(1+\Oh{s^j}\right)
  \end{align*}
  for a $\wt{c}>0$, a $\sigma>0$, and an $s$ with $0 \leq s < 1$. Since the
  number of \wNAF{}s of length $j$ is $\Oh{\abs\tau^{2j}}$, see
  Theorem~\vref{th:w-naf-distribution}, we get $\sigma\leq1$.
  
  We clearly have
  \begin{equation*}
    U_j = \biguplus_{\ell=j-k+1}^j \wt{U}_\ell,
  \end{equation*}
  so we get 
  \begin{equation*}
    \card*{U_j} = \coefficient{Z^j} \f{G}{Z}
    = c \left(\sigma\abs\tau^2\right)^j \left( 1 + \Oh{s^j} \right)
  \end{equation*}
  for some constant $c>0$.
 
  To rule out $\sigma=1$, we insert the ``zero'' $\abs\tau^{-2}$ in
  $\f{q}{Z}$. We obtain
  \begin{align*}
    \f{q}{\abs\tau^{-2}} 
    &= 1 - \abs\tau^{-2} - \card*{\cD^\bullet} \abs\tau^{-2w} 
    + \card*{\cD^\bullet} \abs\tau^{-2(k+1)} \\
    &= 1 - \abs\tau^{-2} 
    - \abs\tau^{2(w-1)}\left(\abs\tau^2-1\right) \abs\tau^{-2w} 
    + \abs\tau^{2(w-1)}\left(\abs\tau^2-1\right) \abs\tau^{-2(k+1)} \\
    &= \left(\abs\tau^2-1\right) \abs\tau^{2(w-k-2)} > 0,
  \end{align*}
  where we used the cardinality of $\cD^\bullet$ from
  Lemma~\vref{lem:complete-res-sys} and $\abs\tau>1$. Therefore we get
  $\sigma<1$. 

  Define
  \begin{equation*}
    U := \set*{\NAFvalue{\bfxi}}{
      \text{$\bfxi\in\wNAFsetinf$ with 
        $\xi_{-\ell}\xi_{-(\ell+1)}\ldots\xi_{-(\ell+k-1)} \neq 0^k$ 
        for all $\ell \geq 1$}}.
  \end{equation*}
  We want to cover $U$ with squares. Let $S$ be the closed paraxial square with
  centre $0$ and width $2$. Using Proposition~\vref{pro:upper-bound-fracnafs}
  yields
  \begin{equation*}
    U \subseteq 
    \bigcup_{z\in\NAFvalue{U_j}} \left(z + f_U\abs\tau^{-j} S\right)
  \end{equation*}
  for all $j\in\N$, i.e., $U$ can be covered with $\card*{U_j}$ boxes of size
  $2f_U\abs\tau^{-j}$. Thus we get for the upper box dimension, cf.\
  Edgar~\cite{Edgar:2008:measur},
  \begin{equation*}
     \uboxdim U \leq \lim_{j\to\infty} 
     \frac{\log \card*{U_j}}{- \log (2f_U\abs\tau^{-j})}.
  \end{equation*}
  Inserting the cardinality $\card*{U_j}$ from above, using the logarithm to
  base $\abs\tau$ and $0 \leq s < 1$ yields
  \begin{equation*}
    \uboxdim U \leq \lim_{j\to\infty}  
    \frac{\log_{\abs\tau} c + j \log_{\abs\tau} (\sigma\abs\tau^2)
      + \log_{\abs\tau} (1 + \Oh{s^j}) }{j + \Oh{1}}
    = 2 + \log_{\abs\tau} \sigma.
  \end{equation*}
  Since $\sigma<1$, we get $\uboxdim U < 2$.

  Now we will show that $\boundary*{\cF} \subseteq U$. Clearly $U \subseteq
  \cF$, so the previous inclusion is equivalent to $\cF \setminus U \subseteq
  \interior{\cF}$. So let $z \in \cF \setminus U$. Then there is a
  $\bfxi\in\wNAFsetinf$ such that $z=\NAFvalue{\bfxi}$ and $\bfxi$ has a block
  of at least $k$ zeros somewhere on the right hand side of the \taupoint{}.
  Let $\ell$ denote the starting index of this block, i.e.,
  \begin{equation*}
    \bfxi=0\bfldot\underbrace{\xi_{-1}\ldots\xi_{-(\ell-1)}}_{=: \bfxi_A}
    0^k\xi_{-(\ell+k)}\xi_{-(\ell+k+1)}\ldots.
  \end{equation*}
  Let $\bfvartheta =
  \bfvartheta_I\bfldot\bfvartheta_A\vartheta_{-\ell}\vartheta_{-(\ell+1)}\ldots
  \in \wNAFsetfininf$ with $\NAFvalue{\bfvartheta}=z$. We have
  \begin{equation*}
    z = \NAFvalue{0\bfldot\bfxi_A} + \tau^{-\ell-w} z_\xi
    = \NAFvalue{\bfvartheta_I\bfldot\bfvartheta_A} 
    + \tau^{-\ell-w} z_\vartheta
  \end{equation*}
  for appropriate $z_\xi$ and $z_\vartheta$. By Lemma~\vref{lem:choosing-k0},
  all expansions of $z_\xi$ are in $\wNAFsetinf$. Thus all expansions of
  \begin{equation*}
    \NAFvalue{\bfvartheta_I\bfvartheta_A}
    + \tau^{-(w-1)} z_\vartheta
    - \NAFvalue{\bfxi_A} 
    = \tau^{\ell-1} z - \NAFvalue{\bfxi_A}
    = \tau^{-(w-1)} z_\xi
  \end{equation*}
  start with $0.0^{w-1}$, since our choice of $k$ is $k_0+w-1$. As the unique
  NAF of $\NAFvalue{\bfvartheta_I\bfvartheta_A} - \NAFvalue{\bfxi_A}$
  concatenated with any NAF of $\tau^{-(w-1)}z_\vartheta$ gives rise to such an expansion, we
  conclude that $\NAFvalue{\bfvartheta_I\bfvartheta_A} - \NAFvalue{\bfxi_A} = 0$
  and therefore $\bfvartheta_I=\bfzero$ and $\bfvartheta_A=\bfxi_A$.  So we
  conclude that all representations of $z$ as a \wNAF{} have to be of the form
  $0\bfldot\bfxi_A0^{w-1}\bfeta$ for some \wNAF{} $\bfeta$. Thus, by using
  Proposition~\vref{pro:char-boundary}, we get $z\not\in\boundary*{\cF}$ and
  therefore $z\in\interior{\cF}$.

  Until now we have proved
  \begin{equation*}
    \uboxdim \boundary*{\cF} \leq \uboxdim U < 2.
  \end{equation*}
  Because the Hausdorff dimension of a set is at most its upper box dimension,
  cf.\ Edgar~\cite{Edgar:2008:measur} again, the desired result follows.
\end{proof}


\section{Cell Rounding Operations}
\label{sec:cell-rounding-op}


Let $\tau\in\C$ be an algebraic integer, imaginary quadratic. In this section,
we define operators working on subsets (regions) of the complex plane. These
will use the lattice $\Ztau$ and the Voronoi cells defined in
Section~\ref{sec:voronoi}. They will be a very useful concept to prove
Theorem~\vref{thm:countdigits}.


\begin{definition}[Cell Rounding Operations]
  \label{def:round-v}

  Let $B\subseteq\C$ and $j\in\R$. We define the \emph{cell packing of $B$}
  (``floor $B$'')
  \begin{align*}
    \floorV{B} &:=  \bigcup_{\substack{z \in \Ztau \\ V_z \subseteq  B}} V_z 
    & &\text{and} &
    \floorV[j]{B} &:= \frac{1}{\tau^j} \floorV{\tau^j B},
    \intertext{the \emph{cell covering of $B$} (``ceil $B$'')}
    \ceilV{B} &:= \closure{\floorV{B^C}^C}
    & &\text{and} &
    \ceilV[j]{B} &:= \frac{1}{\tau^j} \ceilV{\tau^j B}, 
    \intertext{the \emph{fractional cells of $B$}}
    \fracpartV{B} &:= B \setminus \floorV{B}
    & &\text{and} &
    \fracpartV[j]{B} &:= \frac{1}{\tau^j} \fracpartV{\tau^j B},
    \intertext{the \emph{cell covering of the boundary of $B$}}
    \boundaryV{B} &:= \closure{\ceilV{B} \setminus \floorV{B}}
    & &\text{and} &
    \boundaryV[j]{B} &:= \frac{1}{\tau^j} \boundaryV{\tau^j B},
    \intertext{the \emph{cell covering of the lattice points inside $B$}}
    \coverV{B} &:= \bigcup_{\substack{z\in B \cap \Z[\tau]}} V_z
    & &\text{and} &
    \coverV[j]{B} &:= \frac{1}{\tau^j} \coverV{\tau^j B}
    \intertext{and the \emph{number of lattice points inside $B$} as}
    \cardV{B} &:= \card{B \cap \Ztau}
    & &\text{and} &
    \cardV[j]{B} &:= \cardV{\tau^j B}
  \end{align*}
\end{definition}


To get a slight feeling what those operators do, have a look at
Figure~\vref{fig:v-cells-op}. There brief examples are given. For the cell
covering of a set $B$ an alternative, perhaps more intuitive description can be
given by
\begin{equation*}
  \ceilV{B} :=  \bigcup_{\substack{z \in \Ztau \\ V_z \cap B \neq \emptyset}} V_z.
\end{equation*}


\begin{figure}
  \centering
  \includegraphics{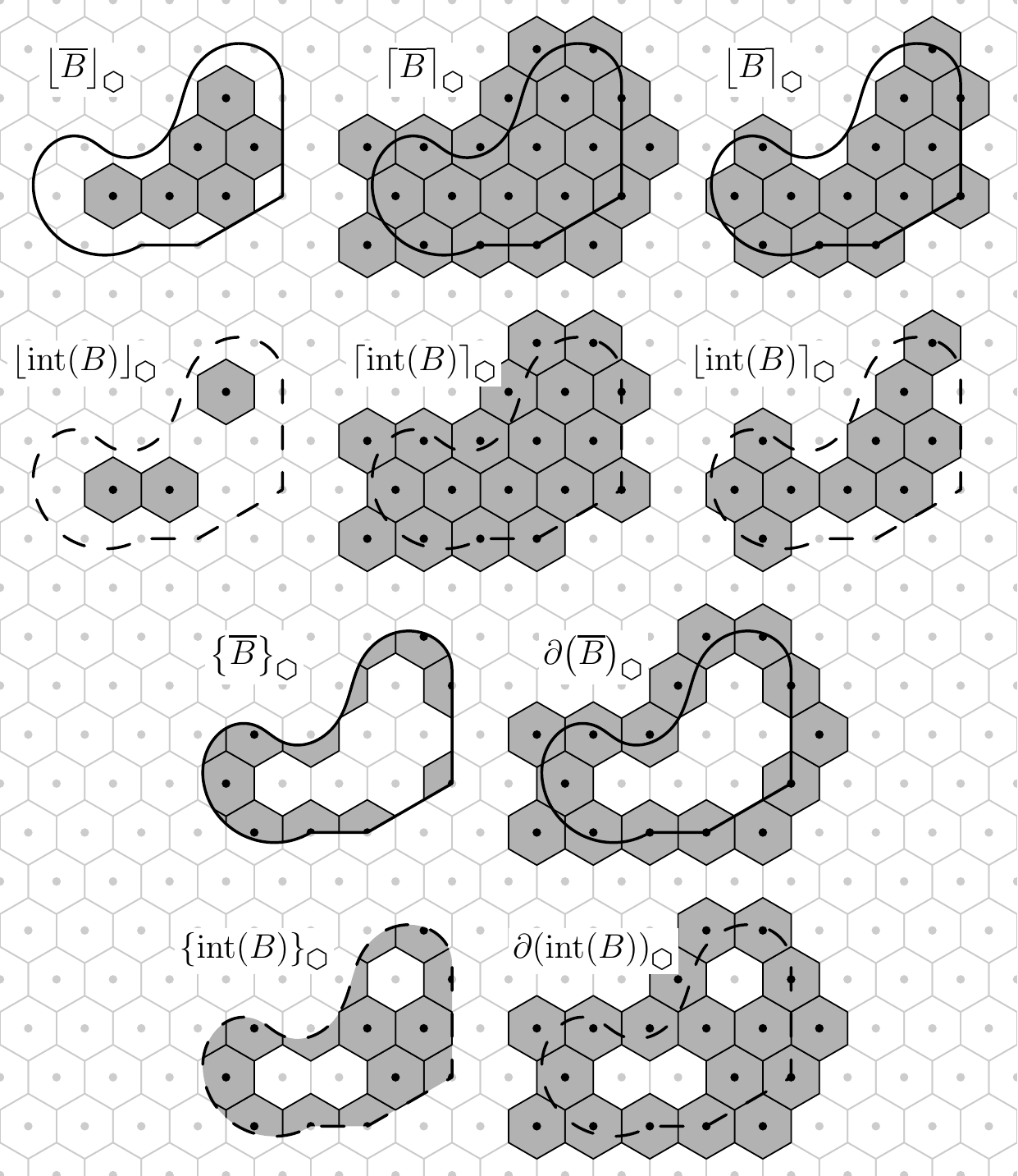}
  \caption[Examples of the cell rounding operators]{Examples of the cell
    rounding operators of Definition~\vref{def:round-v}. As lattice $\Ztau$ with
    $\tau = \frac{3}{2}+\frac{1}{2}\sqrt{-3}$ was used here.}
  \label{fig:v-cells-op}
\end{figure}


The following proposition deals with some basic properties that will be
helpful, when working with those operators.


\begin{proposition}[Basic Properties of Cell Rounding Operations]
  \label{pro:round-v-basic-prop}

  Let $B\subseteq\C$ and $j\in\R$.
  \begin{enumerate}[(a)]
    
  \item \label{enu:round-v-basic-prop:incl}
    We have the inclusions
    \begin{subequations}
      \begin{equation}
        \floorV[j]{B} \subseteq B \subseteq \closure{B} \subseteq \ceilV[j]{B}
      \end{equation}
      and
      \begin{equation}
        \floorV[j]{B} \subseteq \coverV[j]{B} \subseteq \ceilV[j]{B}.    
      \end{equation}
    \end{subequations}
    For $B' \subseteq \C$ with $B \subseteq B'$ we get $\floorV[j]{B} \subseteq
    \floorV[j]{B'}$, $\coverV[j]{B} \subseteq \coverV[j]{B'}$ and $\ceilV[j]{B}
    \subseteq \ceilV[j]{B'}$, i.e., monotonicity with respect to inclusion

  \item \label{enu:round-v-basic-prop:frac-boundary}
    The inclusion
    \begin{equation}
      \fracpartV[j]{B} \subseteq \boundaryV[j]{B}
    \end{equation}
    holds.

  \item \label{enu:round-v-basic-prop:boundary} 
    $\boundary*{B} \subseteq \boundaryV[j]{B}$ and for each cell $V'$ in
    $\boundaryV[j]{B}$ we have $V' \cap \boundary*{B} \neq \emptyset$.
    
  \item \label{enu:round-v-basic-prop:card} 
    For $B' \subseteq \C$ with $B'$ disjoint from $B$, we get
    \begin{equation}
      \cardV[j]{B \cup B'} = \cardV[j]{B} + \cardV[j]{B'},
    \end{equation}
    and therefore the number of lattice points operation is monotonic with
    respect to inclusion, i.e., for $B'' \subseteq \C$ with $B'' \subseteq B$ we
    have $\cardV[j]{B''} \leq \cardV[j]{B}$. Further we get
    \begin{equation}
      \cardV[j]{B} 
      = \cardV[j]{\coverV[j]{B}}
      = \abs\tau^{2j} \frac{\lmeas{\coverV[j]{B}}}{\lmeas{V}}
    \end{equation}
    
  \end{enumerate}
\end{proposition}


\begin{proof}
  \begin{enumerate}[(a)]

  \item $\floorV[j]{B} \subseteq B$ follows directly from the definition. Since
    $\floorV[j]{B^C} \subseteq B^C$, we get
    \begin{equation*}
      \ceilV[j]{B} = 
      \closure{\floorV[j]{B^C}^C} \supseteq
      \closure{\left(B^C\right)^C}
      = \closure{B}.
    \end{equation*}

    The inclusion $\floorV[j]{B} \subseteq \coverV[j]{B}$ follows directly from
    the definitions and $\coverV[j]{B} \subseteq \ceilV[j]{B}$ again by
    considering the complement, because
    $\closure{\coverV[j]{B^C}^C}=\coverV[j]{B}$. Similarly, the monotonicity can
    be shown.

  \item We have
    \begin{equation*}
      \fracpartV[j]{B} 
      = B \setminus \floorV[j]{B}
      \subseteq \ceilV[j]{B} \setminus \floorV[j]{B}
      = \boundaryV[j]{B}.
    \end{equation*}

  \item We assume $j=0$. Using~\itemref{enu:round-v-basic-prop:incl} yields
    $\boundary*{B} \subseteq \closure{B} \subseteq \ceilV{B}$.  Let
    $x\in\boundary*{B}$. If $x \not\in B$, then $\floorV{B}\subseteq B$ implies
    that $x \notin\floorV{B}$. So we get
    \begin{equation*}
      x \in \ceilV{B}\setminus\floorV{B}
      \subseteq \closure{\ceilV{B} \setminus \floorV{B}} 
      = \boundaryV{B}.
    \end{equation*}

    Now suppose $x \in B$. Consider all Voronoi cells $V_i$, $i \in I$, for a
    suitable finite index set $I$, such that $x\in V_i$. We get $x \in
    \interior{\bigcup_{i \in I} V_i}$. If all of the $V_i$ are a subset of
    $\floorV{B}$, then $x\in\interior{B}$, which is a contradiction to
    $x\in\boundary*{B}$. So there is at least one cell $V'$ that is not a subset
    of $\floorV{B}$. Since
    \begin{equation*}
      \ceilV{B}^C = \closure{\floorV{B^C}^C}^C
      = \interior{\bigcup_{\substack{z \in \Ztau \\ V_z \subseteq  B^C}} V_z}
    \end{equation*}
    and $x \not\in B^C$, $V'$ is not in this union of cells. So $V'$
    is in the complement, i.e., $V' \subseteq \ceilV{B}$. And therefore the
    statement follows.

    Now we want to show that there is a subset of the boundary in each
    $V$\nbd-cell $V'$ of $\boundaryV{B}$. Assume $V' \cap \boundary*{B} =
    \emptyset$.  If $V' \cap B = \emptyset$, then $V' \subseteq B^C$, so
    $V'$ is not a subset of $\ceilV{B}$, contradiction. If $V' \cap B
    \neq \emptyset$, then $V' \subseteq B$, since $V'$ does not contain
    any boundary. But then, $V' \subseteq \floorV{B}$, again a
    contradiction.

  \item Since the operator just counts the number of lattice points, the first
    statement follows.

    In the other statement, the first equality follows, because $z\in B \cap
    \Ztau \equivalent V_z \subseteq \coverV{B}$ holds. Since $\coverV[j]{B}$
    consists of cells each with area $\lmeas{\tau^{-j}V}$, the second equality
    is just, after multiplying by $\lmeas{\tau^{-j}V}$, the equality of the
    areas.  \qedhere
  \end{enumerate}
\end{proof}


We will need some more properties concerning cardinality. We want to know the
number of points inside a region after using one of the operators. Especially we
are interested in the asymptotic behaviour, i.e., if our region becomes scaled
very large. The following proposition provides information about that.
 

\begin{proposition}
  \label{pro:set-nu}

  Let $\delta\in\R$ with $\delta>0$, and let $U \subseteq \C$ bounded,
  measurable and such that
  \begin{equation}
    \label{eq:pro:set-nu:covercond}
    \cardV{\boundaryV{NU}} = \Oh{\abs{N}^\delta}
  \end{equation}
  for $N\in\C$.

  \begin{enumerate}[(a)]
    
  \item \label{enu:set-nu:floor-ceil-cover}
    We get
    \begin{equation*}
      \cardV{\floorV{NU}} 
      = \abs{N}^2 \frac{\lmeas{U}}{\lmeas{V}} + \Oh{\abs{N}^\delta},
    \end{equation*}
    \begin{equation*}
      \cardV{\ceilV{NU}} 
      = \abs{N}^2 \frac{\lmeas{U}}{\lmeas{V}} + \Oh{\abs{N}^\delta}
    \end{equation*}
    and
    \begin{equation*}
      \cardV{NU} = \cardV{\coverV{NU}} 
      = \abs{N}^2 \frac{\lmeas{U}}{\lmeas{V}} + \Oh{\abs{N}^\delta}.
    \end{equation*}
    
  \item \label{enu:set-nu:difference} 
    We get
    \begin{equation*}
      \cardV{(N+1)U \setminus NU} = \Oh{\abs{N}^\delta}.
    \end{equation*}

  \end{enumerate}
\end{proposition}


\begin{proof}
  \begin{enumerate}[(a)]

  \item Considering the areas yields
    \begin{equation*}
      \cardV{\floorV{NU}} \lmeas{V} 
      \leq \lmeas{NU} = \abs{N}^2 \lmeas{U}
      \leq \cardV{\ceilV{NU}} \lmeas{V},
    \end{equation*}
    since $\floorV{NU} \subseteq NU \subseteq \ceilV{NU}$, see
    Proposition~\vref{pro:round-v-basic-prop}. If we use $\ceilV{NU} =
    \floorV{NU} \cup \boundaryV{NU}$, we obtain
    \begin{equation*}
      0 \leq \abs{N}^2 \frac{\lmeas{U}}{\lmeas{V}} - \cardV{\floorV{NU}} 
      \leq \cardV{\boundaryV{NU}}
    \end{equation*}
    Because $\cardV{\boundaryV{NU}} = \Oh{\abs{N}^\delta}$ we get
    \begin{equation*}
      \abs{\abs{N}^2 \frac{\lmeas{U}}{\lmeas{V}} - \cardV{\floorV{NU}}}
      = \Oh{\abs{N}^\delta},
    \end{equation*}
    and thus the result follows.

    Combining the previous result and Proposition~\vref{pro:round-v-basic-prop}
    proves the other two statements.

  \item Let $d\in\R$ such that $U\subseteq\ballo{0}{d}$. Let $y\in
    (N+1)U\setminus NU$. Obviously, this is equivalent to $y/(N+1)\in U$ and
    $y/N\notin U$, so there is a $z\in\boundary*{U}$ on the line from $y/N$ to
    $y/(N+1)$. We get
    \begin{equation*}
      \abs{z-\frac{y}{N}} 
      \leq \abs{y} \cdot \abs{\frac{1}{N} - \frac{1}{N+1}}
      = \frac{\abs{y}}{\abs{N+1}}\cdot\frac{1}{\abs{N}}
      \leq \frac{d}{\abs{N}},
    \end{equation*}
    and therefore
    \begin{equation*}
      (N+1)U\setminus NU\subseteq 
      \bigcup_{z\in\boundary{NU}} \ballo{z}{d}.
    \end{equation*}
    Since the boundary $\boundary{NU}$ can be covered by $\Oh{\abs{N}^\delta}$
    cells, cf.\ \itemref{enu:round-v-basic-prop:boundary} of
    Proposition~\vref{pro:round-v-basic-prop} and the discs in
    $\bigcup_{z\in\boundary{NU}} \ballo{z}{d}$ have a fixed size, the result
    follows. \qedhere
  \end{enumerate}
\end{proof}


If the geometry of $U$ is simple, e.g.\ $U$ is a disc or $U$ is a polygon, then
we can check the covering condition~\eqref{eq:pro:set-nu:covercond} of
Proposition~\vref{pro:set-nu} by means of the following proposition.


\begin{proposition}
  \label{pro:simple-geometries}

  Let $U\subseteq\C$ such that the boundary of $U$ consists of finitely many
  rectifiable curves. Then we get
  \begin{equation*}
    \cardV{\boundaryV{NU}} = \Oh{\abs{N}}
  \end{equation*}
  for $N\in\C$.
\end{proposition}


\begin{proof}
  Without loss of generality, we may assume that the boundary of $U$ is a
  rectifiable curve $\gamma\colon \intervalcc{0}{L} \fto \C$, which is
  parametrised by arc length. For any $t\in
  \intervalcc{1/(2\abs{N})}{L-1/(2\abs{N})}$, we have
  \begin{equation*}
    \f{\gamma}{\intervalcc{t-\frac1{2\abs{N}}}{t+\frac{1}{2\abs{N}}}}
    \subseteq \ballo{\f{\gamma}{t}}{\frac{1}{2\abs{N}}},
  \end{equation*}
  as the straight line from $\f{\gamma}{t}$ to $\f{\gamma}{t'}$ is never longer
  than the arc-length of $\f{\gamma}{\intervalcc{t}{t'}}$. Thus $\boundary*{U}$
  can be covered by $\Oh{L\abs{N}}$ discs of radius $1/(2\abs{N})$ and
  consequently, $\boundary{NU}$ can be covered by $\Oh{L\abs{N}}$ discs of
  radius $\frac12$. As $\Ztau$ is a lattice, each disc with radius $\frac12$
  is contained in at most $4$~Voronoi-cells, cf.\
  Proposition~\vref{pro:voronoi-prop}. Therefore, $\Oh{N}$ cells suffice to
  cover $\boundary{NU}$.
\end{proof}


\section{The Characteristic Sets \texorpdfstring{$W_\eta$}{W\_eta}}
\label{sec:sets-w_eta}


Let $\tau\in\C$ be an algebraic integer, imaginary quadratic.  Suppose that
$\abs{\tau}>1$. Let $w\in\N$ with $w\geq2$. Further let $\cD$ be a minimal norm
representatives digit set modulo $\tau^w$ as in
Definition~\vref{def:min-norm-digit-set}. We denote the norm function by
$\normsymbol \colon \Ztau \fto \Z$, and we simply have $\abs{\normtau} =
\abs{\tau}^2$. Again for simplicity we set $\cD^\bullet := \cD \setminus
\set{0}$.


In this section we define characteristic sets for a digit at a specified
position in the \wNAF{} expansion and prove some basic properties of them. Those
will be used in the proof of Theorem~\vref{thm:countdigits}.


\begin{definition}[Characteristic Sets]
  \label{def:w-wk-beta}

  Let $\eta\in\cD^\bullet$. For $j\in\N_0$ define 
  \begin{equation*}
    \cW_{\eta,j} :=
    \set*{\NAFvalue{\bfxi}}{\text{
        $\bfxi\in\wNAFsetellell{0}{j+w}$ with $\xi_{-w}=\eta$}}.
  \end{equation*}
  We call $\coverV[j+w]{\cW_{\eta,j}}$ the \emph{$j$th approximation of the
    characteristic set for $\eta$}, and we define
  \begin{equation*}
    W_{\eta,j} := \fracpartZtau{\coverV[j+w]{\cW_{\eta,j}}}.
  \end{equation*}
  Further we define the \emph{characteristic set for $\eta$}
  \begin{equation*}
    \cW_\eta := 
    \set*{\NAFvalue{\bfxi}}{\text{
        $\bfxi\in \wNAFsetinf$ with $\xi_{-w}=\eta$}}.
  \end{equation*}
  and 
  \begin{equation*}
    W_\eta := \fracpartZtau{\cW_\eta}.
  \end{equation*}

  For $j\in\N_0$ we set
  \begin{equation*}
    \beta_{\eta,j} := \lmeas{\coverV[j+w]{\cW_{\eta,j}}} - \lmeas{\cW_\eta}.
  \end{equation*}
\end{definition}


\begin{figure}
  \centering 
  \subfloat[$\cW_{\eta,j}$ for $\tau=\frac{1}{2} + \frac{1}{2}
  \sqrt{-7}$, $w=2$ and $j=11$]{
    \includegraphics[height=7cm]{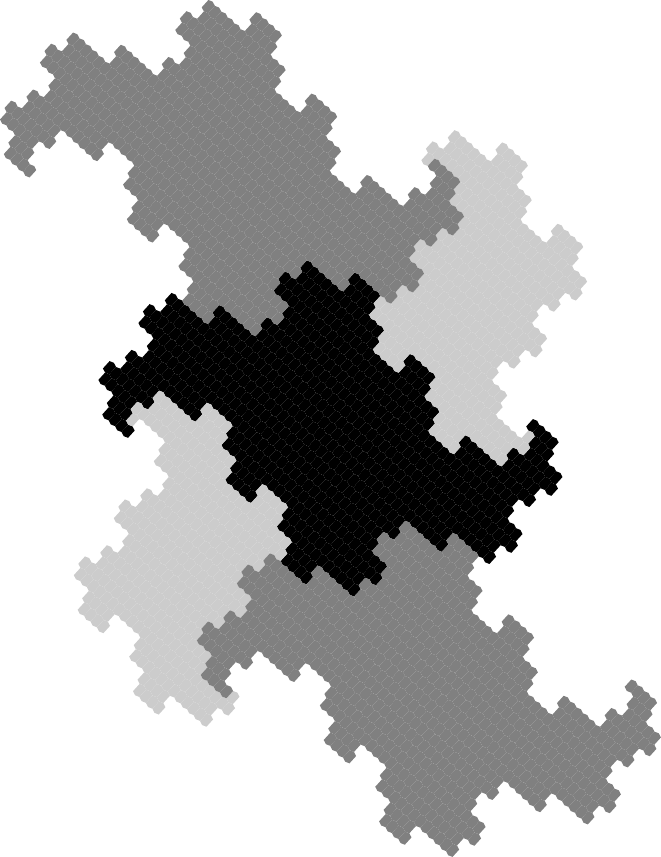}
    \label{fig:w-weta:1}}
  \quad 
  \subfloat[$\cW_{\eta,j}$ for $\tau=1 + \sqrt{-1}$, $w=4$ and $j=11$]{
    \includegraphics[height=7cm]{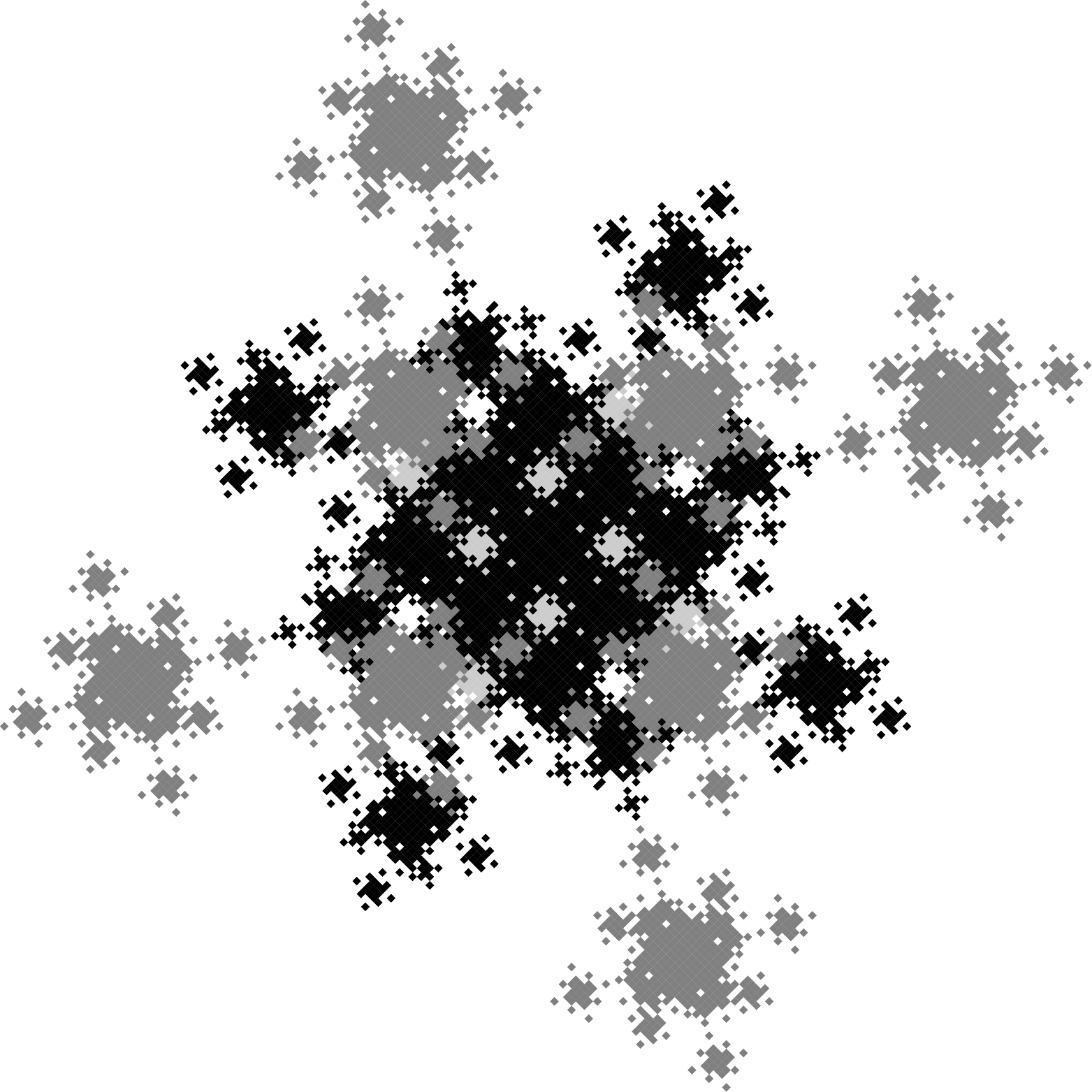}
    \label{fig:w-weta:2}}
  \\
  \subfloat[$\cW_{\eta,j}$ for $\tau=\frac{3}{2} + \frac{1}{2}
  \sqrt{-3}$, $w=2$ and $j=7$]{
    \includegraphics[height=7cm]{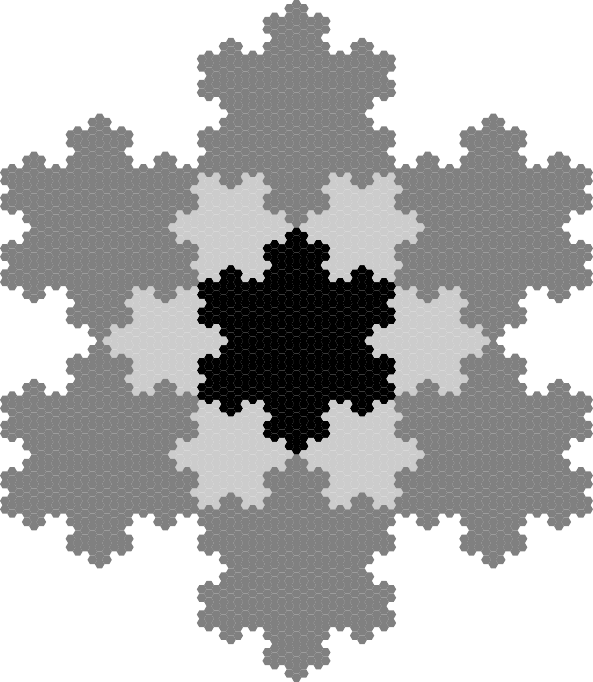}
    \label{fig:w-weta:3}}
  \quad
  \subfloat[$\cW_{\eta,j}$ for $\tau=\frac{3}{2} + \frac{1}{2}
  \sqrt{-3}$, $w=3$ and $j=6$]{
    \includegraphics[height=7cm]{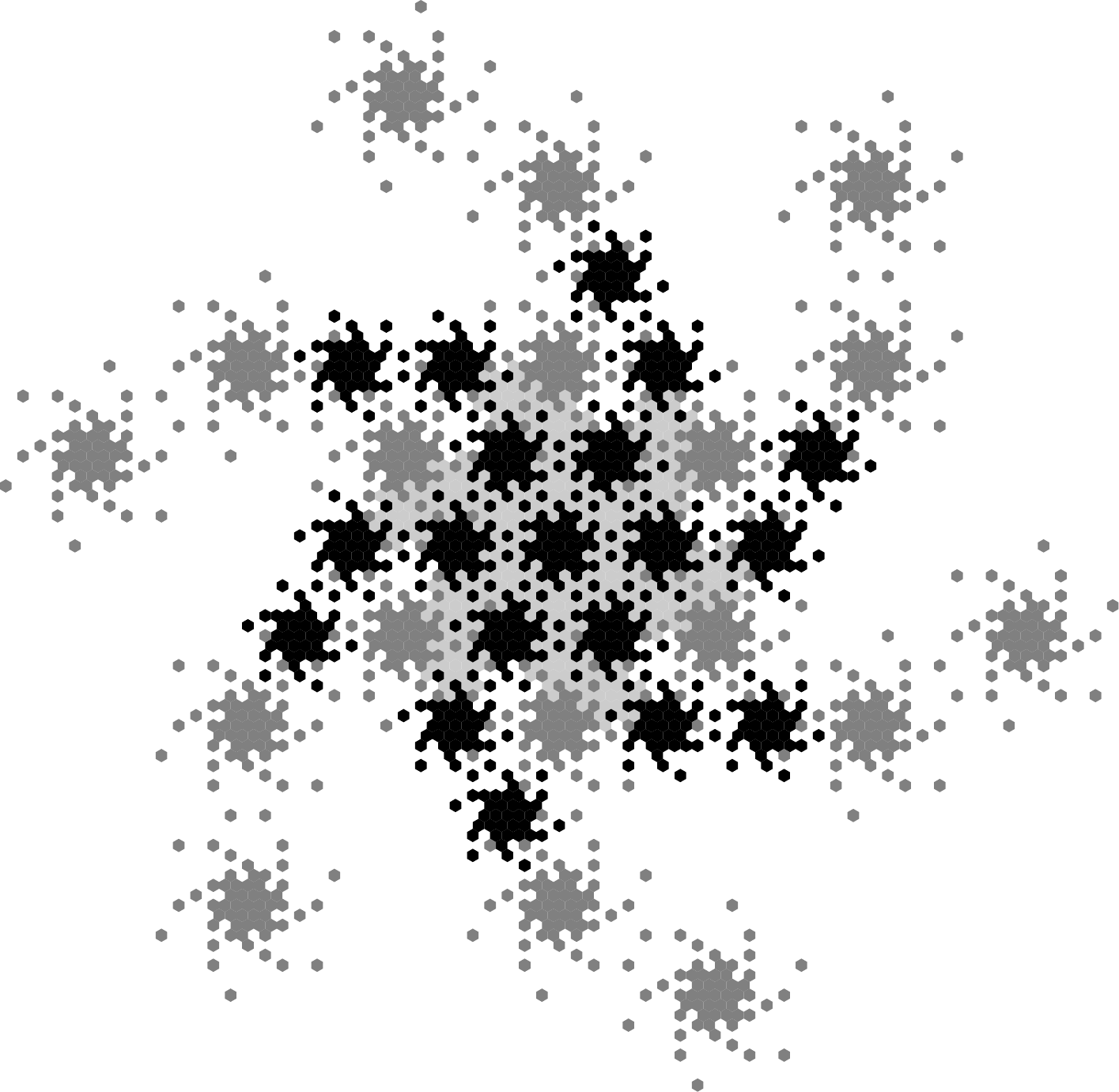}
    \label{fig:w-weta:4}}

  \caption[Characteristic sets $\cW_\eta$]{Characteristic sets $\cW_\eta$. Each
    figure can either be seen as approximation $\cW_{\eta,j}$ for $\cW_\eta$, or
    as values of \wNAF{}s of length $j$, where a scaled Voronoi cell is drawn
    for each point. Different colours correspond to the digits $1$ and $w$ from
    the left in the \wNAF{}. They are ``marked'' whether they are zero or
    non-zero.}
  \label{fig:w-weta}
\end{figure}


Note that sometimes the set $W_\eta$ will also be called \emph{characteristic
  set for $\eta$}, and analogously for the set $W_{\eta,j}$. In
Figure~\vref{fig:w-weta} some of these characteristic sets --- more precisely
some approximations of the characteristic sets --- are shown. The following
proposition will deal with some properties of those defined sets,

\begin{proposition}[Properties of the Characteristic Sets]
  \label{pro:prop-of-w}

  Let $\eta\in\cD^\bullet$.
  \begin{enumerate}[(a)]
  
  \item \label{enu:prop-of-w:fund-domain} We have 
    \begin{equation*}
      \cW_\eta = \eta \tau^{-w} + \tau^{-2w+1} \cF.
    \end{equation*}
  
  \item \label{enu:prop-of-w:compact}  The set $\cW_\eta$ is compact.

  \item \label{enu:prop-of-w:w-is-union} We get
    \begin{equation*}
      \cW_\eta 
      = \closure{\bigcup_{j\in\N_0} \cW_{\eta,j}}
      = \closure{\lim_{j\to\infty} \cW_{\eta,j}}.
    \end{equation*}

  \item \label{enu:prop-of-w:lim} The set $\coverV[j+w]{\cW_{\eta,j}}$ is indeed
    an approximation of $\cW_\eta$, i.e., we have
    \begin{equation*}
      \cW_\eta 
      = \closure{\liminf_{j\in\N_0} \coverV[j+w]{\cW_{\eta,j}}}
      = \closure{\limsup_{j\in\N_0} \coverV[j+w]{\cW_{\eta,j}}}.
    \end{equation*}

  \item \label{enu:prop-of-w:int-in-liminf} We have $\interior*{\cW_\eta}
    \subseteq \liminf_{j\in\N_0} \coverV[j+w]{\cW_{\eta,j}}$.

  \item \label{enu:prop-of-w:w-in-v} We get $\cW_\eta - \eta \tau^{-w} \subseteq
    V$, and for $j\in\N_0$ we obtain $\coverV[j+w]{\cW_{\eta,j}} - \eta
    \tau^{-w} \subseteq V$.

  \item \label{enu:prop-of-w:lmeas} For the Lebesgue measure of the
    characteristic set we obtain $\lmeas{\cW_\eta} = \lmeas{W_\eta}$ and for its
    approximation $\lmeas{\coverV[j+w]{\cW_{\eta,j}}} = \lmeas{W_{\eta,j}}$.

  \item \label{enu:prop-of-w:area-wj}
    Let $j\in\N_0$. If $j<w-1$, then the area of $\coverV[j+w]{\cW_{\eta,j}}$ is
    \begin{equation*}
      \lmeas{\coverV[j+w]{\cW_{\eta,j}}} = \abs\tau^{-2(j+w)} \lmeas{V}.
    \end{equation*}
    If $j \geq w-1$, then the area of $\coverV[j+w]{\cW_{\eta,j}}$ is
    \begin{equation*}
      \lmeas{\coverV[j+w]{\cW_{\eta,j}}} = \lmeas{V} e_w + \Oh{\rho^j}
    \end{equation*}
    with $e_w$ and $\rho$ from Theorem~\vref{th:w-naf-distribution}.    
    
  \item \label{enu:prop-of-w:area-w}
    The area of $W_\eta$ is
    \begin{equation*}
      \lmeas{W_\eta} = \lmeas{V} e_w,
    \end{equation*}
    again with $e_w$ from Theorem~\vref{th:w-naf-distribution}.

  \item \label{enu:prop-of-w:beta}
    Let $j\in\N_0$. We get
    \begin{equation*}
      \beta_{\eta,j} = \int_{x \in V} \f{\left( 
          \indicator*{W_{\eta,j}} - \indicator*{W_\eta} \right)}{x} \dd x.
    \end{equation*}
    If $j<w-1$, then its value is
    \begin{equation*}
      \beta_{\eta,j} = \left(\abs\tau^{-2(j+w)}-e_w\right) \lmeas{V}.
    \end{equation*}
    If $j \geq w-1$, then we get
    \begin{equation*}
      \beta_{\eta,j} = \Oh{\rho^j}.
    \end{equation*}
    Again $e_w$ and $\rho$ can be found in Theorem~\vref{th:w-naf-distribution}.

  \end{enumerate}
\end{proposition}


\begin{proof}
  \begin{enumerate}[(a)]

  \item Is clear, since we have the digit $\eta$ at index $-w$ and an arbitrary
    \wNAF{} starting with index $-2w$. Note that the elements in $\wNAFsetinf$
    start with index $-1$.

  \item Follows directly from~\itemref{enu:prop-of-w:fund-domain}, because $\cF$
    is compact according to Proposition~\vref{pro:fund-domain-compact}.
    
  \item Clearly we have $\cW_{\eta,j} \subseteq \cW_\eta$. Thus
    $\bigcup_{j\in\N_0} \cW_{\eta,j} \subseteq \cW_\eta$, and because $\cW_\eta$
    is closed, the inclusion $\closure{\bigcup_{j\in\N_0} \cW_{\eta,j}}
    \subseteq \cW_\eta$ follows. Now let $z\in\cW_\eta$, and let $\bfxi \in
    \wNAFsetinf$, such that $\NAFvalue{\bfxi}=z$. Then there is a sequence of
    \wNAF{}s $\sequence{\bfxi_\ell}_{\ell\in\N_0}$ with finite right-lengths
    that converges to $\bfxi$ and clearly
    \begin{equation*}
      \NAFvalue{\bfxi_\ell} \in \bigcup_{k\in\N_0} \cW_{\eta,k}.
    \end{equation*}
    Since evaluating the value is a continuous function, see
    Proposition~\vref{pro:value-continuous}, we get
    \begin{equation*}
      z = \NAFvalue{\bfxi} 
      = \NAFvalue{ \lim_{\ell\to\infty} \bfxi_\ell }
      = \lim_{\ell\to\infty} \NAFvalue{\bfxi_\ell}
      \in \closure{\bigcup_{k\in\N_0} \cW_{\eta,k}}.
    \end{equation*}

    The equality $\bigcup_{j\in\N_0} \cW_{\eta,j} = \lim_{j\to\infty}
    \cW_{\eta,j}$ is obvious, since $\cW_{\eta,j}$ is monotonic increasing.

  \item First we show that we have
    \begin{equation*}
      \limsup_{j\to\infty} \coverV[j+w]{\cW_{\eta,j}} \subseteq \cW_\eta.
    \end{equation*}
    Let 
    \begin{equation*}
      z \in \limsup_{j\to\infty} \coverV[j+w]{\cW_{\eta,j}} 
      = \bigcap_{j \in \N_0} \bigcup_{k \geq j} \coverV[k+w]{\cW_{\eta,k}}.
    \end{equation*}
    Then there is a $j_0 \geq 0$ such that $z \in
    \coverV[j_0+w]{\cW_{\eta,j_0}}$. Further, for $j_{\ell-1}$ there is a $j_\ell
    \geq j_{\ell-1}$, such that $z \in \coverV[j_\ell+w]{\cW_{\eta,j_\ell}}$. For
    each $\ell\in\N_0$ there is a $z_\ell \in \cW_{\eta,j_\ell} \subseteq
    \cW_\eta$ with
    \begin{equation*}
      \abs{z-z_\ell} \leq c_V \abs{\tau} \abs{\tau}^{-j_\ell-w},
    \end{equation*}
    since $\coverV[j_\ell+w]{\cW_{\eta,j_\ell}}$ consists of cells
    $\abs{\tau}^{-j_\ell-w} V$ with centres out of $\cW_{\eta,j_\ell}$. Refer to
    Proposition~\vref{pro:voronoi-prop} for the constant $c_V \abs{\tau}$. Thus
    we get $z = \lim_{\ell\to\infty} z_\ell \in \cW_\eta$, since
    $\abs{\tau}^{-j_\ell-w}$ tends to $0$ for large $\ell$ and $\cW_\eta$ is
    closed.
    
    Using the closeness property of $\cW_\eta$ again yields 
    \begin{equation*}
      \closure{\limsup_{j\to\infty} \coverV[j+w]{\cW_{\eta,j}}}
      \subseteq \cW_\eta.
    \end{equation*}
    Now we are ready to show the stated equalities. We obtain
    \begin{equation*}
      \cW_\eta 
      = \closure{\lim_{j\to\infty} \cW_{\eta,j}}
      = \closure{\liminf_{j\to\infty} \cW_{\eta,j}}
      \subseteq \closure{\liminf_{j\to\infty} \coverV[j+w]{\cW_{\eta,j}}}
      \subseteq \closure{\limsup_{j\to\infty} \coverV[j+w]{\cW_{\eta,j}}}
      \subseteq \cW_\eta,
    \end{equation*}
    so equality holds everywhere.

  \item Let $z\in\interior*{\cW_\eta}$. Then there exists an $\eps>0$ such that
    $\ballo{z}{\eps}\subseteq\interior*{\cW_\eta}$. For each $k\in\N_0$ there is
    a $y\in\tau^{-k-w}\Ztau$ with the property that $z$ is in the corresponding
    Voronoi cell, i.e., $z\in y+\tau^{-k-w}V$. For this $y$, there is also an
    $\bfxi\in\wNAFsetfinell{k+w}$ such that $y=\NAFvalue{\bfxi}$. 

    Clearly, if $k$ is large enough, say $k\geq j$, we obtain $y+\tau^{-k-w}V
    \subseteq \ballo{z}{\eps}$. From Proposition~\vref{pro:char-boundary}
    (combined with~\itemref{enu:prop-of-w:fund-domain}) we know that all
    \wNAF{}s corresponding to the values in $\interior*{\cW_\eta}$ must have
    $\eta$ at digit $-w$ and integer part $\bfzero$. But this means that
    $\NAFvalue{\bfxi}\in\cW_{\eta,k}$ and therefore
    $z\in\coverV[k+w]{\cW_{\eta,k}}$. So we conclude
    \begin{equation*}
      z \in \bigcup_{j\in\N_0} \bigcap_{k\geq j} \coverV[k+w]{\cW_{\eta,k}}
      = \liminf_{j\to\infty} \coverV[j+w]{\cW_{\eta,j}}.
    \end{equation*}

  \item Each \wNAF{} $\bfxi\in\wNAFsetinf$ corresponding to a value in $\cW_\eta
    - \eta \tau^{-w}$ starts with $2w-1$ zeros from the left. Therefore
    \begin{equation*}
      \NAFvalue{\bfxi} = \tau^{1-2w} \NAFvalue{\bfvartheta}
    \end{equation*}
    for an appropriate \wNAF{} $\bfvartheta\in\wNAFsetinf$. Thus, using
    $\NAFvalue{\bfvartheta}\in\tau^{2w-1}V$ from
    Proposition~\vref{pro:upper-bound-fracnafs}, the desired inclusion follows.
    
    The set $\coverV[j+w]{\cW_{\eta,j}} - \eta \tau^{-w} \subseteq V$ consists
    of cells of type $\tau^{-j-w}V$, where their centres are the fractional
    value of an element $\bfxi\in\wNAFsetellell{0}{j+w}$. Again the first $2w-1$
    digits are zero, so
    \begin{equation*}
      \NAFvalue{\bfxi} = \tau^{1-2w} \NAFvalue{\bfvartheta}
    \end{equation*}
    for an appropriate \wNAF{} $\bfvartheta\in\wNAFsetellell{0}{j-w+1}$. Suppose
    $j \geq w-1$. Using $\NAFvalue{\bfvartheta} + \tau^{-(j-w+1)}V \subseteq
    \tau^{2w-1}V$ again from Proposition~\vref{pro:upper-bound-fracnafs}, the
    statement follows. If $j<w-1$, then $\NAFvalue{\bfvartheta}=0$ and it
    remains to show that $\tau^{-j-w}V \subseteq V$. But this is clearly true,
    since $\tau^{-1}V \subseteq V$ according to
    Proposition~\vref{pro:voronoi-prop}.

  \item As a shifted version of the sets $\cW_\eta$ and
    $\coverV[j+w]{\cW_{\eta,j}}$ is contained in $V$ by
    \itemref{enu:prop-of-w:w-in-v}, so the equality of the Lebesgue measures
    follows directly.

  \item The set $\coverV[j+w]{\cW_{\eta,j}}$ consists of of cells of type
    $\tau^{-j-w}V$, where their centres are the value of an element
    $\bfxi\in\wNAFsetellell{0}{j+w}$. The intersection of two
    different cells is contained in the boundary of the cells, so a set of
    Lebesgue measure zero.
    
    Suppose $j \geq w-1$. Since the digit $\xi_{-w}=\eta$ is fixed, the first
    $2w-1$ digits from the left are fixed, too. The remaining word
    $\xi_{-2w}\ldots\xi_{-(j+w)}$ can be an arbitrary \wNAF{} of length $j-w+1$,
    so there are $C_{j-w+1,w}$ choices, see
    Theorem~\vref{th:w-naf-distribution}.

    Thus we obtain
    \begin{equation*}
      \lmeas{\coverV[j+w]{\cW_{\eta,j}}} = C_{j-w+1,w} \lmeas{\tau^{-(w+j)} V}
      = C_{j-w+1,w} \abs\tau^{-2(w+j)} \lmeas{V}.
    \end{equation*}
    Inserting the results of Theorem~\vref{th:w-naf-distribution} yields
    \begin{equation*}
      \begin{split}
        \lmeas{\coverV[j+w]{\cW_{\eta,j}}}  
        &= \left( \frac{\abs\tau^{2(j-w+1+w)}}{\left(\abs\tau^2-1\right) w + 1} 
          + \Oh{\left(\rho\abs\tau^2\right)^{j-w+1}} \right)
        \abs\tau^{-2(w+j)} \lmeas{V} \\
        &= \lmeas{V} \underbrace{\frac{1}{\abs\tau^{2(w-1)}
          \left(\left(\abs\tau^2-1\right) w + 1\right)}}_{= e_w}
          + \Oh{\rho^j}.
      \end{split}
    \end{equation*}
    
    If $j<w-1$, then $\coverV[j+w]{\cW_{\eta,j}}$ consists of only one cell of
    size $\tau^{-j-w}V$, so the stated result follows directly. 

  \item Using \itemref{enu:prop-of-w:lim}, \itemref{enu:prop-of-w:int-in-liminf}
    and the continuity of the Lebesgue measure yields
    \begin{equation*}
      \begin{split}
        \lmeas{\liminf_{j\in\N_0} \coverV[j+w]{\cW_{\eta,j}}}
        &\leq \liminf_{j\in\N_0} \lmeas{\coverV[j+w]{\cW_{\eta,j}}}
        \leq \limsup_{j\in\N_0} \lmeas{\coverV[j+w]{\cW_{\eta,j}}} \\
        &\leq \lmeas{\limsup_{j\in\N_0} \coverV[j+w]{\cW_{\eta,j}}}
        \leq \lmeas{\closure{\limsup_{j\in\N_0} \coverV[j+w]{\cW_{\eta,j}}}} \\
        &= \lmeas{\cW_\eta}
        \leq \lmeas{\interior*{\cW_\eta}} + \lmeas{\boundary*{\cW_\eta}} \\
        &\leq \lmeas{\liminf_{j\in\N_0} \coverV[j+w]{\cW_{\eta,j}}}
        + \lmeas{\boundary*{\cW_\eta}}.
      \end{split}
    \end{equation*}
    Since $\lmeas{\boundary*{\cW_\eta}}=0$,
    combine~\itemref{enu:prop-of-w:fund-domain} and
    Proposition~\vref{pro:boundary-fund-dom-dim-upper} to see this, we have
    equality everywhere, so
    \begin{equation*}
      \lmeas{\cW_\eta} = \lim_{j\in\N_0} \lmeas{\coverV[j+w]{\cW_{\eta,j}}}.
    \end{equation*}
    Thus the desired result follows from \itemref{enu:prop-of-w:area-wj},
    because $\rho < 1$.

  \item Using \itemref{enu:prop-of-w:w-in-v} and \itemref{enu:prop-of-w:lmeas}
    yields the first statement. The other result follows directly by using
    \itemref{enu:prop-of-w:area-wj} and \itemref{enu:prop-of-w:area-w}.
    \qedhere
  \end{enumerate}
\end{proof}


Using the results of the previous proposition, we can finally determine the
Lebesgue measure of the fundamental domain $\cF$ defined in
Section~\ref{sec:fundamental-domain}.


\begin{remark}[Lebesgue Measure of the Fundamental Domain]
  \label{rem:meas-fund-domain}

  We get
  \begin{equation*}
    \lmeas{\cF} = \abs\tau^{2(2w-1)} e_w \lmeas{V}
    = \frac{\abs\tau \abs{\im{\tau}}}{(\abs\tau^2-1)w+1} ,
  \end{equation*}
  using~\itemref{enu:prop-of-w:fund-domain} and~\itemref{enu:prop-of-w:area-w}
  from Proposition~\vref{pro:prop-of-w}, $e_w$ from
  Theorem~\vref{th:w-naf-distribution}, and $\lmeas{V}=\abs{\im{\tau}}$ from
  Proposition~\vref{pro:voronoi-prop}.
\end{remark}


The next lemma makes the connection between the \wNAF{}s of elements of the
lattice $\Ztau$ and the characteristic sets $W_{\eta,j}$.


\begin{lemma}
  \label{lem:equiv-digit-char-set}

  Let $\eta\in\cD^\bullet$, $j\geq0$. Let $n\in\Ztau$ and let
  $\bfn\in\wNAFsetfin$ be its \wNAF{}. Then the following statements are
  equivalent:
  \begin{enumequivalences}
  \item \label{enu:e-d-cs:1} 
    The $j$th digit of $\bfn$ equals $\eta$.
  \item \label{enu:e-d-cs:2} 
    The condition $\fracpartZtau{\tau^{-(j+w)}n} \in W_{\eta,j}$ holds.
  \item \label{enu:e-d-cs:3}
    The inclusion $\fracpartZtau{\tau^{-(j+w)}V_n} \subseteq W_{\eta,j}$ holds.
  \end{enumequivalences}
\end{lemma}


\begin{proof}
  Define $\bfm$ by
  \begin{equation*}
    m_k := 
    \begin{cases} 
      n_k & \text{if $k<j+w$,} \\ 
      0 & \text{if $k \geq j+w$}
    \end{cases}
  \end{equation*}
  and $m := \NAFvalue{\bfm}$. Then, by definition, $m \equiv n
  \pmod{\tau^{j+w}}$,
      
  \begin{equation*}
    \fracpartZtau{\tau^{-(j+w)}n} = \fracpartZtau{\tau^{-(j+w)}m}
  \end{equation*}
  and $\tau^{-(j+w)}m \in \cF$.  As the $j$th digit of $\bfn$ only
  depends on the $j+w$ least significant digits of $\bfn$, it is sufficient to
  show the equivalence of the assertions when $\bfn$ and $n$ are replaced by
  $\bfm$ and $m$, respectively.

  By definition, $m_j=\eta$ is equivalent to $\tau^{-(j+w)}m \in \cW_{\eta,j}$.
  
  \begin{descproofequivalences}

  \item[\labelproofequivalent{\ref{enu:e-d-cs:1}}{\ref{enu:e-d-cs:3}}] 

    Assume now that $\tau^{-(j+w)}m \in \cW_{\eta,j}$. Then $m \in
    \tau^{j+w} \cW_{\eta,j}$ and 
    \begin{equation*}
      \tau^{-(j+w)}V_m \subseteq \coverV[j+w]{\cW_{\eta,j}}.
    \end{equation*}
    This implies $\fracpartZtau{\tau^{-(j+w)}V_m} \subseteq W_j$.

    \item[\labelproofequivalent{\ref{enu:e-d-cs:3}}{\ref{enu:e-d-cs:2}}] 

      This implication holds trivially.

    \item[\labelproofequivalent{\ref{enu:e-d-cs:2}}{\ref{enu:e-d-cs:1}}] 

      So now assume that $\fracpartZtau{\tau^{-(j+w)}m} \in W_{\eta,j}$. Thus
      there is an $m'$ such that
      \begin{equation*}
        \tau^{-(j+w)}m' \in \coverV[j+w]{\cW_{\eta,j}}
      \end{equation*}
      and
      \begin{equation*}
        \tau^{-(j+w)}m - \tau^{-(j+w)}m' \in \Ztau.
      \end{equation*}
      This immediately implies $m'\in\Ztau$ and $m\equiv
      m'\pmod{\tau^{j+w}}$. We also conclude that $m'\in \tau^{j+w}
      \coverV[j+w]{\cW_{\eta,j}}$. As $m'\in\Ztau$, this is equivalent to $m'\in
      \tau^{j+w}\cW_{\eta,j}$ and therefore $\tau^{-(j+w)}m'\in\cW_{\eta,j}$. By
      definition of $\cW_{\eta,j}$, there is a $0.\bfm'\in
      \wNAFsetellell{0}{j+w}$ such that $\tau^{-(j+w)}m'=\NAFvalue{0.\bfm'}$,
      i.e., $m'=\NAFvalue{\bfm'}$, and $m'_j=\eta$. From $m'\equiv
      m\pmod{\tau^{j+w}}$ we conclude that $m_j=\eta$, too. (In fact, one can
      now easily show that we have $\bfm'=\bfm$, but this is not really needed.)
      \qedhere
    \end{descproofequivalences}
\end{proof}


\section{Counting the Occurrences of a non-zero Digit in a Region}
\label{sec:counting-digits-region}


Let $\tau\in\C$ be an algebraic integer, imaginary quadratic.  Suppose that
$\abs{\tau}>1$. Let $w\in\N$ with $w\geq2$. Further let $\cD$ be a minimal norm
representatives digit set modulo $\tau^w$ as in
Definition~\vref{def:min-norm-digit-set}.


We denote the norm function by $\normsymbol \colon \Ztau \fto \Z$, and we
simply have $\abs{\normtau} = \abs{\tau}^2$. We write $\tau = \abs{\tau}
e^{i\theta}$ for $\theta \in \intervaloc{-\pi}{\pi}$. Further, recall Iverson's
notation $\iverson{\var{expr}}=1$ if $\var{expr}$ is true and
$\iverson{\var{expr}}=0$ otherwise, cf.\ Graham, Knuth and
Patashnik~\cite{Graham-Knuth-Patashnik:1994}, and that the Lebesgue measure is
denoted by $\lambda$.


In this section we will prove our main result on the asymptototic number of
occurrences of a digit in a given region.

\begin{theorem}[Counting Theorem]
  \label{thm:countdigits}

  Let $0 \neq \eta \in \cD$ and $N\in\R$ with $N\geq0$. Further let $U \subseteq
  \C$ be measurable with respect to the Lebesgue measure, $U \subseteq
  \ball{0}{d}$ with $d$ finite, i.e., $U$ bounded, and set $\delta$ such that
  $\cardV{\boundaryV{NU}} = \Oh{N^\delta}$. Assume $1\leq\delta<2$. We denote
  the number of occurrences of the digit $\eta$ in all width\nbd-$w$
  non-adjacent forms with value in the region~$NU$ by
  \begin{equation*}
    \f{Z_{\tau,w,\eta}}{N}
    = \sum_{z \in NU\cap\Ztau} \sum_{j\in\N_0} 
    \iverson*{$j$th digit of $z$ in its \wNAF{}-expansion equals $\eta$}.
  \end{equation*}
  Then we get  
  \begin{equation*}
      \f{Z_{\tau,w,\eta}}{N} 
      = e_w N^2 \lmeas{U} \log_{\abs\tau} N
      + N^2 \f{\psi_\eta}{\log_{\abs\tau} N}
      + \Oh{N^{\alpha} \log_{\abs\tau} N} 
      + \Oh{N^\delta \log_{\abs\tau} N},
  \end{equation*}
  in which the following expressions are used. We have the constant of the
  expectation
  \begin{equation*}
    e_w = \frac{1}{\abs\tau^{2(w-1)} 
      \left(\left(\abs\tau^2-1\right) w + 1\right)},
  \end{equation*}
  cf.\ Theorem~\vref{th:w-naf-distribution}. Then there is the function
  \begin{equation*}
    \f{\psi_\eta}{x} = \f{\psi_{\eta,\cM}}{x} 
    + \f{\psi_{\eta,\cP}}{x} + \f{\psi_{\eta,\cQ}}{x},
  \end{equation*}
  where
  \begin{equation*}
    \f{\psi_{\eta,\cM}}{x} 
    = \lmeas{U} \left( c + 1 - \fracpart{x} \right) e_w,
  \end{equation*}
  \begin{equation*}
    \f{\psi_{\eta,\cP}}{x} =
    \frac{\abs\tau^{2(c-\fracpart{x})}}{\lmeas{V}} \sum_{j=0}^\infty 
    \int_{y \in \fracpartV[j-w]{\abs\tau^{\fracpart{x}-c} \f{\wh{\theta}}{\floor{x}} U}}
    \left( \indicator{W_\eta}{\fracpartZtau{y \tau^{j-w}}} 
      - \lmeas{W_\eta} \right) \dd y,
  \end{equation*}
  with the rotation $\f{\wh{\theta}}{x} = e^{-i \theta x - i \theta c}$,
  and
  \begin{equation*}
    \f{\psi_{\eta,\cQ}}{x} = \psi_{\eta,\cQ} 
    = \frac{\lmeas{U}}{\lmeas{V}}
    \sum_{j=0}^\infty \frac{\beta_j}{\lmeas{V}}.
  \end{equation*}
  We have $\alpha = 2 + \log_{\abs\tau}\rho < 2$, with $\rho = \left( 1 +
    \frac{1}{\abs\tau^2 w^3} \right)^{-1} < 1$, and
  \begin{equation*}
    c = \floor{ \log_{\abs\tau} d - \log_{\abs\tau} f_L } + 1 
  \end{equation*}
  with the constant $f_L$ of Proposition~\vref{pro:lower-bound-fracnafs}.

  Further, if there is a $p\in\N$, such that $e^{2 i \theta p} U = U$, then
  $\psi_\eta$ is \periodic{p} and continuous.
\end{theorem}


\begin{remark}
  \label{rem:simple-geom}

  Using a disc as region $U$, e.g.\ $U = \ball{0}{1}$, yields that $\psi_\eta$
  is \periodic{1} and continuous for all valid $\tau$. The reason is that the
  condition $e^{i \theta p} U = U$ is then clearly fulfilled for every $p$,
  especially for $p=1$. 
  
  The parameter $\delta$ is $1$ for simple geometries like a disc or a
  polygon. See Proposition~\vref{pro:simple-geometries} for details.
\end{remark}


\begin{remark}
  \label{rem:delta2}

  If $\delta=2$ in the theorem, then the statement stays true, but degenerates
  to
  \begin{equation*}
    \f{Z_{\tau,w,\eta}}{N} = \Oh{N^2 \log_{\abs\tau} N}.
  \end{equation*}
\end{remark}


The proof of Theorem~\vref{thm:countdigits} follows the ideas used by
Delange~\cite{Delange:1975:chiffres}. By Remark~\vref{rem:delta2} we restrict
ourselves to the case $\delta<2$.

We will use the following abbreviations. We set $\f{Z}{N} :=
\f{Z_{\tau,w,\eta}}{N}$, and we set $W:=W_\eta$ and $W_j:=W_{\eta,j}$ for our
fixed $\eta$ of Theorem~\vref{thm:countdigits}. Further we set
$\beta_j:=\beta_{\eta,j}$, cf.\ Proposition~\vref{pro:prop-of-w}.  By $\log$ we
will denote the logarithm to the base $\abs{\tau}$, i.e., $\log
x=\log_{\abs\tau} x$. These abbreviations will be used throughout the remaining
section.


\begin{proof}[Proof of Theorem~\ref{thm:countdigits}]
  We know from Theorem~\vref{thm:wnaf-exist-unique} that every element of
  $\Ztau$ is represented by a unique element of $\wNAFsetfin$. To count the
  occurrences of the digit $\eta$ in $NU$, we sum up over all lattice points $n
  \in NU \cap \Ztau$ and for each $n$ over all digits in the corresponding
  \wNAF{} equal to $\eta$. Thus we get
  \begin{equation*}
    \f{Z}{N} 
    = \sum_{n \in NU \cap \Ztau}
    \sum_{j\in\N_0} \iverson {\f{\eps_j}{\NAFw{n}} = \eta},
  \end{equation*}
  where $\eps_j$ denotes the extraction of the $j$th digit, i.e., for a
  \wNAF{} $\bfxi$ we define $\f{\eps_j}{\bfxi} := \xi_j$. The inner sum over
  $j\in\N_0$ is finite, we will choose a large enough upper bound $J$ later in
  Lemma~\vref{lem:choosing-j}.

  Using
  \begin{equation*}
    \iverson{\f{\eps_j}{\NAFw{n}} = \eta}
    = \indicator{W_j}{\fracpartZtau{\frac{n}{\tau^{j+w}}}}
  \end{equation*}
  from Lemma~\vref{lem:equiv-digit-char-set} yields
  \begin{equation*}
      \f{Z}{N} = \sum_{j=0}^J \sum_{n\in NU \cap \Ztau} 
      \indicator{W_j}{\fracpartZtau{\frac{n}{\tau^{j+w}}}},
  \end{equation*}
  where additionally the order of summation was changed. This enables us to
  rewrite the sum over $n$ as integral
  \begin{equation*}
    \begin{split}
      \f{Z}{N} &= \sum_{j=0}^J \sum_{n\in NU \cap \Ztau} 
      \frac{1}{\lmeas{V_n}} \int_{x\in V_n}
      \indicator{W_j}{\fracpartZtau{\frac{x}{\tau^{j+w}}}} \dd x \\
      &= \frac{1}{\lmeas{V}} \sum_{j=0}^J \int_{x\in \coverV{NU}} 
      \indicator{W_j}{\fracpartZtau{\frac{x}{\tau^{j+w}}}} \dd x.
    \end{split}
  \end{equation*}
  We split up the integrals into the ones over $NU$ and others over the
  remaining region and get
  \begin{equation*}
    \f{Z}{N} = \frac{1}{\lmeas{V}} \sum_{j=0}^J \int_{x \in NU} 
    \indicator{W_j}{\fracpartZtau{\frac{x}{\tau^{j+w}}}} \dd x 
    + \f{\cF_\eta}{N},
  \end{equation*}
  in which $\f{\cF_\eta}{N}$ contains all integrals (with appropriate signs)
  over regions $\coverV{NU} \setminus NU$ and $NU \setminus \coverV{NU}$.

  Substituting $x = \tau^J y$, $\dd x = \abs\tau^{2J}\dd y$ we obtain
  \begin{equation*}
    \f{Z}{N} = \frac{\abs\tau^{2J}}{\lmeas{V}} \sum_{j=0}^J 
    \int_{y \in \tau^{-J}NU} \indicator{W_j}{\fracpartZtau{y \tau^{J-j-w}}} \dd y  
    + \f{\cF_\eta}{N}.
  \end{equation*}
  Reversing the order of summation yields
  \begin{equation*}
    \f{Z}{N} = \frac{\abs\tau^{2J}}{\lmeas{V}} \sum_{j=0}^J 
    \int_{y \in \tau^{-J}NU} 
    \indicator{W_{J-j}}{\fracpartZtau{y \tau^{j-w}}} \dd y 
    + \f{\cF_\eta}{N}.
  \end{equation*}
  We rewrite this as
  \begin{equation*}
    \begin{split}
      \f{Z}{N} 
      &= \frac{\abs\tau^{2J}}{\lmeas{V}} (J+1) 
      \lmeas{W} \int_{y \in \tau^{-J}NU} \dd y \\
      &+ \frac{\abs\tau^{2J}}{\lmeas{V}} \sum_{j=0}^J \int_{y \in \tau^{-J}NU} 
      \left( \indicator{W}{\fracpartZtau{y \tau^{j-w}}} 
        - \lmeas{W} \right) \dd y \\
      &+ \frac{\abs\tau^{2J}}{\lmeas{V}} \sum_{j=0}^J \int_{y \in \tau^{-J}NU} 
      \left( \indicator{W_{J-j}}{\fracpartZtau{y \tau^{j-w}}}
        - \indicator{W}{\fracpartZtau{y \tau^{j-w}}} \right) \dd y \\
      &+ \f{\cF_\eta}{N}.    
    \end{split}
  \end{equation*}
  \begin{figure}
    \centering
    \includegraphics{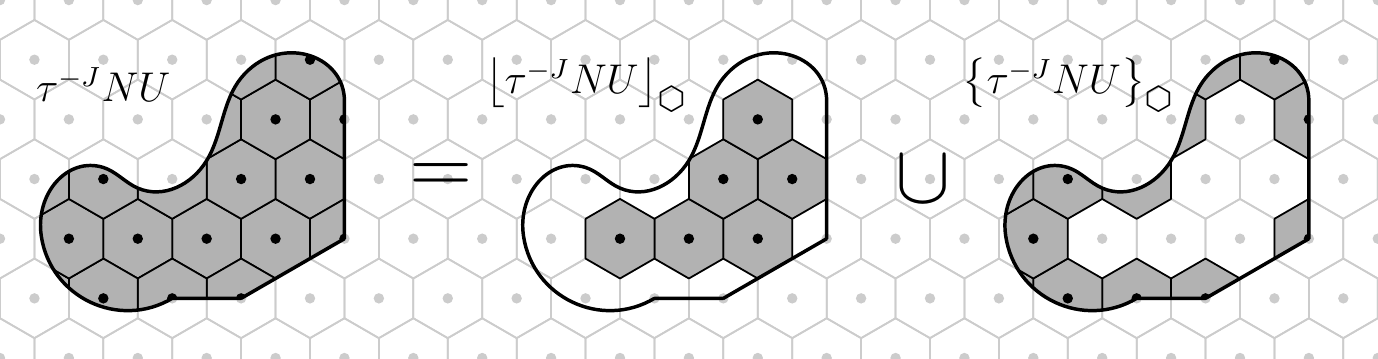}
    \caption{Splitting up the region of integration $\tau^{-J}NU$.}
    \label{fig:region-splitting}
  \end{figure}
  With $\tau^{-J}NU = \floorV[j-w]{\tau^{-J}NU} \cup
  \fracpartV[j-w]{\tau^{-J}NU}$, see Figure~\vref{fig:region-splitting}, for
  each integral region we get
  \begin{equation*}
    \f{Z}{N} 
    = \f{\cM_\eta}{N} 
    + \f{\cZ_\eta}{N} 
    + \f{\cP_\eta}{N} 
    + \f{\cQ_\eta}{N} 
    + \f{\cS_\eta}{N}
    + \f{\cF_\eta}{N},
  \end{equation*}
  in which $\cM_\eta$ is \emph{``The Main Part''}, see
  Lemma~\vref{lem:the-main-part},
  \begin{subequations}
    \label{eq:mainth:parts}
    \begin{align}
      \label{eq:mainth:main-part}
      \f{\cM_\eta}{N} 
      &= \frac{\abs\tau^{2J}}{\lmeas{V}} (J+1) \lmeas{W} 
      \int_{y \in \tau^{-J}NU} \dd y,
      \intertext{$\cZ_\eta$ is \emph{``The Zero Part''}, see 
        Lemma~\vref{lem:the-zero-part},}
      \label{eq:mainth:zero-part}
      \f{\cZ_\eta}{N} 
      &= \frac{\abs\tau^{2J}}{\lmeas{V}} \sum_{j=0}^J 
      \int_{y \in \floorV[j-w]{\tau^{-J}NU}}
      \left( \indicator{W}{\fracpartZtau{y \tau^{j-w}}} 
        - \lmeas{W} \right) \dd y,
      \intertext{$\cP_\eta$ is \emph{``The Periodic Part''}, see 
        Lemma~\vref{lem:the-periodic-part},}
      \label{eq:mainth:periodic-part}
      \f{\cP_\eta}{N} 
      &= \frac{\abs\tau^{2J}}{\lmeas{V}} \sum_{j=0}^J 
      \int_{y \in \fracpartV[j-w]{\tau^{-J}NU}}
      \left( \indicator{W}{\fracpartZtau{y \tau^{j-w}}} 
        - \lmeas{W} \right) \dd y,
      \intertext{$\cQ_\eta$ is \emph{``The Other Part''}, see 
        Lemma~\vref{lem:the-other-part},}
      \label{eq:mainth:other-part}
      \f{\cQ_\eta}{N} 
      &= \frac{\abs\tau^{2J}}{\lmeas{V}} \sum_{j=0}^J 
      \int_{y \in \floorV[j-w]{\tau^{-J}NU}} 
      \f{\left( \indicator*{W_{J-j}} - \indicator*{W} 
        \right)}{\fracpartZtau{y \tau^{j-w}}} \dd y,
      \intertext{$\cS_\eta$ is \emph{``The Small Part''}, see 
        Lemma~\vref{lem:the-small-part},}
      \label{eq:mainth:small-part}
      \f{\cS_\eta}{N}
      &= \frac{\abs\tau^{2J}}{\lmeas{V}} \sum_{j=0}^J 
      \int_{y \in \fracpartV[j-w]{\tau^{-J}NU}} 
      \f{\left( \indicator*{W_{J-j}} - \indicator*{W} 
        \right)}{\fracpartZtau{y \tau^{j-w}}} \dd y
      \intertext{and $\cF_\eta$ is \emph{``The Fractional Cells Part''}, see 
          Lemma~\vref{lem:the-fractionalcells-part},}
      \label{eq:mainth:fraccells-part}
      \begin{split}
        \f{\cF_\eta}{N}
        &= \frac{1}{\lmeas{V}} \sum_{j=0}^J
        \int_{x \in \coverV{NU} \setminus NU} 
        \indicator{W_j}{\fracpartZtau{\frac{x}{\tau^{j+w}}}} \dd x \\
        &\phantom{=}- \frac{1}{\lmeas{V}} \sum_{j=0}^J
        \int_{x \in NU \setminus \coverV{NU}} 
        \indicator{W_j}{\fracpartZtau{\frac{x}{\tau^{j+w}}}} \dd x.
      \end{split}
    \end{align}
  \end{subequations}

  To complete the proof we have to deal with the choice of $J$, see
  Lemma~\vref{lem:choosing-j}, as well as with each of the parts in
  \eqref{eq:mainth:parts}, see
  Lemmata~\vrefrange{lem:the-main-part}{lem:the-fractionalcells-part}. The
  continuity of $\psi_\eta$ is checked in Lemma~\vref{lem:psi-continous}.
\end{proof}


\begin{lemma}[Choosing $J$]
  \label{lem:choosing-j}
  
  Let $N \in \R_{\geq0}$. Then every \wNAF{} of $\wNAFsetfin$ with value in $NU$
  has at most $J+1$ digits, where
  \begin{equation*}
    J = \floor{\log N} + c
  \end{equation*}
  with
  \begin{equation*}
    c = \floor{ \log d - \log f_L } + 1
  \end{equation*}
  with $f_L$ of Proposition~\vref{pro:lower-bound-fracnafs}.
\end{lemma}


\begin{proof}
  Let $z \in NU$, $z\neq0$, with its corresponding \wNAF{}
  $\bfxi\in\wNAFsetfin$, and let $j\in\N_0$ be the largest index, such that the
  digit $\xi_j$ is non-zero. By using Corollary~\vref{cor:bounds-value}, we
  conclude that
  \begin{equation*}
    \abs{\tau}^j f_L \leq \abs{z} < Nd.
  \end{equation*}
  This means
  \begin{equation*}
    j < \log N + \log d - \log f_L,
  \end{equation*}
  and thus we have
  \begin{equation*}
    j \leq \floor{\log N + \log d - \log f_L} 
    \leq \floor{\log N} + \floor{\log d - \log f_L} + 1.
  \end{equation*}
  Defining the right hand side of this inequality as $J$ finishes the proof.
\end{proof}


\begin{remark}
  \label{rem:auxcalc-taujn}
  
  For the parameter used in the region of integration in the proof of
  Theorem~\vref{thm:countdigits} we get
  \begin{equation*}
    \tau^{-J} N = \abs\tau^{\fracpart{\log N}-c} \f{\wh{\theta}}{\log N},
  \end{equation*}
  with the rotation $\f{\wh{\theta}}{x} = e^{-i \theta \floor{x} - i \theta
    c}$. In particular we get $\abs{\tau^{-J} N} = \Oh{1}$.
\end{remark}




\begin{remark}
  \label{rem:auxcalc-gammaj}

  Let $\gamma\in\R$ with $\gamma \geq 1$, then
  \begin{equation*}
    \gamma^J = N^{\log \gamma} \gamma^{c-\fracpart{\log N}} 
    = \Oh{N^{\log \gamma}}.
  \end{equation*}
  In particular $\abs\tau^{2J} = \Oh{N^2}$ and $\abs\tau^J =
  \Oh{N}$.
\end{remark}




\begin{lemma}[The Main Part]
  \label{lem:the-main-part}

  For \eqref{eq:mainth:main-part} in the proof of Theorem~\vref{thm:countdigits}
  we get
  \begin{equation*}
    \f{\cM_\eta}{N}
    = e_w  N^2 \lmeas{U} \log N
    + N^2 \f{\psi_{\eta,\cM}}{\log N} 
  \end{equation*}
  with a \periodic{1} function $\psi_{\eta,\cM}$,
  \begin{equation*}  
    \f{\psi_{\eta,\cM}}{x} 
    = \lmeas{U} \left( c + 1 - \fracpart{x} \right) e_w
  \end{equation*}
  and
  \begin{equation*}
    e_w = \frac{1}{\abs\tau^{2(w-1)}
      \left(\left(\abs\tau^2-1\right) w + 1\right)}.
  \end{equation*}
\end{lemma}


\begin{proof}
  We have
  \begin{equation*}
    \f{\cM_\eta}{N}
    = \frac{\abs\tau^{2J}}{\lmeas{V}} (J+1) \lmeas{W} 
    \int_{y \in \tau^{-J}NU} \dd y.
  \end{equation*}
  As $\lmeas{\tau^{-J}NU} = \abs\tau^{-2J}N^2\lmeas{U}$ we obtain
  \begin{equation*}
    \f{\cM_\eta}{N}
    = \frac{\lmeas{W}}{\lmeas{V}} (J+1) N^2 \lmeas{U}.
  \end{equation*}
  By taking $\lmeas{W} = \lmeas{V} e_w$ from \itemref{enu:prop-of-w:area-w} of
  Proposition~\vref{pro:prop-of-w} and $J$ from Lemma~\vref{lem:choosing-j} we
  get
  \begin{equation*}
    \f{\cM_\eta}{N}
    = N^2 \lmeas{U} e_w \left( \floor{\log N} 
        + c + 1 \right).
  \end{equation*}
  Finally, the desired result follows by using $x = \floor{x} + \fracpart{x}$.
\end{proof}


\begin{lemma}[The Zero Part]
  \label{lem:the-zero-part}

  For \eqref{eq:mainth:zero-part} in the proof of Theorem~\vref{thm:countdigits}
  we get
  \begin{equation*}
    \f{\cZ_\eta}{N} = 0.
  \end{equation*}
\end{lemma}


\begin{proof}
  Consider the integral
  \begin{equation*}    
      I_j
      := \int_{y \in \floorV[j-w]{\tau^{-J}NU}}
      \left( \indicator{W}{\fracpartZtau{y \tau^{j-w}}} 
        - \lmeas{W} \right) \dd y.
  \end{equation*}
  We can rewrite the region of integration as
  \begin{equation*}
    \floorV[j-w]{\tau^{-J}NU} 
    = \frac{1}{\tau^{j-w}} \floorV{\tau^{j-w} \tau^{-J}NU}
    = \frac{1}{\tau^{j-w}} \bigcup_{z\in T_{j-w}} V_z
  \end{equation*}
  for some appropriate $T_{j-w} \subseteq \Ztau$. Substituting $x =
  \tau^{j-w} y$, $\dd x = \abs\tau^{2(j-w)} \dd y$ yields
  \begin{equation*}
      I_j
      = \frac{1}{\abs\tau^{2(j-w)}} \int_{x \in \bigcup_{z\in T_{j-w}} V_z}
      \left( \indicator{W}{\fracpartZtau{x}} 
        - \lmeas{W} \right) \dd x.    
  \end{equation*}
  We split up the integral and eliminate the fractional part $\fracpartZtau{x}$
  by translation to get
  \begin{equation*}
      I_j
      = \frac{1}{\abs\tau^{2(j-w)}} \sum_{z\in T_{j-w}} 
      \underbrace{\int_{x \in V} \left( \indicator{W}{x}  
          - \lmeas{W} \right) \dd x}_{=0}.    
  \end{equation*}
  Thus, for all $j\in\N_0$ we obtain $I_j=0$, and therefore
  $\f{\cZ_\eta}{N}=0$.
\end{proof}


\begin{lemma}[The Periodic Part]
  \label{lem:the-periodic-part}

  For \eqref{eq:mainth:periodic-part} in the proof of
  Theorem~\vref{thm:countdigits} we get
  \begin{equation*}
    \f{\cP_\eta}{N} = 
    N^2 \f{\psi_{\eta,\cP}}{\log N} + \Oh{N^\delta}
  \end{equation*}
  with a function $\psi_{\eta,\cP}$,
  \begin{equation*}
    \f{\psi_{\eta,\cP}}{x} =
    \frac{\abs\tau^{2(c-\fracpart{x})}}{\lmeas{V}} \sum_{j=0}^\infty 
    \int_{y \in \fracpartV[j-w]{\abs\tau^{\fracpart{x}-c} \f{\wh{\theta}}{\floor{x}} U}}
    \left( \indicator{W}{\fracpartZtau{y \tau^{j-w}}} - \lmeas{W} \right) \dd y,
  \end{equation*}
  with the rotation $\f{\wh{\theta}}{x} = e^{-i \theta x - i \theta c}$.

  If there is a $p\in\N$, such that $e^{i \theta p} U = U$, then
  $\psi_{\eta,\cP}$ is \periodic{p}.
\end{lemma}


\begin{proof}
  Consider
  \begin{equation*}
    I_j
    := \int_{y \in \fracpartV[j-w]{\tau^{-J}NU}}
      \left( \indicator{W}{\fracpartZtau{y \tau^{j-w}}}
        - \lmeas{W} \right) \dd y.
  \end{equation*}
  The region of integration satisfies
  \begin{equation}
    \label{eq:proof-periodic-part:region}
    \begin{split}
      \fracpartV[j-w]{\tau^{-J}NU} 
      \subseteq \boundaryV[j-w]{\tau^{-J}NU} 
      = \frac{1}{\tau^{j-w}} \bigcup_{z\in T_{j-w}} V_z
    \end{split}
  \end{equation}
  for some appropriate $T_{j-w} \subseteq \Ztau$. 
  
  We use the triangle inequality and substitute $x = \tau^{j-w} y$, $\dd x =
  \abs\tau^{2(j-w)} \dd y$ in the integral to get
  \begin{equation*}
    \abs{I_j} \leq 
    \frac{1}{\abs\tau^{2(j-w)}} \int_{x \in \bigcup_{z\in T_{j-w}} V_z}
      \underbrace{\abs{ \indicator{W}{\fracpartZtau{x}} 
          - \lmeas{W}}}_{\leq 1+\lmeas{W}} \dd x.
  \end{equation*}
  After splitting up the integral and using translation to eliminate the
  fractional part, we get
  \begin{equation*}
    \abs{I_j} \leq  \frac{1+\lmeas{W}}{\abs\tau^{2(j-w)}} \
    \sum_{z\in T_{j-w}} \int_{x \in V} \dd x 
    = \frac{1+\lmeas{W}}{\abs\tau^{2(j-w)}} \lmeas{V} \card{T_{j-w}}.
  \end{equation*}
  Using $\cardV{\boundaryV{NU}} = \Oh{N^\delta}$ as assumed and
  Equation~\eqref{eq:proof-periodic-part:region} we gain
  \begin{equation*}
    \card{T_{j-w}} = \Oh{\abs{\tau}^{(j-w)\delta}\abs{\tau^{-J}N}^\delta}
    = \Oh{\abs{\tau}^{(j-w)\delta}},
  \end{equation*}
  because $\abs{\tau^{-J}N} = \Oh{1}$, see Remark~\vref{rem:auxcalc-taujn}, and
  thus
  \begin{equation*}
    \abs{I_j} = \Oh{\abs\tau^{\delta (j-w) - 2(j-w)}} 
    = \Oh{\abs\tau^{(\delta-2)j}}.
  \end{equation*}

  Now we want to make the summation in $\cP_\eta$ independent from
  $J$, so we consider 
  \begin{equation*}
    I := \frac{\abs\tau^{2J}}{\lmeas{V}} \sum_{j=J+1}^\infty I_j
  \end{equation*}
  Again we use triangle inequality and we calculate the sum to obtain
  \begin{equation*}
    \abs{I} 
    = \Oh{\abs\tau^{2J}} \sum_{j=J+1}^\infty \Oh{\abs\tau^{(\delta-2) j}}
    = \Oh{\abs\tau^{2J} \abs\tau^{(\delta-2) J}}
    = \Oh{\abs\tau^{\delta J}}.
  \end{equation*}
  Note that $\Oh{\abs\tau^J} = \Oh{N}$, see Remark~\vref{rem:auxcalc-gammaj}, so
  we obtain $\abs{I} = \Oh{N^\delta}$.

  Let us look at the growth of 
  \begin{equation*}
    \f{\cP_\eta}{N} 
    = \frac{\abs\tau^{2J}}{\lmeas{V}} \sum_{j=0}^J I_j.
  \end{equation*}
  We get
  \begin{equation*}
    \abs{\f{\cP_\eta}{N}} 
    = \Oh{\abs\tau^{2J}} \sum_{j=0}^J \Oh{\abs\tau^{(\delta-2) j}}
    = \Oh{\abs\tau^{2J}} = \Oh{N^2},
  \end{equation*}
  using $\delta<2$, and, to get the last equality,
  Remark~\vref{rem:auxcalc-gammaj}.

  Finally, inserting the result of Remark~\vref{rem:auxcalc-taujn} for the
  region of integration, rewriting $\abs\tau^{2J}$ according to
  Remark~\vref{rem:auxcalc-gammaj} and extending the sum to infinity, as above
  described, yields
  \begin{equation*}
    \begin{split}
      \f{\cP_\eta}{N}
      &= \frac{\abs\tau^{2J}}{\lmeas{V}} \sum_{j=0}^J 
      \int_{y \in \fracpartV[j-w]{\tau^{-J}NU}}
      \left( \indicator{W}{\fracpartZtau{y \tau^{j-w}}} 
        - \lmeas{W} \right) \dd y \\
      &= N^2 \underbrace{
        \frac{\abs\tau^{2(c-\fracpart{\log N})}}{\lmeas{V}} 
        \sum_{j=0}^\infty \int_{y \in 
          \fracpartV[j-w]{\abs\tau^{\fracpart{\log N}-c} 
            \f{\wh{\theta}}{\floor{\log N}} U}}
      \left( \indicator{W}{\fracpartZtau{y \tau^{j-w}}} 
        - \lmeas{W} \right) 
      \dd y}_{=: \f{\psi_{\eta,\cP}}{\log N}} 
    \\ &\phantom{=}+ \Oh{N^\delta},
    \end{split}
  \end{equation*}
  with the rotation $\f{\wh{\theta}}{x} = e^{-i \theta x - i \theta c}$.

  Now let
  \begin{equation*}
    e^{i \theta p} U = U
    \equivalent
    e^{- i \theta p} U = e^{- i \theta 0} U.
  \end{equation*}
  Clearly the region of integration in $\f{\psi_{\eta,\cP}}{x}$ is \periodic{p},
  since $x$ occurs as $\fracpart{x}$ and $\floor{x}$. All other occurrences of
  $x$ are of the form $\fracpart{x}$, i.e., \periodic{1}, so period $p$ is
  obtained.
\end{proof}


\begin{lemma}[The Other Part]
  \label{lem:the-other-part}

  For \eqref{eq:mainth:other-part} in the proof of
  Theorem~\vref{thm:countdigits} we get
  \begin{equation*}
    \f{\cQ_\eta}{N} = N^2 \psi_{\eta,\cQ}
    + \Oh{N^{\alpha} \log N}  + \Oh{N^\delta},
  \end{equation*}
  with
  \begin{equation*}
    \psi_{\eta,\cQ} = \frac{\lmeas{U}}{\lmeas{V}}
    \sum_{j=0}^\infty \frac{\beta_j}{\lmeas{V}}
  \end{equation*}
  and $\alpha = 2+\log\rho<2$, where $\rho<1$ can be found in
  Theorem~\vref{th:w-naf-distribution}.
\end{lemma}


\begin{proof}
  Consider
  \begin{equation*}
      I_{j,\ell}
      := \int_{y \in \floorV[j-w]{\tau^{-J}NU}} 
      \f{\left( \indicator*{W_{\eta,\ell}} - \indicator*{W} 
        \right)}{\fracpartZtau{y \tau^{j-w}}} \dd y.
  \end{equation*}
  We can rewrite the region of integration and get
  \begin{equation*}
    \floorV[j-w]{\tau^{-J}NU} 
    = \frac{1}{\tau^{j-w}} \floorV{\tau^{j-w} \tau^{-J}NU}
    = \frac{1}{\tau^{j-w}} \bigcup_{z\in T_{j-w}} V_z
  \end{equation*}
  for some appropriate $T_{j-w} \subseteq \Ztau$, as in the proof of
  Lemma~\vref{lem:the-zero-part}. Substituting $x = \tau^{j-w} y$, $\dd x =
  \abs\tau^{2(j-w)} \dd y$ yields
  \begin{equation*}
      I_{j,\ell}
      = \frac{1}{\abs\tau^{2(j-w)}} \int_{x \in \bigcup_{z\in T_{j-w}} V_z} 
      \f{\left( \indicator*{W_{\eta,\ell}} - \indicator*{W} 
        \right)}{\fracpartZtau{x}} \dd x
  \end{equation*}
  and further
  \begin{equation*}
      I_{j,\ell}
      = \frac{1}{\abs\tau^{2(j-w)}} \sum_{z\in T_{j-w}} 
      \underbrace{\int_{x \in V} \f{\left( \indicator*{W_{\eta,\ell}} 
            - \indicator*{W} \right)}{x} \dd x}_{= \beta_\ell}
      = \frac{1}{\abs\tau^{2(j-w)}} \card{T_{j-w}} \beta_\ell,
  \end{equation*}
  by splitting up the integral, using translation to eliminate the fractional
  part and taking $\beta_\ell$ according to \itemref{enu:prop-of-w:beta} of
  Proposition~\vref{pro:prop-of-w}.  From Proposition~\vref{pro:set-nu} we
  obtain
  \begin{equation*}
    \frac{\card{T_{j-w}}}{\abs\tau^{2(j-w)}}
    = \frac{\abs{\tau^{j-w} \tau^{-J} N}^2}{\abs\tau^{2(j-w)}} 
    \frac{\lmeas{U}}{\lmeas{V}}
    + \Oh{\frac{\abs{\tau^{j-w} \tau^{-J} N}^\delta}{\abs\tau^{2(j-w)}}}
    = \abs{\tau^{-J} N}^2 \frac{\lmeas{U}}{\lmeas{V}}
    + \Oh{\abs\tau^{(\delta-2) j}},
  \end{equation*}
  because $\abs{\tau^{-J}N} = \Oh{1}$, see Remark~\vref{rem:auxcalc-taujn}. 

  Now let us have a look at
  \begin{equation*}
      \f{\cQ_\eta}{N} 
      = \frac{\abs\tau^{2J}}{\lmeas{V}} \sum_{j=0}^J I_{j,J-j}.
  \end{equation*}
  Inserting the result above and using $\beta_{\ell} = \Oh{\rho^\ell}$, see
  \itemref{enu:prop-of-w:beta} of Proposition~\vref{pro:prop-of-w}, yields
  \begin{equation*}
    \f{\cQ_\eta}{N} 
    = \abs\tau^{2J} \abs{\tau^{-J} N}^2
    \frac{\lmeas{U}}{\lmeas{V}}
    \sum_{j=0}^J \frac{\beta_{J-j}}{\lmeas{V}}
    + \abs\tau^{2J} \sum_{j=0}^J 
    \Oh{\abs\tau^{(\delta-2)j}} \Oh{\rho^{J-j}}
  \end{equation*}
  We notice that $\abs\tau^{2J}\abs{\tau^{-J}N}^2 = N^2$.

  Therefore, after reversing the order of the first summation, we obtain
  \begin{equation*}
    \f{\cQ_\eta}{N} 
    = N^2 \frac{\lmeas{U}}{\lmeas{V}}
    \sum_{j=0}^J \frac{\beta_{j}}{\lmeas{V}}
    + \abs\tau^{2J} \rho^J \sum_{j=0}^J 
    \Oh{\left(\rho\abs\tau^{2-\delta}\right)^{-j}}.
  \end{equation*}
  If $\rho\abs\tau^{2-\delta}\geq1$, then the second sum is $J \Oh{1}$,
  otherwise the sum is $\Oh{\rho^{-J} \abs\tau^{(\delta-2)J}}$. So we obtain
  \begin{equation*}
    \f{\cQ_\eta}{N} 
    = N^2 \frac{\lmeas{U}}{\lmeas{V}}
    \sum_{j=0}^J \frac{\beta_{j}}{\lmeas{V}}
    + \Oh{\abs\tau^{2J} \rho^J J}
    + \Oh{\abs\tau^{\delta J}}.
  \end{equation*}
  Using $J = \f{\Theta}{\log N}$, see Lemma~\vref{lem:choosing-j},
  Remark~\vref{rem:auxcalc-gammaj}, and defining $\alpha = 2 + \log\rho$ yields
  \begin{equation*}
    \f{\cQ_\eta}{N} 
    = N^2 \frac{\lmeas{U}}{\lmeas{V}}
    \sum_{j=0}^J \frac{\beta_{j}}{\lmeas{V}}
    + \underbrace{\Oh{N^{2 + \log\rho} \log N}}_{=\Oh{N^\alpha \log N}}
    + \Oh{N^\delta}.
  \end{equation*}

  Now consider the first sum. Since $\beta_j = \Oh{\rho^j}$, see
  \itemref{enu:prop-of-w:beta} of Proposition~\vref{pro:prop-of-w}, we obtain
  \begin{equation*}
    N^2 \sum_{j=J+1}^\infty \beta_j = N^2 \Oh{\rho^J} 
    = \Oh{N^\alpha}.
  \end{equation*}
  Thus the lemma is proved, because we can extend the sum to infinity.
\end{proof}


\begin{lemma}[The Small Part]
  \label{lem:the-small-part}

  For \eqref{eq:mainth:small-part} in the proof of
  Theorem~\vref{thm:countdigits} we get
  \begin{equation*}
    \f{\cS_\eta}{N} = \Oh{N^{\alpha} \log N} + \Oh{ N^\delta }
  \end{equation*}  
  with $\alpha = 2 + \log\rho<2$ and $\rho<1$ from
  Theorem~\vref{th:w-naf-distribution}.
\end{lemma}


\begin{proof}
  Consider
  \begin{equation*}
    I_{j,\ell}
    := \int_{y \in \fracpartV[j-w]{\tau^{-J}NU}} 
    \f{\left( \indicator*{W_\ell} - \indicator*{W} 
      \right)}{\fracpartZtau{y \tau^{j-w}}} \dd y.
    \end{equation*}
    Again, as in the proof of Lemma~\vref{lem:the-periodic-part}, the region of
    integration satisfies
  \begin{equation}
    \label{eq:proof-small-part:region}
    \fracpartV[j-w]{\tau^{-J}NU} 
    \subseteq \boundaryV[j-w]{\tau^{-J}NU}
    = \frac{1}{\tau^{j-w}} \bigcup_{z\in T_{j-w}} V_z,
  \end{equation}
  for some appropriate $T_{j-w} \subseteq \Ztau$. 

  We substitute $x = \tau^{j-w} y$, $\dd x = \abs\tau^{2(j-w)} \dd y$ in the
  integral to get
  \begin{equation*}
    \abs{I_{j,\ell}} =
    \frac{1}{\abs\tau^{2(j-w)}} \abs{\int_{x \in \bigcup_{z\in T_{j-w}} V_z}
    \f{\left(\indicator*{W_\ell} 
      - \indicator*{W}\right)}{\fracpartZtau{x}} \dd x}.
  \end{equation*}
  Again, after splitting up the integral, using translation to eliminate the
  fractional part and the triangle inequality, we get
  \begin{equation*}
    \abs{I_{j,\ell}} \leq  \frac{1}{\abs\tau^{2(j-w)}}
    \sum_{z\in T_{j-w}} \underbrace{\abs{\int_{x \in V}
      \f{\left(\indicator*{W_\ell} - \indicator*{W}\right)}{x} 
      \dd x}}_{= \abs{\beta_\ell}}
    = \frac{1}{\abs\tau^{2(j-w)}} \card{T_{j-w}} \abs{\beta_\ell} ,
  \end{equation*}
  in which $\abs{\beta_\ell} = \Oh{\rho^\ell}$ is known from
  \itemref{enu:prop-of-w:beta} of Proposition~\vref{pro:prop-of-w}. Using
  $\cardV{\boundaryV{NU}} = \Oh{N^\delta}$, Remark~\vref{rem:auxcalc-taujn}, and
  Equation~\eqref{eq:proof-small-part:region} we get
  \begin{equation*}
    \card{T_{j-w}} = \Oh{\abs{\tau}^{(j-w)\delta} \abs{\tau^{-J}N}^\delta} 
    = \Oh{\abs\tau^{\delta (j-w)}},
  \end{equation*}
  because $\abs{\tau^{-J}N} = \Oh{1}$. Thus
  \begin{equation*}
    \abs{I_{j,\ell}} 
    = \Oh{ \rho^\ell \abs\tau^{(\delta - 2)(j-w)} } 
    = \Oh{ \rho^\ell \abs\tau^{(\delta-2) j} } 
  \end{equation*}
  follows by assembling all together.

  Now we are ready to analyse
  \begin{equation*}
    \f{\cS_\eta}{N}
    = \frac{\abs\tau^{2J}}{\lmeas{V}} \sum_{j=0}^J I_{j,J-j}.    
  \end{equation*}
  Inserting the result above yields
  \begin{equation*}
    \abs{\f{\cS_\eta}{N}}
    = \frac{\abs\tau^{2J}}{\lmeas{V}} \sum_{j=0}^J 
    \Oh{ \rho^{J-j} \abs\tau^{(\delta-2) j} } 
    = \frac{\rho^J \abs\tau^{2J}}{\lmeas{V}} \sum_{j=0}^J 
    \Oh{ \left( \rho \abs\tau^{2-\delta}\right)^{-j} } 
  \end{equation*}
  and thus, by the same argument as in the proof of
  Lemma~\vref{lem:the-other-part},
  \begin{equation*}
    \abs{\f{\cS_\eta}{N}}
    = \rho^J \abs\tau^{2J} \Oh{J + \rho^{-J} \abs\tau^{(\delta - 2)J}} 
    = \Oh{\rho^J \abs\tau^{2J} J} + \Oh{ \abs\tau^{\delta J} },
  \end{equation*}
  Finally, using Lemma~\vref{lem:choosing-j} and
  Remark~\vref{rem:auxcalc-gammaj}, we obtain
  \begin{equation*}
    \abs{\f{\cS_\eta}{N}} = \Oh{N^{\alpha} \log N} + \Oh{ N^\delta }
  \end{equation*}
  with $\alpha = 2+\log\rho$. Since $\rho<1$, we have $\alpha<2$.
\end{proof}


\begin{lemma}[The Fractional Cells Part]
  \label{lem:the-fractionalcells-part}

  For \eqref{eq:mainth:fraccells-part} in the proof of
  Theorem~\vref{thm:countdigits} we get
  \begin{equation*}
    \f{\cF_\eta}{N} = \Oh{N^\delta \log N}
  \end{equation*}    
\end{lemma}


\begin{proof}
  For the regions of integration in $\cF_\eta$ we obtain
  \begin{subequations}
    \begin{align*}
      NU \setminus \coverV{NU}
      &\subseteq \ceilV{NU} \setminus \floorV{NU}
      = \boundaryV{NU}
      = \bigcup_{z\in T} V_z
      \intertext{and}
      \coverV{NU} \setminus NU
      &\subseteq \ceilV{NU} \setminus \floorV{NU}
      = \boundaryV{NU}
      = \bigcup_{z\in T} V_z  
    \end{align*}
  \end{subequations}
  for some appropriate $T \subseteq \Ztau$ using
  Proposition~\vref{pro:round-v-basic-prop}. Thus we get
  \begin{equation*}
    \abs{\f{\cF_\eta}{N}} 
    \leq \frac{2}{\lmeas{V}} \sum_{j=0}^J
    \int_{x \in \bigcup_{z\in T} V_z} 
    \indicator{W_j}{\fracpartZtau{\frac{x}{\tau^{j+w}}}} \dd x
    \leq \frac{2}{\lmeas{V}} \sum_{j=0}^J \sum_{z\in T}
    \int_{x \in V_z} \dd x,
  \end{equation*}
  in which the indicator function was replaced by $1$. Dealing with the sums and
  the integral, which is $\Oh{1}$, we obtain
  \begin{equation*}
    \abs{\f{\cF_\eta}{N}} = (J+1) \card*{T} \Oh{1}.
  \end{equation*}
  Since $J = \Oh{\log N}$, see Lemma~\vref{lem:choosing-j}, and
  $\card*{T} = \Oh{N^\delta}$, the desired result
  follows.
\end{proof}


\begin{lemma}
  \label{lem:psi-continous}

  If the $\psi_\eta$ from Theorem~\vref{thm:countdigits} is \periodic{p},
  then $\psi_\eta$ is also continuous.
\end{lemma}


\begin{proof}
  There are two possible parts of $\psi_\eta$, where an discontinuity could
  occur. The first is $\fracpart{x}$ for an $x\in\Z$, the second is building
  $\fracpartV[j-w]{\dots}$ in the region of integration in $\psi_{\eta,\cP}$.

  The latter is no problem, i.e., no discontinuity, since 
  \begin{multline*}
    \int_{y \in \fracpartV[j-w]{\abs\tau^{\fracpart{x}-c} \f{\wh{\theta}}{\floor{x}} U}}
    \left( \indicator{W}{\fracpartZtau{y \tau^{j-w}}} - \lmeas{W} \right) 
    \dd y \\
    = \int_{y \in \abs\tau^{\fracpart{x}-c} \f{\wh{\theta}}{\floor{x}} U}
    \left( \indicator{W}{\fracpartZtau{y \tau^{j-w}}} - \lmeas{W} \right) \dd y,
  \end{multline*}
  because the integral of the region $\floorV[j-w]{\abs\tau^{\fracpart{x}-c}
    \f{\wh{\theta}}{\floor{x}} U}$ is zero, see proof of
  Lemma~\vref{lem:the-zero-part}.

  Now we deal with the continuity for $x\in\Z$. Let $m \in x + p\Z$, let
  $M=\abs\tau^m$, and consider
  \begin{equation*}
    \f{Z_\eta}{M} - \f{Z_\eta}{M-1}.
  \end{equation*}
  For an appropriate $a\in\R$ we get
  \begin{equation*}
    \f{Z_\eta}{M} 
    = a M^2 \log M + M^2 \f{\psi_\eta}{\log M}
    + \Oh{M^{\alpha} \log M}
    + \Oh{M^{\delta} \log M},
  \end{equation*}
  and thus
  \begin{equation*}
    \f{Z_\eta}{M} = a M^2 m 
    + M^2 \underbrace{\f{\psi_\eta}{m}}_{= \f{\psi_\eta}{x}}
    + \Oh{M^\alpha m}
    + \Oh{M^\delta m}.
  \end{equation*}
  Further we obtain
  \begin{multline*}
    \f{Z_\eta}{M-1} 
    = a \left(M-1\right)^2 \f{\log}{M-1}
    + \left(M-1\right)^2
    \f{\psi_\eta}{\f{\log}{M-1}} \\
    + \Oh{\left(M-1\right)^{\alpha} \f{\log}{M-1}}
    + \Oh{\left(M-1\right)^{\delta} \f{\log}{M-1}},
  \end{multline*}
  and thus, using the abbreviation $L = \f{\log}{1 - M^{-1}}$ and $\delta\geq1$,
  \begin{equation*}
    \f{Z_\eta}{M-1}
    = a M^2 m 
    + M^2 \underbrace{\f{\psi_\eta}{m + L}}_{= \f{\psi_\eta}{x+L}}
    + \Oh{M^\alpha m}
    + \Oh{M^\delta m}.
  \end{equation*}
  Therefore we obtain
  \begin{equation*}
    \frac{\f{Z_\eta}{M} - \f{Z_\eta}{M-1}} {M^2} 
    = \f{\psi_\eta}{x} - \f{\psi_\eta}{x + L}
    + \Oh{M^{\alpha-2} m}
    + \Oh{M^{\delta-2} m}.
  \end{equation*}
  Since $\cardV{M U \setminus \left(M-1\right) U}$ is clearly an upper bound for
  the number of \wNAF{}s with values in $M U \setminus \left(M-1\right) U$ and
  each of these \wNAF{}s has less than $\floor{\log M} + c$ digits, see
  Lemma~\vref{lem:choosing-j}, we obtain
  \begin{equation*}
    \f{Z_\eta}{M} - \f{Z_\eta}{M-1}
    \leq \cardV{M U \setminus \left(M-1\right) U}
    \left( m + c \right).
  \end{equation*}
  Using \itemref{enu:set-nu:difference} of Proposition~\vref{pro:set-nu} yields
  then
  \begin{equation*}
    \f{Z_\eta}{M} - \f{Z_\eta}{M-1} = \Oh{M^\delta m}.
  \end{equation*}
  Therefore we get
  \begin{equation*}
    \f{\psi_\eta}{x} - \f{\psi_\eta}{x + L} 
    = \Oh{M^{\delta-2} m}
    + \Oh{M^{\alpha-2} m}
    + \Oh{M^{\delta-2} m}.
  \end{equation*}
  Taking the limit $m\to\infty$ in steps of $p$, thus $L$ tends to $0$, and
  using $\alpha<2$ and $\delta<2$ yields
  \begin{equation*}
    \f{\psi_\eta}{x} - \lim_{\eps \to 0^-} \f{\psi_\eta}{x+\eps} = 0, 
  \end{equation*}
  i.e., $\psi_\eta$ is continuous for $x\in\Z$.
\end{proof}




\section*{Acknowledgements}

We thank Stephan Wagner for contributing the proof of
Lemma~\vref{le:wagner-lemma}.


\renewcommand{\MR}[1]{}

\bibliographystyle{amsplain}
\bibliography{../../bib/cheub}

\providecommand{\Submitted}{Submitted} \providecommand{\availableat}{ available
  at } \providecommand{\alsoavailableat}{ also available at }
  \providecommand{\evavailableat}{ earlier version available at }
  \providecommand{\toappearin}{To appear in }
  \providecommand{\doi}[1]{\href{http://dx.doi.org/#1}{\path{doi:#1}}}
  \providecommand{\etc}{\emph{etc.}}\def\cprime{$'$}
\providecommand{\bysame}{\leavevmode\hbox to3em{\hrulefill}\thinspace}
\providecommand{\MR}{\relax\ifhmode\unskip\space\fi MR }
\providecommand{\MRhref}[2]{%
  \href{http://www.ams.org/mathscinet-getitem?mr=#1}{#2}
}
\providecommand{\href}[2]{#2}
\begin{thebibliography}{10}

\bibitem{Aurenhammer:1991:voron-diagr}
Franz Aurenhammer, \emph{Voronoi diagrams --- a survey of a fundamental
  geometric data structure}, ACM Comput. Surv. \textbf{23} (1991), no.~3,
  345--405.

\bibitem{Avanzi-Heuberger-Prodinger:2010:arith-of}
Roberto Avanzi, Clemens Heuberger, and Helmut Prodinger, \emph{Arithmetic of
  supersingular {K}oblitz curves in characteristic three}, Tech. Report 2010-8,
  Graz University of Technology, 2010,
  \url{http://www.math.tugraz.at/fosp/pdfs/tugraz_0166.pdf}, also available as
  Cryptology ePrint Archive, Report 2010/436, \url{http://eprint.iacr.org/}.

\bibitem{Barnsley:1988:fractals-ew}
Michael Barnsley, \emph{Fractals everywhere}, Academic Press, Inc, 1988.

\bibitem{Blake-Kumar-Xu:2005:effic-algor}
Ian~F. Blake, V.~Kumar Murty, and Guangwu Xu, \emph{Efficient algorithms for
  {K}oblitz curves over fields of characteristic three}, J. Discrete Algorithms
  \textbf{3} (2005), no.~1, 113--124. \MR{MR2167767 (2006f:11069)}

\bibitem{Blake-Murty-Xu:2005:naf}
\bysame, \emph{A note on window $\tau$-{NAF} algorithm}, Inform. Process. Lett.
  \textbf{95} (2005), 496--502.

\bibitem{Blake-Murty-Xu:ta:nonad-radix}
\bysame, \emph{Nonadjacent radix-$\tau$ expansions of integers in {E}uclidean
  imaginary quadratic number fields}, Canad. J. Math. \textbf{60} (2008),
  no.~6, 1267--1282.

\bibitem{Delange:1975:chiffres}
Hubert Delange, \emph{Sur la fonction sommatoire de la fonction ``somme des
  chiffres''}, Enseignement Math. (2) \textbf{21} (1975), 31--47. \MR{52 \#319}

\bibitem{Edgar:2008:measur}
Gerald~A. Edgar, \emph{Measure, topology, and fractal geometry}, second ed.,
  Undergraduate Texts in Mathematics, Springer-Verlag, New York, 2008.

\bibitem{Gordon:1998}
Daniel~M. Gordon, \emph{A survey of fast exponentiation methods}, J. Algorithms
  \textbf{27} (1998), 129--146. \MR{99g:94014}

\bibitem{Grabner-Heuberger-Prodinger:2004:distr-results-pairs}
Peter~J. Grabner, Clemens Heuberger, and Helmut Prodinger, \emph{Distribution
  results for low-weight binary representations for pairs of integers},
  Theoret. Comput. Sci. \textbf{319} (2004), 307--331.

\bibitem{Graham-Knuth-Patashnik:1994}
Ronald~L. Graham, Donald~E. Knuth, and Oren Patashnik, \emph{Concrete
  mathematics. {A} foundation for computer science}, second ed.,
  Addison-Wesley, 1994. \MR{97d:68003}

\bibitem{Heuberger-Prodinger:2006:analy-alter}
Clemens Heuberger and Helmut Prodinger, \emph{Analysis of alternative digit
  sets for nonadjacent representations}, Monatsh. Math. \textbf{147} (2006),
  219--248.

\bibitem{Hwang:1998}
Hsien-Kuei Hwang, \emph{On convergence rates in the central limit theorems for
  combinatorial structures}, European J. Combin. \textbf{19} (1998), 329--343.
  \MR{99c:60014}

\bibitem{Koblitz:1992:cm}
Neal Koblitz, \emph{C{M}-curves with good cryptographic properties}, Advances
  in cryptology---CRYPTO '91 (Santa Barbara, CA, 1991), Lecture Notes in
  Comput. Sci., vol. 576, Springer, Berlin, 1992, pp.~279--287. \MR{94e:11134}

\bibitem{Koblitz:1998:ellip-curve}
\bysame, \emph{An elliptic curve implementation of the finite field digital
  signature algorithm}, Advances in cryptology---CRYPTO '98 (Santa Barbara, CA,
  1998), Lecture Notes in Comput. Sci., vol. 1462, Springer, Berlin, 1998,
  pp.~327--337. \MR{MR1670960 (99j:94052)}

\bibitem{Matula:1982:basic}
David~W. Matula, \emph{Basic digit sets for radix representation}, J. Assoc.
  Comput. Mach. \textbf{29} (1982), no.~4, 1131--1143. \MR{MR674260
  (83k:68017)}

\bibitem{Muir-Stinson:2004:alter-digit}
James~A. Muir and Douglas~R. Stinson, \emph{Alternative digit sets for
  nonadjacent representations}, Selected areas in cryptography, Lecture Notes
  in Comput. Sci., vol. 3006, Springer, Berlin, 2004, pp.~306--319.
  \MR{MR2094738 (2005g:94083)}

\bibitem{Reitwiesner:1960}
George~W. Reitwiesner, \emph{Binary arithmetic}, Advances in computers, vol.~1,
  Academic Press, New York, 1960, pp.~231--308.

\bibitem{Solinas:1997:improved-algorithm}
Jerome~A. Solinas, \emph{An improved algorithm for arithmetic on a family of
  elliptic curves}, Advances in Cryptology --- {CRYPTO} '97. 17th annual
  international cryptology conference. Santa Barbara, {CA}, {USA}. August
  17--21, 1997. Proceedings (B.~S. Kaliski, jun., ed.), Lecture Notes in
  Comput. Sci., vol. 1294, Springer, Berlin, 1997, pp.~357--371.

\bibitem{Solinas:2000:effic-koblit}
\bysame, \emph{Efficient arithmetic on {K}oblitz curves}, Des. Codes Cryptogr.
  \textbf{19} (2000), 195--249. \MR{2002k:14039}

\end{thebibliography}


\end{document}

